\documentclass[12pt,a4paper]{amsart}
\usepackage{amssymb,eucal}
\usepackage[all,cmtip]{xy}

\calclayout

\newcommand{\+}{\protect\nobreakdash-}
\renewcommand{\:}{\colon}

\newcommand{\rarrow}{\longrightarrow}

\newcommand{\ot}{\otimes}

\DeclareFontFamily{U}{mathx}{\hyphenchar\font45}
\DeclareFontShape{U}{mathx}{m}{n}{
      <5> <6> <7> <8> <9> <10>
      <10.95> <12> <14.4> <17.28> <20.74> <24.88>
      mathx10
      }{}
\DeclareSymbolFont{mathx}{U}{mathx}{m}{n}
\DeclareFontSubstitution{U}{mathx}{m}{n}
\DeclareMathAccent{\widecheck}{0}{mathx}{"71}

\DeclareMathOperator{\Mor}{Mor}
\DeclareMathOperator{\Hom}{Hom}

\DeclareMathOperator{\Cohom}{Cohom}

\newcommand{\id}{\mathrm{id}}
\newcommand{\rop}{{\mathrm{op}}}
\newcommand{\sop}{{\mathsf{op}}}

\newcommand{\Modl}{{\operatorname{\mathsf{--Mod}}}}
\newcommand{\Modr}{{\operatorname{\mathsf{Mod--}}}}
\newcommand{\Bimod}{{\operatorname{\mathsf{--Bimod--}}}}
\newcommand{\Contra}{{\operatorname{\mathsf{--Contra}}}}
\newcommand{\Contrar}{{\operatorname{\mathsf{Contra--}}}}
\newcommand{\Comodl}{{\operatorname{\mathsf{--Comod}}}}
\newcommand{\Comodr}{{\operatorname{\mathsf{Comod--}}}}
\newcommand{\Simodl}{{\operatorname{\mathsf{--Simod}}}}
\newcommand{\Simodr}{{\operatorname{\mathsf{Simod--}}}}
\newcommand{\Sicntr}{{\operatorname{\mathsf{--Sicntr}}}}
\newcommand{\SicntrR}{{\operatorname{\mathsf{Sicntr--}}}}
\newcommand{\Bicomod}{{\operatorname{\mathsf{--Bicomod--}}}}
\newcommand{\Vect}{{\operatorname{\mathsf{--Vect}}}}

\newcommand{\nBimodn}{{\operatorname{\mathsf{--{}^n Bimod^n--}}}}
\newcommand{\Modln}{{\operatorname{\mathsf{--{}^n Mod}}}}
\newcommand{\Modrn}{{\operatorname{\mathsf{Mod^n--}}}}
\newcommand{\tBimodt}{{\operatorname{\mathsf{--{}^t Bimod^t--}}}}
\newcommand{\Modlt}{{\operatorname{\mathsf{--{}^t Mod}}}}
\newcommand{\Modrt}{{\operatorname{\mathsf{Mod^t--}}}}
\newcommand{\Modlc}{{\operatorname{\mathsf{--{}^c Mod}}}}
\newcommand{\Modlz}{{\operatorname{\mathsf{--{}^0 Mod}}}}

\newcommand{\Modls}{{\operatorname{\mathsf{--{}^s Mod}}}}
\newcommand{\Modrs}{{\operatorname{\mathsf{Mod^s--}}}}

\newcommand{\flattBimodtint}{{\operatorname{\mathsf{
   --{}^{\mskip6.1\thinmuskip t}_{flat} Bimod^t_{\C-int}--}}}}
\newcommand{\flattBimodt}{{\operatorname{\mathsf{
   --{}^{\mskip6.1\thinmuskip t}_{flat} Bimod^t--}}}}
\newcommand{\projtBimodtint}{{\operatorname{\mathsf{
   --{}^{\mskip7.4\thinmuskip t}_{proj} Bimod^t_{\C-int}--}}}}
\newcommand{\projtBimodt}{{\operatorname{\mathsf{
   --{}^{\mskip7.4\thinmuskip t}_{proj} Bimod^t--}}}}
\newcommand{\injBicomod}{{\operatorname{\mathsf{--{}_{inj}Bicomod--}}}}

\newcommand{\Sets}{\mathsf{Sets}}
\newcommand{\Ab}{\mathsf{Ab}}
\newcommand{\Br}{\mathsf{Br}}
\newcommand{\TL}{\mathsf{TL}}

\newcommand{\oc}{\mathbin{\text{\smaller$\square$}}}

\newcommand{\bul}{*=0{\bullet}}

\newcommand{\lrarrow}{\mskip.5\thinmuskip\relbar\joinrel\relbar
   \joinrel\rightarrow\mskip.5\thinmuskip\relax}

\newcommand{\LL}{\mathcal L}
\newcommand{\M}{\mathcal M}
\newcommand{\N}{\mathcal N}
\newcommand{\B}{\mathcal B}
\newcommand{\C}{\mathcal C}
\newcommand{\D}{\mathcal D}
\newcommand{\E}{\mathcal E}
\newcommand{\J}{\mathcal J}
\newcommand{\R}{\mathcal R}

\newcommand{\fP}{\mathfrak P}
\newcommand{\fQ}{\mathfrak Q}

\newcommand{\bS}{\boldsymbol{\mathcal S}}
\newcommand{\bL}{\boldsymbol{\mathcal L}}
\newcommand{\bM}{\boldsymbol{\mathcal M}}
\newcommand{\bN}{\boldsymbol{\mathcal N}}

\newcommand{\bP}{\boldsymbol{\mathfrak P}}
\newcommand{\bQ}{\boldsymbol{\mathfrak Q}}

\newcommand{\bm}{\mathbf m}
\newcommand{\be}{\mathbf e}
\newcommand{\bn}{\mathbf n}
\newcommand{\bp}{\mathbf p}

\newcommand{\fg}{\mathfrak g}
\newcommand{\fh}{\mathfrak h}
\newcommand{\fa}{\mathfrak a}

\newcommand{\sC}{\mathsf C}
\newcommand{\sE}{\mathsf E}
\newcommand{\sF}{\mathsf F}
\newcommand{\sG}{\mathsf G}
\newcommand{\sH}{\mathsf H}
\newcommand{\sM}{\mathsf M}
\newcommand{\sN}{\mathsf N}
\newcommand{\sP}{\mathsf P}
\newcommand{\sB}{\mathsf B}

\newcommand{\boZ}{\mathbb Z}
\newcommand{\boQ}{\mathbb Q}
\newcommand{\boR}{\mathbb R}

\newcommand{\Section}[1]{\bigskip\section{#1}\medskip}
\theoremstyle{plain}
\newtheorem{thm}{Theorem}[section]
\newtheorem{prop}[thm]{Proposition}
\newtheorem{lem}[thm]{Lemma}
\newtheorem{cor}[thm]{Corollary}
\theoremstyle{definition}
\newtheorem{rem}[thm]{Remark}
\newtheorem{ex}[thm]{Example}

\begin{document}

\title{Semialgebras associated with nonunital algebras \\
and $k$-linear subcategories}

\author{Leonid Positselski}

\address{Institute of Mathematics, Czech Academy of Sciences \\
\v Zitn\'a~25, 115~67 Prague~1 \\ Czech Republic} 

\email{positselski@math.cas.cz}

\begin{abstract}
 This paper is a sequel to~\cite{Plfin,Pnun}.
 A construction associating a semialgebra with an algebra, subalgebra,
and a coalgebra dual to the subalgebra played a central role in
the author's book~\cite{Psemi}.
 In this paper, we extend this construction to certain nonunital
algebras.
 The resulting semialgebra is still semiunital over a counital
coalgebra.
 In particular, we associate semialgebras to locally finite
subcategories in $k$\+linear categories.
 Examples include the Temperley--Lieb and Brauer diagram categories
and the Reedy category of simplices in a simplicial set.
\end{abstract}

\maketitle

\tableofcontents

\section*{Introduction}
\medskip

\setcounter{subsection}{-1}
\subsection{{}}
 \emph{Semialgebras} can be informally defined as ``algebras over
coalgebras''.
 Given a coassociative, counital coalgebra $\C$ over a field~$k$,
the category of $\C$\+$\C$\+bicomodules $\C\Bicomod\C$ is a monoidal
category with respect to the functor~$\oc_\C$ of cotensor product
over~$\C$.
 A (semiassociative, semiunital) \emph{semialgebra} over $\C$ is
a monoid object in the monoidal category $\C\Bicomod\C$.

 Semialgebras appeared under the name of ``internal categories'' in
Aguiar's dissertation~\cite{Agu}.
 The terminology ``semialgebras'' was introduced by the present author
in the monograph on semi-infinite homological algebra~\cite{Psemi},
where semialgebras were the main object of study.
 The key discovery in~\cite{Psemi} was that the semi-infinite homology
and cohomology of associative algebraic structures are properly
assigned to module objects (semimodules and semicontramodules) over
semialgebras.
 In subsequent literature, the concept and the terminology 
``semialgebra'' was used in the paper of Holstein and Lazarev 
on categorical Koszul duality~\cite{HL}.

 The generality level in~\cite{Psemi} included semialgebras over
corings over noncommutative, nonsemisimple rings.
 Natural examples of semialgebras over, say, coalgebras over
commutative rings certainly arise in many contexts; so the complicated
generality of three-story towers (ring, coalgebra/coring, semialgebra)
is quite useful.
 But it made the exposition in~\cite{Psemi} somewhat heavy and full
of technical details.
 In the present paper, we restrict ourselves to the less complicated
setting of coalgebras over fields and semialgebras over such coalgebras.

\subsection{{}} \label{introd-semialgebra-construction}
 All the examples of semialgebras discussed in~\cite{Psemi} were
produced by one and the same construction, spelled out
in~\cite[Chapter~10]{Psemi}.
 Restricted to our present context of coalgebras over a field~$k$, this
construction works as follows.
 Let $K$ and $R$ be associative algebras over~$k$ and $f\:K\rarrow R$
be a $k$\+algebra homomorphism.
 Suppose that we are given a coalgebra $\C$ over~$k$ which is
``dual to~$K$'' in a weak sense.
 Specifically, this means that a multiplicative, unital bilinear
pairing $\phi\:\C\times_k K\rarrow k$ is given as a part of the data.
 Assume also that $f$~makes $R$ a flat left $K$\+module.
 Then, under a certain further assumption (called ``integrability''
in this paper), the tensor product $\bS=\C\ot_KR$ becomes a semialgebra
over~$\C$.

 It was also explained in~\cite[Chapter~10]{Psemi} that some abelian
module categories over the semialgebra $\C\ot_KR$ can be described
in terms of $R$\+modules and $\C$\+comodules or $\C$\+contramodules.
 Specifically, the module categories admitting such a description in
this generality are the categories of \emph{right\/ $\bS$\+semimodules}
and \emph{left\/ $\bS$\+semi\-con\-tra\-mod\-ules}.
 Roughly speaking, the answer is that the datum of a right
$\bS$\+semimodule is equivalent to that of a right $R$\+module whose
underlying right $K$\+module structure can be integrated to a right
$\C$\+comodule structure.
 The datum of a left $\bS$\+semicontramodule is equivalent to that of
a left $R$\+module whose underlying left $K$\+module structure can be
(or has been) integrated to a left $\C$\+contramodule structure.
 The latter assertion assumes that $R$ is a projective left $K$\+module.

\subsection{{}}
 The terminology ``integrable'' or ``integrated'' in the preceding
paragraphs comes from Lie theory.
 In fact, finite- and infinite-dimensional Lie theory provided
the thematic examples of semialgebras which motivated the exposition
in~\cite{Psemi}.

 For example, if $H$ is an algebraic group over a field~$k$ of
characteristic~$0$, then the universal enveloping algebra
$K=U(\fh)$ of the Lie algebra of~$H$ can be described as the algebra
of distributions on $H$ supported at the unit element $e\in H$, with
respect to the convolution multiplication.
 On the other hand, the convolution comultiplication defines
a coalgebra (in fact, a Hopf algebra) structure on the algebra
$\C=\C(H)$ of regular functions on~$H$.
 The algebra $K$ and coalgebra $\C$ come together with a natural
multiplicative, unital bilinear pairing $\phi\:\C\times_k K\rarrow k$
as above~\cite[Section~2.7]{Prev}.

 Furthermore, suppose that $H$ is a closed subgroup in an algebraic
group~$G$.
 Assume for the sake of a technical simplification that the algebraic
group $H$ is \emph{unimodular}, i.~e., a nonzero biinvariant
differential top form exists on $H$ (and has been chosen).
 Then the vector space $\bS=\bS(G,H)$ of distributions on $G$,
supported in $H$ and regular along $H$, becomes a semialgebra
over the coalgebra $\C=\C(H)$ \,\cite[Section~C.4]{Psemi}.
 The semialgebra $\bS$ can be also produced as the tensor product
$\bS=\C\ot_{U(\fh)}U(\fg)$, where $R=U(\fg)$ is the enveloping
algebra of the Lie algebra of~$G$.

 Notice that the construction of the semialgebra of distributions
$\bS(G,H)$ is left-right symmetric, but the construction of
the tensor product semialgebra $\C\ot_KR$ is \emph{not}.
 Actually, in our Lie theory context we could consider
two semialgebras $\bS^l=\bS^r(\fg,H)=\C\ot_{U(\fh)}U(\fg)$ and
$\bS^r=\bS^l(\fg,H)=U(\fg)\ot_{U(\fh)}\C$, which are naturally
left-right opposite to each other.
 Under the unimodularity assumption above, we have isomorphisms of
semialgebras $\bS^r(\fg,H)\simeq\bS(G,H)\simeq\bS^l(\fg,H)$
\,\cite[Remark~C.4.4]{Psemi}.
 When $H$ is not unimodular, the two semialgebras $\bS^r$ and $\bS^l$
are \emph{no longer} isomorphic, but only Morita
equivalent~\cite[Sections~C.4.3\+-4]{Psemi}.

\subsection{{}}
 What does one achieve by having two equivalent constructions (up to
isomorphism or Morita equivalence) of the same semialgebra $\bS$,
viz., $\bS^r=\C\ot_KR$ and $\bS^l=R\ot_K\C$\,?
 Actually, one achieves a lot.
 The point is that the semialgebra $\bS^r$ comes together with
a natural description of \emph{right} $\bS$\+semimodules and
left $\bS$\+semi\-con\-tra\-mod\-ules (as explained
in Section~\ref{introd-semialgebra-construction}).
 Similarly, the semialgebra $\bS^l$ comes together with a natural
description of \emph{left} $\bS$\+semimodules (and right
$\bS$\+semicontramodules).

 If one wants to know simultaneously what the left and the right
$\bS$\+semimodules are (e.~g., for the purposes of semi-infinite
homology~\cite[Chapter~2]{Psemi}), or if one wants to know
simultaneously what the left $\bS$\+semimodules and the left
$\bS$\+semicontramodules are (e.~g., for the purposes of
the semimodule-semicontramodule correspondence~\cite[Chapter~6]{Psemi},
\cite[Sections~8 and~10.3]{PS1}), then one needs to have both
a ``left'' and a ``right'' description of one's semialgebra~$\bS$.
 The two descriptions can possibly arise from two different algebras
$K$, of course; or in any event, from two different algebras~$R$.
 One is lucky if one obtains isomorphisms (or Morita equivalences)
of semialgebras like $\C\ot_KR^r\simeq\bS\simeq R^l\ot_K\C$, where
$R^r$ and $R^l$ are two different algebras sharing a common
subalgebra~$K$.
 Then right $\bS$\+semimodules are described as right $R^r$\+modules
with a right $\C$\+comodule structure and left $\bS$\+semicontramodules
as left $R^r$\+modules with a left $\C$\+contramodule structure,
while left $\bS$\+semimodules are described as left $R^l$\+modules
with a left $\C$\+comodule structure.

\subsection{{}}
 A prime example of the latter situation occurs in infinite-dimensional
Lie theory.
 Let $\fg$~be a Tate (locally linearly compact) Lie algebra over
a field~$k$, such that, e.~g., the Lie algebra of vector fields on
the punctured formal circle $\fg=k((t))d/dt$, or the Lie algebra of
currents on the punctured formal circle $\fg=\fa((t))$, where $\fa$~is
a finite-dimensional Lie algebra.
 Let $\fh\subset\fg$ be a pro-nilpotent compact open Lie subalgebra
in~$\fg$, such as $\fh=t^nk[[t]]t/dt$ with some $n\ge1$ in the former
(Virasoro) case or $\fh=t^n\fa[[t]]$ with some $n\ge1$ in the latter
(Kac--Moody) case.
 Assume that the action of~$\fh$ in the quotient space $\fg/\fh$ is
ind-nilpotent.

 One considers a certain kind of central extensions~$\varkappa$ of
the pairs~$(\fg,\fh)$.
 In particular, there is the \emph{canonical central
extension}~$\varkappa_0$ (in the case of the Virasoro, this means
the central charge $\varkappa_0=-26$).
 Let $R^r=U_{\varkappa+\varkappa_0}(\fg)$ and $R^l=U_\varkappa(\fg)$ be
the corresponding two enveloping algebras describing respresentations
of the respective central extensions of~$\fg$ with the central charge
fixed as a given (generally nonzero) scalar from~$k$.
 Put $K=U(\fh)$, and let $\C$ be the coenveloping coalgebra of
the conilpotent Lie coalgebra dual to~$\fh$
\,\cite[Section~D.6.1]{Psemi}.
 Then one has an isomorphism $\C\ot_KR^r\simeq R^l\ot_K\C$ of
semialgebras over~$\C$ \,\cite[Theorem~D.3.1]{Psemi},
\cite[Sections~2.7\+-2.8]{Prev}.
 The classical semi-infinite homology and cohomology of Tate Lie
algebras arise in this context~\cite[Section~D.6.2]{Psemi}.

\subsection{{}}
 The aim of the present paper is to extend the construction of
the semialgebra $\C\ot_KR$ given in~\cite[Chapter~10]{Psemi} from
\emph{algebras} to \emph{categories}.
 To any small $k$\+linear category $\sE$ one assigns a $k$\+algebra
$R=R_\sE$, but when the set of objects of $\sE$ is infinite,
the algebra $R$ \emph{has no unit}.
 The algebra $R_\sE$ only has ``local units''.
 The rings/algebras $K$ and $R$ were, of course, assumed to be unital
in~\cite[Chapter~10]{Psemi}.
 In this paper we work out the case of \emph{nonunital} algebras $K$
and~$R$.
 The coalgebra $\C$ is still presumed to be counital, and
the resulting semialgebra $\bS$ which we construct is semiunital.
 
 Intuitively, one can explain the situation with local and global
units as follows.
 A unit in a ring is an element, and as such, it is a global datum:
one cannot really glue it from local pieces.
 If a unital algebra $R$ is a union of its vector subspaces, then
the unit has be an element of one of such subspaces.
 But the counit of a coalgebra $\C$ is a linear function $\epsilon\:\C
\rarrow k$, and as such, it can be obtained by gluing a compatible
family of linear functions defined on vector subspaces of~$\C$.
 Likewise, the semiunit of a semialgebra $\bS$ is
a $\C$\+$\C$\+bicomodule map $\be\:\C\rarrow\bS$, and it can be also
possibly obtained by gluing a compatible family of maps defined on
vector subspaces or subbicomodules of~$\C$.
 This vague and perhaps not necessarily very convincing wording
expresses the intuition behind our setting of nonunital algebras,
but counital coalgebras and semiunital semialgebras in this paper.

 Striving to approach the natural generality, we do not restrict
ourselves in this paper to algebras with enough idempotents or
``locally unital algebras'' (which correspond precisely to small
$k$\+linear categories).
 Rather, our main results are stated for more general classes of what
we call \emph{left flat t\+unital} or \emph{left projective
t\+unital} algebras, and much of the exposition is written in
an even wider generality.

\subsection{{}}
 The modest aims of the exposition in this paper do not include any
attempt to construct isomorphisms of Morita equivalences of ``left''
and ``right'' semialgebras.
 Presumably, such constructions always arise in more or less specific
contexts, like the Lie theory context discussed above.
 The principal examples which motivated this paper are different:
these are the Temperley--Lieb and Brauer diagram categories and
the $k$\+linearization of the Reedy category of simplices in
a simplicial set, with the natural triangular decompositions which
these categories possess.
 Our understanding of these examples is not sufficient to attempt any
constructions of isomorphisms of ``left'' and ``right'' semialgebras.

 In fact, the ``right'' semialgebra $\bS=\C\ot_KR$ always has
a technically important property of being an injective \emph{left}
$\C$\+comodule (essentially, because $R$ is presumed to be
a flat left $K$\+module).
 This property guarantees that the categories of right
$\bS$\+semimodules and left $\bS$\+semicontramodules are abelian.
 One may (and should) wonder whether $\bS$ turns out to be
an injective \emph{right} $\C$\+comodule in the specific examples
one is interested in.
 But for the examples discussed in this paper (viz.,
the Temperley--Lieb, Brauer, and simplicial set categories),
we do \emph{not} know that.
 
\subsection{{}}
 This paper is a third installment in a series.
 It is largely based on two previous preprints~\cite{Plfin}
and~\cite{Pnun}.
 Coalgebras associated to (what we call) locally finite
$k$\+linear categories, together with comodules and contramodules
over such coalgebras, were studied in~\cite{Plfin}.
 Nonunital rings were discussed in~\cite{Pnun}, with an emphasis
on the tensor/Hom formalism for nonunital rings and other preparatory
results for the present paper.
 Much of the content of~\cite{Pnun} is covered by the much earlier
(1996) unpublished manuscript of Quillen~\cite{Quil}, which has
somewhat different aims and emphasis.
 We refer either to~\cite{Quil} or to~\cite{Pnun}, or to both
whenever appropriate.

\subsection*{Acknowledgement}
 This paper owes its existence to Catharina Stroppel, who kindly
invited me to give a talk at her seminar in Bonn in July~2023,
suggested the Temperley--Lieb (as well as Brauer) diagram category
example, and explained me the triangular decomposition of
the Temperley--Lieb category.
 I~wish to thank Jan \v St\!'ov\'\i\v cek for numerous helpful
discussions, and in particular for suggesting the Reedy
category/simplicial set example.
 The author is supported by the GA\v CR project 23-05148S and
the Czech Academy of Sciences (RVO~67985840).

\Section{Preliminaries on Nonunital Algebras}
\label{prelim-nonunital-secn}

 Throughout this paper, $k$~denotes a fixed ground field.
 All \emph{rings}, \emph{algebras} and \emph{modules} are associative,
usually noncommutative, but \emph{nonunital} by default.
 Still all \emph{coalgebras}, \emph{comodules}, and \emph{contramodules}
are coassociative/contraassociative and counital/contraunital, while
all \emph{semialgebras}, \emph{semimodules}, and
\emph{semicontramodules} are semiassociative/semicontraassociative
and semiunital/semicontraunital.

 We denote by by $\Ab$ the category of abelian groups and by $k\Vect$
the category of $k$\+vector spaces.
 The notation $A\Modl$ and $\Modr A$ stands for the abelian categories
of \emph{unital} left $A$\+modules and right $A$\+modules,
respectively; so it presumes $A$ to be a unital ring.
 In all \emph{bimodules} (unital or nonunital), the left and right
actions of the field~$k$ are presumed to agree.
 So the notation $A\Bimod B$ presumes $A$ and $B$ to be unital
$k$\+algebras and stands for the category
$A\Bimod B=(A\ot_k B^\rop)\Modl$.

 This section presents a discussion of nonunital $k$\+algebras, mostly
following the manuscript~\cite{Quil} and the preprint~\cite{Pnun}.
 The concepts related to coalgebras and semialgebras will be defined
in the next section.

\subsection{Nonunital modules over $k$-algebras}
 The preprint~\cite{Pnun} is written in the generality of abstract
rings.
 There is no ground field or ground ring in~\cite{Pnun}.
 The exposition in the (much earlier) manusript~\cite{Quil} is more
general, and makes room for a ground field or ground ring.

 One of the aims of the whole Section~\ref{prelim-nonunital-secn} is
to discuss, largely following~\cite{Quil}, the comparison between
abstract nonunital rings and nonunital $k$\+algebras.
 The point is that, when one views a nonunital $k$\+algebra $K$ as
an abstract nonunital ring, a $K$\+module \emph{need not} be
a $k$\+vector space in general: for example, the group/ring of
integers $\boZ$ is a $K$\+module with the zero action of~$K$.
 But all the $K$\+modules we are interested in \emph{turn out} to be
$k$\+vector spaces, as explained in the next
Section~\ref{prelim-t-unital-c-unital-subsecn}.

 Given a nonunital ring $K$, the \emph{unitalization} of $K$ is
the unital ring $\widetilde K=\boZ\oplus K$.
 So the ring of integers $\boZ$ is a subring in $\widetilde K$ and
$K$ is a two-sided ideal in~$\widetilde K$.

 Given a nonunital $k$\+algebra $K$, the \emph{$k$\+unitalization} of
$K$ is the unital $k$\+algebra $\widecheck K=k\oplus K$.
 So the $k$\+algebra~$k$ is a subalgebra in $\widecheck K$ and
$K$ is a two-sided ideal in~$\widecheck K$.
 There is a natural unital ring homomorphism $\widetilde K\rarrow
\widecheck K$ acting by the identity on~$K$.

 Both the unital $\widetilde K$\+modules and the unital
$\widecheck K$\+modules can (and should) be simply thought of
as nonunital $K$\+modules.
 The difference is that the unital $\widecheck K$\+modules are
$k$\+vector spaces (with the actions of~$k$ and $K$ compatible in
the obvious sense), while the unital $\widetilde K$\+modules
are just abelian groups~\cite[\S1]{Quil}.

\subsection{t-Unital and c-unital modules}
\label{prelim-t-unital-c-unital-subsecn}
 Let $K$ be a nonunital ring.
 The notation $KM\subset M$ and $NK\subset N$ for a left $K$\+module
$M$ and a right $K$\+module $N$ stands for the subgroups/submodules
spanned by the products, as usual.
 The notation ${}_KM$ means the $K$\+submodule consisting of all
elements $m\in M$ such that $Km=0$ in~$M$.
 The following lemma is due to Quillen~\cite[proof of
Proposition~9.2]{Quil}.

\begin{lem}[\cite{Quil}] \label{hom-tensor-over-tilde-check-agree}
 Let $K$ be a nonunital $k$\+algebra, $M$ and $P$ be unital left
$\widecheck K$\+modules, and $N$ be a unital right
$\widecheck K$\+module.
 Then \par
\textup{(a)} if either $NK=N$ or $KM=M$, then the natural surjective
map of tensor products
$$
 N\ot_{\widetilde K}M\lrarrow N\ot_{\widecheck K}M
$$
is an isomorphism; \par
\textup{(b)} if either $KM=M$ or ${}_KP=0$, then the natural injective
map of\/ $\Hom$ groups/vector spaces
$$
 \Hom_{\widecheck K}(M,P)\lrarrow\Hom_{\widetilde K}(M,P)
$$
is an isomorphism.
\end{lem}

\begin{proof}
 Part~(a): consider the case when $KM=M$; so for every $m\in M$
there exist elements $r_1$,~\dots, $r_j\in K$ and $m_1$,~\dots,
$m_j\in M$ such that $m=\sum_{i=1}^jr_im_i$ in~$M$.
 Then, for all $n\in N$ and $\check r\in\widecheck K$ (e.~g,
$\check r\in k$), one has $n\check r\ot_{\widetilde K}m=
\sum_{i=1}^j n\check r\ot_{\widetilde K}r_im_i=
\sum_{i=1}^jn\check rr_i\ot_{\widetilde K}m_i=
\sum_{i=1}^jn\ot_{\widetilde K}\check rr_im_i=
n\ot_{\widetilde K}\check rm$ (since
$\check rr_i\in K\subset\widetilde K$), as desired.

 Part~(b): let $f\:M\rarrow P$ be a $\widetilde K$\+linear map.
 Assuming $M=KM$, for any $\check r\in\widecheck K$ and $m\in M$
we have $m=\sum_{i=1}^jr_im_i$ as above and $f(\check rm)=
\sum_{i=1}^jf(\check rr_im_i)=\sum_{i=1}^j\check rr_if(m_i)=
\sum_{i=1}^j\check rf(r_im_i)=\check rf(m)$.
 Assuming ${}_KP=0$, for any $r\in K$ we compute
$rf(\check rm)=f(r\check rm)=r\check rf(m)$, hence
$K(f(\check rm)-\check rf(m))=0$ and it follows that
$f(\check rm)-\check rf(m)=0$ in~$P$.
\end{proof}

 Let $K$ be a nonunital ring.
 Following the terminology in~\cite[Section~1]{Pnun}, we say that
a unital left $\widetilde K$\+module $M$ is \emph{t\+unital} if
the natural $\widetilde K$\+module map $K\ot_{\widetilde K}M\rarrow M$
is an isomorphism.
 Similarly, a unital right $\widetilde K$\+module $N$ is
\emph{t\+unital} if the natural $\widetilde K$\+module map
$N\ot_{\widetilde K}K\rarrow N$ is an isomorphism.

 Dual-analogously, a unital left $\widetilde K$\+module $P$ is said
to be \emph{c\+unital}~\cite[Section~3]{Pnun} if the natural
$\widetilde K$\+module map $P\rarrow\Hom_{\widetilde K}(K,P)$ is
an isomorphism.

 Clearly, for any t\+unital left $\widetilde K$\+module $M$ one has
$KM=M$, and for any t\+unital right $\widetilde K$\+module $N$ one has
$NK=N$.
 Similarly, for any c\+unital left $\widetilde K$\+module $P$ one has
${}_KP=0$.

\begin{lem}[{\cite[Proposition~9.2]{Quil}}]
\label{module-structure-extended-lemma}
 Let $K$ be a nonunital $k$\+algebra.  Then \par
\textup{(a)} for any t\+unital $\widetilde K$\+module $M$,
the $\widetilde K$\+module structure on $M$ arises from a unique
$\widecheck K$\+module structure; \par
\textup{(b)} for any c\+unital $\widetilde K$\+module $P$,
the $\widetilde K$\+module structure on $P$ arises from a unique
$\widecheck K$\+module structure.
\end{lem}

\begin{proof}
 Part~(a): the isomorphism $K\ot_{\widetilde K}M\simeq M$ endows $M$
with the left $\widecheck K$\+module structure induced by the left
$\widecheck K$\+module structure on~$K$.
 This proves the existence.
 The uniqueness holds for any $\widetilde K$\+module $M$ such that
$KM=M$ (apply the first case of
Lemma~\ref{hom-tensor-over-tilde-check-agree}(b) to $M$ and $P$
denoting two different $\widecheck K$\+module structures on one and
the same $\widetilde K$\+module).
 Part~(b): the isomorphism $P\simeq\Hom_{\widetilde K}(K,P)$ endows $P$
with the left $\widecheck K$\+module structure induced by the right
$\widecheck K$\+module structure on~$K$.
 This proves the existence.
 The uniqueness holds for any $\widetilde K$\+module $P$ such that
${}_KP=0$ (apply the second case of
Lemma~\ref{hom-tensor-over-tilde-check-agree}(b) to $M$ and $P$
denoting two different $\widecheck K$\+module structures on one and
the same $\widetilde K$\+module).
\end{proof}

\begin{lem}[{\cite[Proposition~9.2]{Quil}}]
\label{t-c-unitality-equivalence-lemma}
 Let $K$ be a nonunital $k$\+algebra.  Then \par
\textup{(a)} a unital left $\widecheck K$\+module $M$ is t\+unital
as a $\widetilde K$\+module if and only if the natural
$\widecheck K$\+module map $K\ot_{\widecheck K}M\rarrow M$ is
an isomorphism; \par
\textup{(b)} a unital left $\widecheck K$\+module $P$ is c\+unital
as a $\widetilde K$\+module if and only if the natural
$\widecheck K$\+module map $P\rarrow\Hom_{\widecheck K}(K,P)$ is
an isomorphism.
\end{lem}

\begin{proof}
 Part~(a): any one of the two maps $K\ot_{\widetilde K}M\rarrow M$
or $K\ot_{\widecheck K}M\rarrow M$ being an isomorphism implies $KM=M$.
 Then it remains to apply the second case of
Lemma~\ref{hom-tensor-over-tilde-check-agree}(a) (for $N=K$).
 Part~(b): any one of the two maps $P\rarrow\Hom_{\widetilde K}(K,P)$
or $P\rarrow\Hom_{\widecheck K}(K,P)$ being an isomorphism implies
${}_KP=0$.
 Then it remains to apply the second case of
Lemma~\ref{hom-tensor-over-tilde-check-agree}(b) (for $M=K$).
\end{proof}

 Let $K$ be a nonunital $k$\+algebra.
 In view of Lemmas~\ref{module-structure-extended-lemma}(a)
and~\ref{t-c-unitality-equivalence-lemma}(a), one can equivalently
speak of ``t\+unital $\widetilde K$\+modules'' or
``t\+unital $\widecheck K$\+modules''.
 Hence we will simply call them \emph{t\+unital $K$\+modules}.
 The category of t\+unital left $K$\+modules will be denoted by
$K\Modlt$, and the category of t\+unital right $K$\+modules by
$\Modrt K$.

 Similarly, in view of Lemmas~\ref{module-structure-extended-lemma}(b)
and~\ref{t-c-unitality-equivalence-lemma}(b), one can equivalently
speak of ``c\+unital $\widetilde K$\+modules'' or
``c\+unital $\widecheck K$\+modules''.
 Hence we will simply call them \emph{c\+unital $K$\+modules}.
 The category of c\+unital left $K$\+modules will be denoted by
$K\Modlc$.

\begin{lem} \label{duality-preserves-reflects-t-c}
 Let $K$ be a nonunital $k$\+algebra and $N$ be a unital right
$\widecheck K$\+module.
 For any $k$\+vector space $V$, we endow the vector space\/
$\Hom_k(N,V)$ with the induced structure of unital left
$\widecheck K$\+module.
 Then the following conditions are equivalent:
\begin{enumerate}
\item $N$ is a t\+unital right $K$\+module;
\item $\Hom_k(N,k)$ is a c\+unital left $K$\+module;
\item $\Hom_k(N,V)$ is a c\+unital left $K$\+module for
every $k$\+vector space~$V$.
\end{enumerate}
\end{lem}

\begin{proof}
 (1)\,$\Longleftrightarrow$\,(2) and
 (1)\,$\Longrightarrow$\,(3) Similar to~\cite[Lemma~8.17]{Pnun}.

 (3)\,$\Longrightarrow$\,(2) Obvious.
 
 (2)\,$\Longrightarrow$\,(3) By~\cite[Proposition~5.5]{Quil}
or~\cite[Lemma~3.5]{Pnun}, the full subcategory $K\Modlc$ is closed
under kernels and direct products in $K\Modln$.
 Hence it is closed also under direct summands.
 The left $K$\+module $\Hom_k(N,V)$ is a direct summand of a product
of copies of the left $K$\+module $\Hom_k(N,k)$ (since any vector
space $V$ is a direct summand of a product of one-dimensional vector
spaces).
\end{proof}

 A ring or $k$\+algebra $K$ is called \emph{t\+unital} if it is
t\+unital as a left $K$\+module, or equivalently, as a right
$K$\+module~\cite[Section~1]{Pnun}.
 Clearly, any t\+unital ring is \emph{idempotent}, i.~e., $K^2=K$.

\subsection{Null-modules and the abelian category equivalence}
\label{null-modules-prelim-subsecn}
 We will say that a unital left $\widetilde K$\+module $N$ is
a \emph{null-module} if $KN=0$.
 The full subcategory of null-modules will be denoted by
$\widetilde K\Modlz\subset\widetilde K\Modl$.
 The full subcategory of (similarly defined) null-modules over
$\widecheck K$ is denoted by $\widecheck K\Modlz\subset
\widecheck K\Modl$.

 For an idempotent ring/$k$\+algebra $K$, the full subcategory
$\widetilde K\Modlz$ is closed under submodules, quotients, direct
sums, direct products, and extensions in $\widetilde K\Modl$
\,\cite[Lemma~5.1]{Pnun}; so $\widetilde K\Modlz\subset
\widetilde K\Modl$ and $\widecheck K\Modlz\subset\widecheck K\Modl$
are Serre subcategories.
 Hence one can form the abelian quotient categories
$\widetilde K\Modl/\widetilde K\Modlz$ and
$\widecheck K\Modl/\widecheck K\Modlz$.
 There are natural equivalences of categories
\begin{equation} \label{four-category-equivalence}
 K\Modlt\simeq\widetilde K\Modl/\widetilde K\Modlz\simeq
 \widecheck K\Modl/\widecheck K\Modlz\simeq K\Modlc
\end{equation}
\cite[Theorems~4.5 and~5.6, and Proposition~9.1]{Quil},
\cite[Theorem~5.8]{Pnun}.
 Thus the categories $K\Modlt$ and $K\Modlc$ are abelian.

 The equivalences of categories~\eqref{four-category-equivalence}
make the fully faithful inclusion functors $K\Modlt\rarrow
\widetilde K\Modl$ and $K\Modlc\rarrow\widetilde K\Modl$ adjoint on
the left and on the right, respectively, to the Serre quotient functor
$\widetilde K\Modl\rarrow\widetilde K\Modl/\widetilde K\Modlz$.
 The same applies to the inclusion functors
$K\Modlt\rarrow\widecheck K\Modl$
and $K\Modlc\rarrow\widecheck K\Modl$ (and the Serre quotient functor
$\widecheck K\Modl\rarrow\widecheck K\Modl/\widecheck K\Modlz$).

 The full subcategory $K\Modlt$ is closed under colimits and
extensions in $\widetilde K\Modl$ or $\widecheck K\Modl$
\cite[Lemma~1.5]{Pnun}, but it need not be closed under
kernels~\cite[Example~4.7]{Pnun}.
 Dual-analogously, the full subcategory $K\Modlc$ is closed under
limits and extensions in $\widetilde K\Modl$ or $\widecheck K\Modl$
\cite[Lemma~3.5]{Pnun}, but it need not be closed under
cokernels~\cite[Examples~4.2 and~4.8]{Pnun}.

\subsection{t-Flat, c-projective, and t-injective modules}
\label{t-c-flat-proj-inj-prelim-subsecn}
 In this section we assume $K$ to be a t\+unital $k$\+algebra.

 A unital left $\widetilde K$\+module $Q$ is called
\emph{c\+projective}~\cite[Definition~8.1]{Pnun} if the covariant
$\Hom$ functor $\Hom_{\widetilde K}(Q,{-})$ takes cokernels
\emph{computed within the full subcategory} $K\Modlc\subset
\widetilde K\Modl$ to cokernels in~$\Ab$.
 It is clear from the second case of
Lemma~\ref{hom-tensor-over-tilde-check-agree}(b) that a unital left
$\widecheck K$\+module $Q$ is c\+projective (as
a $\widetilde K$\+module) if and only if the functor
$\Hom_{\widecheck K}(Q,{-})$ takes cokernels computed within the full
subcategory $K\Modlc\subset\widecheck K\Modl$ to cokernels in $k\Vect$.

 A unital $\widetilde K$\+module is c\+projective if and only if it
represents a projective object in the quotient category
$\widetilde K\Modl/\widetilde K\Modlz$ \,\cite[Remark~8.5]{Pnun}.
 By~\cite[Corollary~8.6]{Pnun}, a t\+unital $K$\+module is c\+projective
if and only if it is a projective unital $\widetilde K$\+module.
 Similarly one proves that a t\+unital $K$\+module is c\+projective if
and only if it is a projective unital $\widecheck K$\+module (just do
the same argument over a field instead of over~$\boZ$).
 We arrive to the following corollary.

\begin{cor} \label{projectivity-over-tilde-check-equivalence}
 Let $K$ be a t\+unital $k$\+algebra and $Q$ be a t\+unital $K$\+module.
 Then $Q$ is a projective $\widetilde K$\+module if and only if it is
a projective $\widecheck K$\+module.  \qed
\end{cor}

 A unital left $\widetilde K$\+module $J$ is called
\emph{t\+injective}~\cite[Definition~8.7]{Pnun} if the contravariant
$\Hom$ functor $\Hom_{\widetilde K}({-},J)$ takes kernels
\emph{computed within the full subcategory} $K\Modlt\subset
\widetilde K\Modl$ to cokernels in~$\Ab$.
 It is clear from the first case of
Lemma~\ref{hom-tensor-over-tilde-check-agree}(b) that a unital
left $\widecheck K$\+module $J$ is t\+injective (as
a $\widetilde K$\+module) if and only if the functor
$\Hom_{\widecheck K}({-},J)$ takes kernels computed within the full
subcategory $K\Modlt\subset\widecheck K\Modl$ to cokernels in $k\Vect$.

 A unital $\widetilde K$\+module is t\+injective if and only if it
represents an injective object in the quotient category
$\widetilde K\Modl/\widetilde K\Modlz$ \,\cite[Remark~8.11]{Pnun}.
 By~\cite[Corollary~8.12]{Pnun}, a c\+unital $K$\+module is t\+injective
if and only if it is an injective unital $\widetilde K$\+module.
 Similarly one proves that a c\+unital $K$\+module is t\+injective if
and only if it is an injective unital $\widecheck K$\+module.
 We arrive to the following corollary, which is a particular case
of~\cite[last assertion of Proposition~9.2]{Quil}.

\begin{cor}[\cite{Quil}]
\label{injectivity-over-tilde-check-equivalence}
 Let $K$ be a t\+unital $k$\+algebra and $J$ be a c\+unital $K$\+module.
 Then $J$ is an injective $\widetilde K$\+module if and only if it is
an injective $\widecheck K$\+module.  \qed
\end{cor}

 A unital left $\widetilde K$\+module $F$ is called
\emph{t\+flat}~\cite[Definition~8.13]{Pnun} if the tensor product
functor ${-}\ot_{\widetilde K}F$ takes kernels \emph{computed
within the full subcategory} $\Modrt K\subset\Modr\widetilde K$
to kernels in~$\Ab$.
 It is clear from the first case of
Lemma~\ref{hom-tensor-over-tilde-check-agree}(a) that a unital left
$\widecheck K$\+module $F$ is t\+flat (as a $\widetilde K$\+module)
if and only if the tensor product functor ${-}\ot_{\widecheck K}F$
takes kernels computed within the full subcategory
$\Modrt K\subset\Modr\widecheck K$ to kernels in $k\Vect$.

 By~\cite[Corollary~8.20]{Pnun}, a t\+unital $K$\+module is t\+flat
if and only if it is a flat unital $\widetilde K$\+module.
 Similarly one proves that a t\+unital $K$\+module is t\+flat if and
only if it is a flat unital $\widecheck K$\+module.
 We arrive to the next corollary, which is also a particular case
of~\cite[last assertion of Proposition~9.2]{Quil}.

\begin{cor}[\cite{Quil}] \label{flatness-over-tilde-check-equivalence}
 Let $K$ be a t\+unital $k$\+algebra and $F$ be a t\+unital $K$\+module.
 Then $F$ is a flat $\widetilde K$\+module if and only if it is
a flat $\widecheck K$\+module.  \qed
\end{cor}

\begin{prop} \label{t-unital-kernels-flat-ring}
 Let $K$ be a t\+unital $k$\+algebra.  Then the following conditions
are equivalent:
\begin{enumerate}
\item $K$ is a flat right $\widetilde K$\+module;
\item $K$ is a flat right $\widecheck K$\+module;
\item the full subcategory $K\Modlt$ is closed under kernels in
$\widetilde K\Modl$;
\item the full subcategory $K\Modlt$ is closed under kernels in
$\widecheck K\Modl$.
\end{enumerate}
\end{prop}

\begin{proof}
 (1)\,$\Longleftrightarrow$\,(2) is a particular case of
the Wodzicki--Quillen theorem~\cite[Corollary~9.5]{Quil}.
 It is also a particular case of
Corollary~\ref{flatness-over-tilde-check-equivalence}.

 (1)\,$\Longleftrightarrow$\,(3) is~\cite[Corollary~7.4]{Pnun}.
 The implication (1)\,$\Longrightarrow$\,(3) does not depend on
the assumption of t\+unitality of~$K$ \,\cite[Remark~1.6]{Pnun}.
 
 (2)\,$\Longleftrightarrow$\,(4) is similar to
(1)\,$\Longleftrightarrow$\,(3).
 The implication (2)\,$\Longrightarrow$\,(4) does not depend on
the assumption of t\+unitality of~$K$.

 (3)\,$\Longleftrightarrow$\,(4) is straightforward and also does not
depend on the assumption of t\+unitality of~$K$.
\end{proof}

\begin{prop} \label{c-unital-cokernels-projective-ring}
 Let $K$ be a t\+unital $k$\+algebra.  Then the following conditions
are equivalent:
\begin{enumerate}
\item $K$ is a projective left $\widetilde K$\+module;
\item $K$ is a projective left $\widecheck K$\+module;
\item the full subcategory $K\Modlc$ is closed under cokernels in
$\widetilde K\Modl$;
\item the full subcategory $K\Modlc$ is closed under cokernels in
$\widecheck K\Modl$.
\end{enumerate}
\end{prop}

\begin{proof}
 (1)\,$\Longleftrightarrow$\,(2) is a particular case of
Corollary~\ref{projectivity-over-tilde-check-equivalence}.

 (1)\,$\Longleftrightarrow$\,(3) is~\cite[Corollary~7.6]{Pnun}.
 The implication (1)\,$\Longrightarrow$\,(3) does not depend on
the assumption of t\+unitality of~$K$ \,\cite[Remark~3.6]{Pnun}.

 (2)\,$\Longleftrightarrow$\,(4) is similar to
(1)\,$\Longleftrightarrow$\,(3).
 The implication (2)\,$\Longrightarrow$\,(4) does not depend on
the assumption of t\+unitality of~$K$.

 (3)\,$\Longleftrightarrow$\,(4) is straightforward and also does not
depend on the assumption of t\+unitality of~$K$.
\end{proof}

\subsection{nt-Flat and nc-projective modules}
\label{nt-nc-flat-proj-prelim-subsecn}
 In this section $K$ is a nonunital $k$\+algebra.
 The assertions involving t\+flatness or c\+projectivity will
presume $K$ to be t\+unital.

 Let us say that a $\widetilde K$\+module $Q$ is \emph{nc\+projective}
if the covariant $\Hom$ functor $\Hom_{\widetilde K}(Q,{-})$ preserves
right exactness of all right exact sequences $L\rarrow M\rarrow E
\rarrow0$ in $\widetilde K\Modl$ with $L$, $M\in K\Modlc$.
 It is clear that any such right exact sequence
$L\rarrow M\rarrow E\rarrow0$ is actually a right exact sequence
in $\widecheck K\Modl$.

 By~\cite[Lemma~8.3]{Pnun}, a t\+unital module $Q$ over a t\+unital
algebra $K$ is nc\+projective if and only if it is c\+projective.
 According to~\cite[Corollary~8.6]{Pnun}
and Corollary~\ref{projectivity-over-tilde-check-equivalence},
this holds if and only if $Q$ is a projective $\widetilde K$\+module,
or equivalently, a projective $\widecheck K$\+module.
 In particular, the left $K$\+module $K$ is nc\+projective if
and only if it is projective as a left $\widecheck K$\+module.

 On the other hand, it is obvious from the definition that
any projective $\widetilde K$\+module is nc\+projective, and
any projective $\widecheck K$\+module is nc\+projective;
while such modules \emph{need not} be c\+projective in
general~\cite[Remarks~8.2(1\+-2)]{Pnun}.
 When the full subcategory $K\Modlc$ is closed under cokernels in
$\widetilde K\Modl$ or $\widecheck K\Modl$, the classes of
c\+projective and nc\+projective left $K$\+modules coincide.

\begin{lem} \label{nc-projective-tensor-product}
 Let $B\in\widecheck K\Bimod\widecheck K$ be a $K$\+$K$\+bimodule and
$F$ be a left $\widecheck K$\+module.
 Assume that both the left $K$\+modules $B$ and $F$ are nc\+projective,
and the right $K$\+module $B$ is t\+unital.
 Then the tensor product $B\ot_{\widecheck K}F$ is an nc\+projective
left $K$\+module.
\end{lem}

\begin{proof}
 The assertion follows from the natural isomorphism
$\Hom_K(B\ot_KF,\>{-})\simeq\Hom_K(F,\Hom_K(B,{-}))$ together with
the fact that the functor $\Hom_K(B,{-})$ takes all left $K$\+modules
to c\+unital ones~\cite[Lemma~3.2]{Pnun}.
\end{proof}

 Similarly, we will say that a left $\widetilde K$\+module $F$ is
\emph{nt\+flat} if the tensor product functor ${-}\ot_{\widetilde K}F$
preserves left exactness of all left exact sequences $0\rarrow N
\rarrow L\rarrow M$ in $\Modr\widetilde K$ with $L$, $M\in\Modrt K$.
 It is clear that any such left exact sequence
$0\rarrow N\rarrow L\rarrow M$ is actually a left exact sequence
in $\Modr\widecheck K$.

 By~\cite[Lemma~8.15]{Pnun}, a t\+unital module $F$ over a t\+unital
algebra $K$ is nt\+flat if and only if it is t\+flat.
 According to~\cite[Corollary~8.20]{Pnun}
and Corollary~\ref{flatness-over-tilde-check-equivalence},
this holds if and only if $F$ is a flat $\widetilde K$\+module,
or equivalently, a flat $\widecheck K$\+module.
 In particular, the left $K$\+module $K$ is nt\+flat if and only if
it is flat as a left $\widecheck K$\+module.

 On the other hand, it is obvious from the definition that any flat
$\widetilde K$\+module is nt\+flat, and any flat $\widecheck K$\+module
is nt\+flat; while such modules \emph{need not} be t\+flat in
general~\cite[Remarks~8.14(1\+-2)]{Pnun}.
 When the full subcategory $\Modrt K$ is closed under kernels in
$\Modr\widetilde K$ or $\Modr\widecheck K$, the classes of
t\+flat and nt\+flat left $K$\+modules coincide.

\begin{lem} \label{nt-flat-tensor-product}
 Let $B\in\widecheck K\Bimod\widecheck K$ be a $K$\+$K$\+bimodule
and $F$ be a left $\widecheck K$\+module.
 Assume that both the left $K$\+modules $B$ and $F$ are nt\+flat, and
the right $K$\+module $B$ is t\+unital.
 Then the tensor product $B\ot_{\widecheck K}F$ is an nt\+flat
left $K$\+module.
\end{lem}

\begin{proof}
 The assertion follows from the natural isomorphism ${-}\ot_K(B\ot_KF)
\simeq({-}\ot_K\nobreak B)\ot_KF$ together with the fact that
the functor ${-}\ot_KB$ takes all right $K$\+modules to t\+unital
ones~\cite[Lemma~1.2(b)]{Pnun}.
\end{proof}

 In the context of the following proposition, the assumption that
$K$ is a $k$\+algebra is irrelevant; so $K$ can be any nonunital ring.
 For any right $K$\+module $N$, we denote by $N^+=\Hom_\boZ(N,
\boQ/\boZ)$ the character dual left $K$\+module.

\begin{prop} \label{nc-projective-are-nt-flat}
 For any nonunital ring $K$, any nc\+projective $K$\+module is nt\+flat.
\end{prop}

\begin{proof}
 The argument is similar to (but simpler than) \cite[first proof of
Proposition~8.21]{Pnun}.
 Let $Q$ be an nc\+projective left $K$\+module and $0\rarrow N
\rarrow L\rarrow M$ be a left exact sequence in $\Modr\widetilde K$
with $L$, $M\in\Modrt K$.
 Then $M^+\rarrow L^+\rarrow N^+\rarrow0$ is a right exact sequence
in $\widetilde K\Modl$ with $L^+$, $M^+\in K\Modlc$
(by~\cite[Lemma~8.17]{Pnun}).
 Now the sequence of abelian groups $\Hom_{\widetilde K}(Q,M^+)\rarrow
\Hom_{\widetilde K}(Q,L^+)\rarrow\Hom_{\widetilde K}(Q,N^+)\rarrow0$
can be obtained by applying the functor $\Hom_\boZ({-},\boQ/\boZ)$ to
the sequence of abelian groups $0\rarrow N\ot_{\widetilde K}Q\rarrow
L\ot_{\widetilde K}Q\rarrow M\ot_{\widetilde K}Q$.
 The former sequence is right exact by assumption, so it follows that
the latter one is left exact.
\end{proof}

\subsection{Bimodules}
 Let $K$ be a (nonunital) $k$\+algebra.
 As mentioned in the beginning of Section~\ref{prelim-nonunital-secn},
by a \emph{unital $\widecheck K$\+$\widecheck K$\+bimodule} we mean
a $k$\+vector space endowed with commuting structures of a left unital
$\widecheck K$\+module and a right unital $\widecheck K$\+module.
 In other words, the left and right actions of~$k$ in bimodules are
presumed to agree.

 We will say that a $\widecheck K$\+$\widecheck K$\+bimodule is
\emph{t\+unital} if it is t\+unital both as a left and as
a right module.
 We will denote the full subcategory of t\+unital
$\widecheck K$\+$\widecheck K$\+bimodules by
$K\tBimodt K\subset\widecheck K\Bimod\widecheck K$ and call its objects
simply ``t\+unital $K$\+$K$\+bimodules''.

\begin{prop} \label{t-c-unital-monoidal-model-categories}
 Let $K$ be a t\+unital $k$\+algebra.  Then \par
\textup{(a)} the additive category of t\+unital $K$\+$K$\+bimodules
$K\tBimodt K$ is an associative and unital monoidal category, with
the unit object $K\in K\tBimodt K$, with respect to the tensor
product operation\/~$\ot_{\widecheck K}$; \par
\textup{(b)} the additive category of t\+unital left $K$\+modules
$K\Modlt$ is an associative and unital left module category over
$K\tBimodt K$, and the additive category of t\+unital right
$K$\+modules $\Modrt K$ is an associative and unital right module
category over $K\tBimodt K$, with respect to the tensor product
operation\/~$\ot_{\widecheck K}$; \par
\textup{(c)} the opposite category $K\Modlc^\sop$ to the additive
category of c\+unital left $K$\+modules $K\Modlc$ is an associative
and unital right module category over the monoidal category
$K\tBimodt K$, with respect to the Hom operation
$$
 P^\sop*B=\Hom_{\widecheck K}(B,P)^\sop
 \quad\text{for all $P\in K\Modlc$ and $B\in K\tBimodt K$}.
$$
\end{prop}

\begin{proof}
 This is essentially~\cite[Corollaries~1.4 and~3.4]{Pnun}.
 The restriction that the left and right actions of~$k$ in
the bimodules agree (which is imposed here, but not in~\cite{Pnun})
does not affect the validity of the assertions.
\end{proof}

\subsection{s-Unital rings and modules}
 The following definitions and a part of the results are due to
Tominaga~\cite{Tom}; we refer to the survey~\cite{Nys} and
the preprint~\cite[Section~2]{Pnun} for additional discussion.

 A left module $M$ over a (nonunital) ring $K$ is said to be
\emph{s\+unital} if for every $m\in M$ there exists $e\in K$
such that $em=m$ in~$M$.
 Similarly, a right $K$\+module $N$ is said to be \emph{s\+unital}
if for every $n\in N$ there exists $e\in K$ such that $ne=n$ in~$N$.

 We will denote the full subcategory of s\+unital left $K$\+modules
by $K\Modls\subset\widetilde K\Modl$ and the full subcategory of
s\+unital right $K$\+modules by $\Modrs K\subset\Modr\widetilde K$.

\begin{prop} \label{s-unital-modules-hereditary-torsion}
 For any (nonunital) ring $K$, the full subcategory of s\+unital
$K$\+modules $K\Modls$ is a hereditary torsion class in the abelian
category $\widetilde K\Modl$ of nonunital $K$\+modules.
 In other words, the full subcategory $K\Modls$ is closed under
submodules, quotients, extensions, and infinite direct sums
in $\widetilde K\Modl$.
\end{prop}

\begin{proof}
 This is~\cite[Proposition~2.2]{Pnun}.
\end{proof}

\begin{prop}[\cite{Tom}] \label{s-unital-simultaneous-for-several}
 Let $K$ be a (nonunital) ring and $M$ be an s\+unital left
$K$\+module.
 Then, for any finite collection of elements $m_1$,~\dots, $m_n\in M$
there exists an element $e\in K$ such that $em_i=m_i$ for all\/
$1\le i\le n$.
\end{prop}

\begin{proof}
 This is~\cite[Theorem~1]{Tom}; see also~\cite[Proposition~2.8 of
the published version or Proposition~8 of the \texttt{arXiv}
version]{Nys} or~\cite[Corollary~2.3]{Pnun}.
\end{proof}

 A ring (or $k$\+algebra) $K$ is said to be \emph{left s\+unital} if
it is an s\+unital left $K$\+module, and \emph{right s\+unital} if
it is an s\+unital right $K$\+module.
 The ring $K$ is called \emph{s\+unital} if it is both left and
right s\+unital.

\begin{prop} \label{s-unital-implies-t-unital}
\textup{(a)} Any left s\+unital ring is t\+unital.
 Similarly, any right s\+unital ring is t\+unital. \par
\textup{(b)} Over a left s\+unital ring, a left module is
t\+unital if and only if it is s\+unital.
 Over a right s\+unital ring, a right module is t\+unital if
and only if it is s\+unital.
\end{prop}

\begin{proof}
 This result is due to Tominaga~\cite[Remarks~(1\+-2) in Section~1
on p.~121\+-122]{Tom}.
 See~\cite[Corollaries~2.7 and~2.9]{Pnun} for the details.
\end{proof}

\begin{cor}
 Let $K$ be a left s\+unital ring.  In this context: \par
\textup{(a)} The full subcategory of t\+unital left $K$\+modules
$K\Modlt$ is closed under kernels in $\widetilde K\Modl$.
 If $K$ is a $k$\+algebra, then $K\Modl$ is also closed under
kernels in $\widecheck K\Modl$. \par
\textup{(b)} The right $\widetilde K$\+module $K$ is flat.
 If $K$ is a $k$\+algebra, then the right $\widecheck K$\+module $K$
is also flat.
\end{cor}

\begin{proof}
 The first assertion of part~(a) follows by comparing
Propositions~\ref{s-unital-modules-hereditary-torsion}
and~\ref{s-unital-implies-t-unital}(b).
 All the other assertions are then provided by
Proposition~\ref{t-unital-kernels-flat-ring}.
\end{proof}

\Section{Preliminaries on Semialgebras over Coalgebras}

 As mentioned in the beginning of Section~\ref{prelim-nonunital-secn},
all \emph{coalgebras} in this paper are (coassociative and) counital;
all comodules are counital and contramodules are contraunital.
 Furthermore, all \emph{semialgebras} are (semiassociative and)
semiunital; all semimodules are semiunital and semicontramodules
are semicontraunital.

 The definitions of semialgebras, semimodules, and semicontramodules
go back to the book~\cite{Psemi}, which is written in the generality
level of ``three-story towers'': a semialgebra over a coring over
a (noncommutative, nonsemisimple, associative, unital) ring.
 Hence in the exposition of~\cite{Psemi} it was necessary to deal with
such issues as the nonassociativity of the cotensor product of
comodules over corings, etc.
 The related definitions and discussion are presented
in~\cite[Chapters~1 and~3]{Psemi}.

 In this paper, we avoid many complications by restricting ourselves
to coalgebras over a field and semialgebras over such coalgebras.
 An introductory discussion of this context can be found
in~\cite[Section~0.3]{Psemi} and~\cite[Sections~2.5\+-2.6]{Prev}.

\subsection{Coalgebras, comodules, and contramodules}
 Before passing to semialgebras, let us say a few words about
coalgebras.
 Very briefly, a (coassociative, counital) \emph{coalgebra} $\C$ is
a $k$\+vector space endowed with $k$\+linear maps of
\emph{comultiplication} $\mu\:\C\rarrow\C\ot_k\C$ and \emph{counit}
$\epsilon\:\C\rarrow k$ satisfying the usual coassociativity and
counitality axioms obtainable by inverting the arrows in the definition
of an associative, unital algebra written down in the tensor notation.

 A (coassociative, counital) \emph{left comodule} $\M$ over $\C$ is
a $k$\+vector space endowed with a $k$\+linear \emph{left coaction}
map $\nu\:\M\rarrow\C\ot_k\M$ satisfying the usual coassociativity
and counitality axioms.
 A \emph{right comodule} $\N$ over $\C$ is a $k$\+vector space endowed
with a \emph{right coaction} map $\nu\:\N\rarrow\N\ot_k\C$ satisfying
the similar axioms.

 The definition of a \emph{left contramodule} over $\C$ is less
familiar, so let us spell it out in a little more detail.
 A left $\C$\+contramodule $\fP$ is a $k$\+vector space endowed
a $k$\+linear \emph{left contraaction} map $\pi\:\Hom_k(\C,\fP)
\rarrow\fP$ satisfying the following \emph{contraassociativity} and
\emph{contraunitality} axioms.

 Firstly (contraassociativity), the two compositions
$$
 \Hom_k(\C\ot_k\C,\>\fP)\simeq\Hom_k(\C,\Hom_k(\C,\fP))
 \rightrightarrows\Hom_k(\C,\fP)\rarrow\fP
$$
of the two maps
$$
 \Hom_k(\mu,\fP)\:\Hom_k(\C\ot_k\C,\>\fP)\rarrow
 \Hom_k(\C,\fP)
$$
and
$$
 \Hom_k(\C,\pi)\:\Hom_k(\C,\Hom_k(\C,\fP))\rarrow\Hom_k(\C,\fP)
$$
with the contraaction map $\pi\:\Hom_k(\C,\fP)\rarrow\fP$ must be
equal to each other.
 Here the presumed identification $\Hom_k(\C\ot_k\C,\>\fP)\simeq
\Hom_k(\C,\Hom_k(\C,\fP))$ is obtained as a particular case of
the natural isomorphism $\Hom_k(U\ot_k V,\>W)\simeq
\Hom_k(V,\Hom_k(U,W))$ for all $k$\+vector spaces $U$, $V$, and~$W$.

 Secondly (contraunitality), the composition
$$
 \fP\rarrow\Hom_k(\C,\fP)\rarrow\fP
$$
of the map $\Hom_k(\epsilon,\fP)\:\fP\rarrow\Hom_k(\C,\fP)$
with the contraaction map $\pi\:\Hom_k(\C,\allowbreak\fP)
\rarrow\fP$ must be equal to the identity map $\id_\fP$.

 The definition of a \emph{right contramodule} over $\C$ is very
similar to that of a left contramodule.
 The only difference is that, in the case of right contramodules $\fQ$,
the identification $\Hom_k(\C\ot_k\C,\>\fQ)\simeq
\Hom_k(\C,\Hom_k(\C,\fQ))$ obtained as a particular case of
the natural isomorphism $\Hom_k(U\ot_k V,\>W)\simeq
\Hom_k(U,\Hom_k(V,W))$ for all $k$\+vector spaces $U$, $V$, and $W$
is used in the contraassociativity axiom.

 For any right $\C$\+comodule $\N$ and any $k$\+vector space $V$,
the vector space $\fP=\Hom_k(\N,V)$ has a natural left
$\C$\+contramodule structure.
 The contraaction map $\pi\:\Hom_k(\C,\fP)\rarrow\fP$ is constructed
as the composition
$$
 \Hom_k(\C,\Hom_k(\N,V))\simeq\Hom_k(\N\ot_k\C,\>V)\rarrow
 \Hom_k(\N,V)
$$
of the natural isomorphism of vector spaces $\Hom_k(\C,\Hom_k(\N,V))
\simeq\Hom_k(\N\ot_k\C,\allowbreak\>V)$ with the map $\Hom_k(\nu,V)\:
\Hom_k(\N\ot_k\C,\>V)\rarrow\Hom_k(\N,V)$ induced by the coaction
map $\nu\:\N\rarrow\N\ot_k\C$.

 We will denote the $k$\+linear category of left $\C$\+comodules by
$\C\Comodl$, the $k$\+linear category of right $\C$\+comodules by
$\Comodr\C$, the $k$\+linear category of left $\C$\+contramodules
by $\C\Contra$, and the $k$\+linear category of right
$\C$\+contramodules by $\Contrar\C$.
 The $k$\+vector space of morphisms $\LL\rarrow\M$ in the category
$\C\Comodl$ is denoted by $\Hom_\C(\LL,\M)$, while the $k$\+vector
space of morphisms $\fP\rarrow\fQ$ in the category $\C\Contra$ is
denoted by $\Hom^\C(\fP,\fQ)$.

 We refer to the book~\cite{Swe}, the introductory
section and appendix~\cite[Section~0.2 and Appendix~A]{Psemi},
the surveys~\cite[Section~1]{Prev} and~\cite[Sections~3 and~8]{Pksurv}, 
and the preprints~\cite[Section~1]{Phff}, \cite[Section~1]{Plfin} for
a further discussion of coalgebras over fields together with comodules
and contramodules over them.

\subsection{Injective comodules and projective contramodules}
 Let $\C$ be a coalgebra over~$k$.
 Then the category of left $\C$\+comodules $\C\Comodl$ is a locally
finite Grothendieck abelian category.
 The forgetful functor $\C\Comodl\rarrow k\Vect$ is exact and
preserves coproducts (but usually does \emph{not} preserve products;
and the products are \emph{not} exact in $\C\Comodl$ in general).

 In particular, there are enough injective objects in $\C\Comodl$
(but usually \emph{no} projectives).
 Left $\C$\+comodules of the form $\C\ot_kV$, where $V\in k\Vect$
ranges over the $k$\+vector spaces, are called the \emph{cofree} left
$\C$\+comodules.
 Similarly, the cofree right $\C$\+comodules are the ones of the form
$V\ot_k\C$, where $V\in k\Vect$.
 For any left $\C$\+comodule $\LL$, there is a natural isomorphism of
$k$\+vector spaces~\cite[Section~1.1.2]{Psemi}
$$
 \Hom_\C(\LL,\>\C\ot_kV)\simeq\Hom_k(\LL,V).
$$
 Hence the cofree $\C$\+comodules are injective.
 A $\C$\+comodule is injective if and only if it is a direct summand
of a cofree one.

 The category of left $\C$\+contramodules $\C\Contra$ is a locally
presentable abelian category with enough projective objects.
 The forgetful functor $\C\Contra\rarrow k\Vect$ is exact and preserves
products (but usually does \emph{not} preserve coproducts; and
the coproducts are \emph{not} exact in $\C\Contra$ in general).

 Left $\C$\+contramodules of the form $\Hom_k(\C,V)$ (with
the $\C$\+contramodule structure arising from the right $\C$\+comodule
structure on~$\C$) are called the \emph{free} left $\C$\+contramodules.
 For any left $\C$\+contramodule $\fQ$, there is a natural isomorphism
of $k$\+vector spaces~\cite[Section~3.1.2]{Psemi}
$$
 \Hom^\C(\Hom_k(\C,V),\fQ)\simeq\Hom_k(V,\fQ).
$$
 Hence the free $\C$\+contramodules are projective.
 A $\C$\+contramodule is projective if and only if it is a direct
summand of a free one.

\subsection{Bicomodules}
 Let $\C$ and $\D$ be two (coassociative, counital) coalgebras over~$k$.
 A \emph{$\C$\+$\D$\+bicomodule} $\B$ is a $k$\+vector space endowed
with a left $\C$\+comodule and a right $\D$\+comodule structures
commuting with each other.
 The latter condition means that the square diagram of coaction maps
$\B\rarrow\C\ot_k\B\rarrow\C\ot_k\B\ot_k\D$, \ 
$\B\rarrow\B\ot_k\D\rarrow\C\ot_k\B\ot_k\D$ is commutative.
 Equivalently, a $\C$\+$\D$\+bicomodule can be defined as a $k$\+vector
space $\B$ endowed with a \emph{bicoaction} map $\B\rarrow\C\ot_k\B
\ot_k\D$ satisfying the coassociativity and counitality axioms.
 We will denote the abelian category of $\C$\+$\D$\+bicomodules
by $\C\Bicomod\D$.

\subsection{Cotensor product and cohomomorphisms}
 Let $\N$ be a right $\C$\+comodule and $\N$ be a left $\C$\+comodule.
 The \emph{cotensor product} $\N\oc_\C\M$ is a $k$\+vector space
constructed as the kernel of (the difference of) the natural pair of
maps
$$
 \N\ot_k\M\rightrightarrows\N\ot_k\C\ot_k\M
$$
one of which is induced by the right coaction map
$\nu_\N\:\N\rarrow\N\ot_k\C$ and the other one by the left
coaction map $\nu_\M\:\M\rarrow\C\ot_k\M$.

 Given three coalgebras $\C$, $\D$, $\E$, a $\C$\+$\D$\+bicomodule
$\N$, and a $\D$\+$\E$\+bicomodule $\M$, the cotensor product
$\N\oc_\D\M$ is naturally endowed with a $\C$\+$\E$\+bicomodule
structure.

 For any left $\C$\+comodule $\M$ and right $\C$\+comodule $\N$,
there are natural cotensor product unitality isomorphisms of vector
spaces (or left/right $\C$\+comodules)
\begin{equation} \label{cotensor-unitality}
 \C\oc_\C\M\simeq\M \quad\text{and}\quad
 \N\oc_\C\C\simeq\N
\end{equation}
\cite[Section~1.2.1]{Psemi}.
 For any right $\C$\+comodule $\N$, any $\C$\+$\D$\+bicomodule $\B$,
and any left $\D$\+comodule $\M$, there is a natural associativity
isomorphism of $k$\+vector spaces
\begin{equation} \label{cotensor-associativity}
 (\N\oc_\C\B)\oc_\D\M \simeq \N\oc_\C(\B\oc_\D\M).
\end{equation}
 In fact, both $(\N\oc_\C\B)\oc_\D\M$ and $\N\oc_\C(\B\oc_\D\M)$ are
one and the same vector subspace in the vector space $\N\ot_k\B\ot_k\M$.

 Let $\M$ be a left $\C$\+comodule and $\fP$ be a left
$\C$\+contramodule.
 The $k$\+vector space of \emph{cohomomorphisms} $\Cohom_\C(\M,\fP)$ is
constructed as the cokernel of (the difference of) the natural pair of
maps
$$
 \Hom_k(\C\ot_k\M,\>\fP)\simeq\Hom_k(\M,\Hom_k(\C,\fP))
 \rightrightarrows\Hom_k(\C,\fP)
$$
one of which is induced by the left coaction map
$\nu_\M\:\M\rarrow\C\ot_k\M$ and the other one by the left
contraaction map $\pi_\fP\:\Hom_k(\C,\fP)\rarrow\fP$.

 Given two coalgebras $\C$ and $\D$, a $\C$\+$\D$\+bicomodule $\B$,
and a left $\C$\+contramodule $\fP$, the vector space of cohomomorphisms
$\Cohom_\C(\B,\fP)$ is naturally endowed with a left $\D$\+contramodule
structure arising from the right $\D$\+comodule structure on~$\B$.

 For any left $\C$\+contramodule $\fP$, there is a natural
$\Cohom$ unitality isomorphism of vector spaces (or left
$\C$\+contramodules)
\begin{equation} \label{cohom-unitality}
 \Cohom_\C(\C,\fP)\simeq\fP
\end{equation}
\cite[Section~3.2.1]{Psemi}.
 For any left $\D$\+comodule $\M$, any $\C$\+$\D$\+bicomodule $\B$, and
any left $\C$\+contramodule $\fP$, there is a natural associativity
isomorphism of $k$\+vector spaces
\begin{equation} \label{cohom-associativity}
 \Cohom_\C(\B\oc_\D\M,\>\fP)\simeq\Cohom_\D(\M,\Cohom_\C(\B,\fP)).
\end{equation}
 In fact, both $\Cohom_\C(\B\oc_\D\M,\>\fP)$ and
$\Cohom_\D(\M,\Cohom_\C(\B,\fP))$ are naturally identified with
the quotient space of the vector space $\Hom_k(\B\ot_k\M,\>\fP)\simeq
\Hom_k(\M,\Hom_k(\B,\fP))$ by one and the same vector subspace.

\subsection{Coflatness, coprojectivity, and coinjectivity}
\label{adjusted-co-contra-mods-subsecn}
 Let $\C$ be a coalgebra over a field~$k$.
 A left $\C$\+comodule $\M$ is called \emph{coflat} if the cotensor
product functor ${-}\oc_\C\M\:\Comodr\C\rarrow k\Vect$ is exact.
 A left $\C$\+comodule $\M$ is called \emph{coprojective} if
the covariant $\Cohom$ functor $\Cohom_\C(\M,{-})\:\C\Contra
\rarrow k\Vect$ is exact.
 A left $\C$\+contramodule $\fP$ is called \emph{coinjective} if
the contravariant $\Cohom$ functor $\Cohom_\C({-},\fP)\:
\C\Comodl^\sop\rarrow k\Vect$ is exact.

 It is not difficult to show that a $\C$\+comodule is coflat if and
only if it is coprojective, and if and only if it is
injective~\cite[Lemma~3.1(a)]{Prev}.
 The assertion that a $\C$\+contramodule is coinjective if and only if
it is projective is also true, but a bit harder to
prove~\cite[Lemma~3.1(b)]{Prev}.

\begin{lem} \label{cotensor-product-injective}
 Let\/ $\C$ and\/ $\D$ be coalgebras, $\B$ be
a\/ $\C$\+$\D$\+bicomodule, and\/ $\J$ be a left\/ $\D$\+comodule.
 Assume that the left\/ $\C$\+comodule\/ $\B$ and the left\/
$\D$\+comodule\/ $\J$ are injective.
 Then the left\/ $\C$\+comodule\/ $\B\oc_\D\J$ is also injective.
\end{lem}

\begin{proof}
 One can replace comodule injectivity by coflatness and argue that
the composition of two exact cotensor product functors is exact,
using the associativity of cotensor products.
 Alternatively, let the left $\D$\+comodule $\J$ be a direct summand of
a cofree left $\D$\+comodule $\D\ot_kV$, where $V\in k\Vect$; then
the left $\C$\+comodule $\B\oc_\D\J$ is a direct summand of
the left $\C$\+comodule $\B\oc_\D(\D\ot_kV)\simeq\B\ot_kV$.
\end{proof}

\subsection{Semialgebras} \label{semialgebras-prelim-subsecn}
 Let $\C$ be a (coassociative, counital) coalgebra over~$k$.
 The cotensor product unitality and associativity
isomorphisms~(\ref{cotensor-unitality}\+-\ref{cotensor-associativity})
tell that the category of $\C$\+$\C$\+bicomodules $\C\Bicomod\C$ is
an associative, unital monoidal category with the unit object $\C$
with respect to the cotensor product operation~$\oc_\C$.
 
 A \emph{semialgebra} over $\C$ is an associative, unital monoid
object in this monoidal category.
 In other words, a semialgebra $\bS$ over $\C$ is
a $\C$\+$\C$\+bicomodule endowed with a \emph{semimultiplication} map
$\bm\:\bS\oc_\C\bS\rarrow\bS$ and a \emph{semiunit} map
$\be\:\C\rarrow\bS$, which must be $\C$\+$\C$\+bicomodule morphisms
satisfying the following \emph{semiassociativity} and
\emph{semiunitality} axioms.

 Firstly (semiassociativity), the two compositions
$$
 \bS\oc_\C\bS\oc_\C\bS\rightrightarrows\bS\oc_\C\bS\rarrow\bS
$$
of the two maps $\bm\oc_\C\id_{\bS}$, $\id_{\bS}\oc_\C\bm\:
\bS\oc_\C\bS\oc_\C\bS\rightrightarrows\bS\oc_\C\bS$
induced by the semimultiplication map~$\bm$ with the semimultiplication
map $\bm\:\bS\oc_\C\bS\rarrow\bS$ must be equal to each other.
 Secondly (semiunitality), the two compositions
$$
 \bS\rightrightarrows\bS\oc_\C\bS\rarrow\bS
$$
of the two maps $\be\oc_\C\id_{\bS}$, $\id_{\bS}\oc_\C\be\:
\bS\rightrightarrows\bS\oc_\C\bS$ induced by the seminunit map
$\be\:\C\rarrow\bS$ with the semimultiplication map~$\bm$ must be
equal to the identity map~$\id_{\bS}$
\,\cite[Sections~0.3.2 and~1.3.1]{Psemi}, \cite[Section~2.6]{Prev}.

\subsection{Semimodules} \label{semimodules-prelim-subsecn}
 Let $\C$ be a coalgebra over~$k$.
 The unitality and associativity
isomorphisms~(\ref{cotensor-unitality}\+-\ref{cotensor-associativity})
also tell that the category of left $\C$\+comodules $\C\Comodl$ is
an associative, unital left module category, with respect to
the cotensor product operation~$\oc_\C$, over the monoidal category
$\C\Bicomod\C$.

 Let $\bS$ be a semialgebra over~$\C$.
 A \emph{left semimodule} $\bM$ over $\bS$ is a module object in
the module category $\C\Comodl$ over the monoid object $\bS$ in
the monoidal category $\C\Bicomod\C$.
 In other words, a left semimodule $\bM$ over $\bS$ is
a left $\C$\+comodule endowed with a \emph{left semiaction} map
$\bn\:\bS\oc_\C\bM\rarrow\bM$, which must be a left $\C$\+comodule
morphism satisfying the following semiassociativity and seminunitality
axioms.

 Firstly (semiassociativity), the two compositions
$$
 \bS\oc_\C\bS\oc_\C\bM\rightrightarrows\bS\oc_\C\bM\rarrow\bM
$$
of the two maps $\bm\oc_\C\id_{\bM}$, $\id_{\bS}\oc_\C\bn\:
\bS\oc_\C\bS\oc_\C\bM\rightrightarrows\bS\oc_\C\bM$ induced
by the semimultiplication and semiaction maps $\bm$ and~$\bn$
with the semiaction map $\bn\:\bS\oc_\C\bM\rarrow\bM$ must be
equal to each other.
 Secondly (semiunitality), the composition
$$
 \bM\rarrow\bS\oc_\C\bM\rarrow\bM
$$
of the map $\be\oc_\C\id_{\bM}\:\bM\rarrow\bS\oc_\C\bM$ induced by
the semiunit map~$\be$ with the semiaction map~$\bn$ must be
equal to the identity map~$\id_{\bM}$
\,\cite[Sections~0.3.2 and~1.3.1]{Psemi}, \cite[Section~2.6]{Prev}.

 Similarly, the category of right $\C$\+comodules $\Comodr\C$ is
an associative, unital right module category, with respect to
the cotensor product operation~$\oc_\C$, over the monoidal category
$\C\Bicomod\C$.
 A \emph{right semimodule} over $\bS$ is a module object in
the module category $\Comodr\C$ over the monoid object $\bS\in
\C\Bicomod\C$.
 In other words, a right semimodule $\bN$ over $\bS$ is
a right $\C$\+comodule endowed with a \emph{right semiaction} map
$\bn\:\bN\oc_\C\bS\rarrow\bN$, which must be a right $\C$\+comodule
morphism satisfying the semiassociativity and semiunitality axioms
similar to the above ones.

 We will denote the $k$\+linear category of left $\bS$\+semimodules by
$\bS\Simodl$ and the $k$\+linear category of right $\bS$\+semimodules
by $\Simodr\bS$.
 The $k$\+vector space of morphisms $\bL\rarrow\bM$ in the category
$\bS\Simodl$ is denoted by $\Hom_{\bS}(\bL,\bM)$.

\begin{prop} \label{comodule-categories-abelian}
 Let\/ $\bS$ be a semialgebra over a coalgebra\/ $\C$ over a field~$k$.
 Then \par
\textup{(a)} the following two conditions on\/ $\bS$ are equivalent:
\begin{itemize}
\item the category of left\/ $\bS$\+semimodules\/ $\bS\Simodl$ is
abelian \emph{and} the forgetful functor\/ $\bS\Simodl\rarrow\C\Comodl$
is exact;
\item $\bS$ is an injective right\/ $\C$\+comodule;
\end{itemize} \par
\textup{(b)} the following two conditions on\/ $\bS$ are equivalent:
\begin{itemize}
\item the category of right\/ $\bS$\+semimodules \/ $\Simodr\bS$ is
abelian \emph{and} the forgetful functor\/ $\Simodr\bS\rarrow\Comodr\C$
is exact;
\item $\bS$ is an injective left\/ $\C$\+comodule.
\end{itemize}
\end{prop}

\begin{proof}
 This is a semialgebra version of~\cite[Proposition~2.12(a)]{Prev}.
 The point is that a comodule is coflat if and only if it is injective
(see Section~\ref{adjusted-co-contra-mods-subsecn}).
 If $\bS$ is an injective right $\C$\+comodule, then the functor
$\bS\oc_\C{-}\,\:\C\Comodl\rarrow\C\Comodl$ is exact, and one can
construct a left $\bS$\+semimodule structure on the kernel and
cokernel of the underlying $\C$\+comodule map of any left
$\bS$\+semimodule morphism.
 Conversely, one needs to observe that the induction functor
$\bS\oc_\C{-}\,\:\C\Comodl\rarrow\bS\Simodl$ is always left adjoint to
the forgetful functor $\bS\Simodl\rarrow\C\Comodl$
\,\cite[Section~0.3.2, Lemma~1.1.2, and Section~1.3.1]{Psemi}.
 Hence, if $\bS\Simodl$ is abelian, then the induction functor is
right exact.
 If the forgetful functor is exact, it follows that the composition
$\C\Comodl\rarrow\bS\Simodl\rarrow\C\Comodl$ is right exact.
 This composition is the cotensor product functor
$\bS\oc_\C{-}\,\:\C\Comodl\rarrow\C\Comodl$.
 It remains to notice that the cotensor product over a coalgebra over
a field is always left exact (since it is constructed as the kernel of
a map of tensor products over~$k$).
\end{proof}

\subsection{Semicontramodules} \label{semicontramodules-prelim-subsecn}
 Let $\C$ be a coalgebra over~$k$.
 The $\Cohom$ unitality and associativity
isomorphisms~(\ref{cohom-unitality}\+-\ref{cohom-associativity}) can be
interpreted as saying that the opposite category $\C\Contra^\sop$
to the category of left $\C$\+contramodules $\C\Contra$ is
an associative, unital right module category, with respect to
the operation of cohomomorphisms
$$
 \fP^\sop*\B=\Cohom_\C(\B,\fP)^\sop
 \quad\text{for $\fP\in\C\Contra$ and $\B\in\C\Bicomod\C$},
$$
over the monoidal category $\C\Bicomod\C$.
 The (objects opposite to the) module objects in $\C\Contra^\sop$ over
a monoid object $\bS\in\C\Bicomod\C$ are called the \emph{left
semicontramodules} over~$\bS$.

 Explicitly, a \emph{left semicontramodule} $\bP$ over $\bS$ is
a left $\C$\+contramodule endowed with a \emph{left semicontraaction}
map $\bp\:\bP\rarrow\Cohom_\C(\bS,\bP)$, which must be a left
$\C$\+contramodule morphism satisfying the following
semicontraassociativity and semicontraunitality equations.

 Firstly (semicontraassociativity), the two compositions
$$
 \bP\rarrow\Cohom_\C(\bS,\bP)\rightrightarrows
 \Cohom_\C(\bS\oc_\C\bS,\>\bP)\simeq\Cohom_\C(\bS,\Cohom_\C(\bS,\bP))
$$
of the semicontraaction map $\bp\:\bP\rarrow\Cohom_\C(\bS,\bP)$
with the two maps
$$
 \Cohom_\C(\bm,\bP)\:\Cohom_\C(\bS,\bP)\rarrow
 \Cohom_\C(\bS\oc_\C\bS,\>\bP)
$$
and
$$
 \Cohom_\C(\bS,\bp)\:\Cohom_\C(\bS,\bP)\rarrow
 \Cohom_\C(\bS,\Cohom_\C(\bS,\bP))
$$
must be equal to each other.
 Here the presumed identification
$\Cohom_\C(\bS\oc_\C\bS,\>\bP)\simeq\Cohom_\C(\bS,\Cohom_\C(\bS,\bP))$
is obtained as a particular case of the $\Cohom$ associativity
isomorphism~\eqref{cohom-associativity}.

 Secondly (semicontraunitality), the composition
$$
 \bP\rarrow\Cohom_\C(\bS,\bP)\rarrow\bP
$$
of the semicontraaction map $\bp\:\bP\rarrow\Cohom_\C(\bS,\bP)$ with
the map $\Cohom_\C(\be,\bP)\:\allowbreak\Cohom_\C(\bS,\bP)\rarrow
\Cohom_\C(\C,\bP)\simeq\bP$ must be equal to the identity
map~$\id_{\bP}$.
 Here the presumed identification $\Cohom_\C(\C,\bP)\simeq\bP$ is
provided by the $\Cohom$ unitality isomorphism~\eqref{cohom-unitality}
\,\cite[Sections~0.3.5 and~3.3.1]{Psemi}, \cite[Section~2.6]{Prev}.

 We will denote the $k$\+linear category of left
$\bS$\+semicontramodules by $\bS\Sicntr$.
 The definition of the category of right $\bS$\+semicontramodules
$\SicntrR\bS$ is similar (we suppress the details).
 The $k$\+vector space of morphisms $\bP\rarrow\bQ$ in
the category $\bS\Sicntr$ is denoted by $\Hom^{\bS}(\bP,\bQ)$.

\begin{prop} \label{contramodule-category-abelian}
 Let\/ $\bS$ be a semialgebra over a coalgebra\/ $\C$ over a field~$k$.
 Then the following two conditions on\/ $\bS$ are equivalent:
\begin{itemize}
\item the category of left\/ $\bS$\+semicontramodules\/ $\bS\Sicntr$ is
abelian \emph{and} the forgetful functor\/ $\bS\Sicntr\rarrow\C\Contra$
is exact;
\item $\bS$ is an injective left\/ $\C$\+comodule.
\end{itemize} \par
\end{prop}

\begin{proof}
 This is a semialgebra version of~\cite[Proposition~2.12(b)]{Prev}.
 The point is that a comodule is coprojective if and only if it is
injective (see Section~\ref{adjusted-co-contra-mods-subsecn}).
 If $\bS$ is an injective left $\C$\+comodule, then the functor
$\Cohom_\C(\bS,{-})\:\C\Contra\rarrow\C\Contra$ is exact, and one can
construct a left $\bS$\+semicontramodule structure on the kernel and
cokernel of the underlying $\C$\+contramodule map of any left
$\bS$\+semi\-con\-tra\-mod\-ule morphism.
 Conversely, one needs to observe that the coinduction functor
$\Cohom_\C(\bS,{-})\:\C\Contra\rarrow\bS\Sicntr$ is always right
adjoint to the forgetful functor $\bS\Sicntr\rarrow\C\Contra$
\,\cite[Sections~0.3.5 and~3.3.1]{Psemi}.
 Hence, if $\bS\Sicntr$ is abelian, then the coinduction functor is
left exact.
 If the forgetful functor is exact, it follows that the composition
$\C\Contra\rarrow\bS\Sicntr\rarrow\C\Contra$ is left exact.
 This composition is the $\Cohom$ functor
$\Cohom_\C(\bS,{-})\:\C\Contra\rarrow\C\Contra$.
 It remains to notice that the $\Cohom$ functor over a coalgebra over
a field is always right exact (since it is constructed as the cokernel
of a map of $\Hom$ spaces over~$k$).
\end{proof}

\Section{Left Flat, Right Integrable Bimodules}
\label{flat-integrable-bimodules-secn}

 Starting from this section, everything in this paper happens over
a field~$k$.
 All (possibly nonunital) rings are $k$\+algebras and all (nonunital)
modules are $k$\+vector spaces; so by a ``$K$\+module'' we mean
a unital $\widecheck K$\+module.
 All module morphisms are, at least, $k$\+linear; so the notation
$\Hom_K(M,P)$ means $\Hom_{\widecheck K}(M,P)$; and similarly for
the tensor product, $N\ot_KM$ means $N\ot_{\widecheck K}M$.
 Keeping in mind Lemma~\ref{hom-tensor-over-tilde-check-agree}, this
notation change mostly makes no difference.

 We will use the notation $K\Modln=\widecheck K\Modl$ and
$\Modrn K=\Modr\widecheck K$ for the categories of nonunital left
and right $K$\+modules.
 The category of nonunital $K$\+$K$\+bimodules is denoted by
$K\nBimodn K=\widecheck K\Bimod\widecheck K=
(\widecheck K\ot_k\widecheck K^\rop)\Modl$.
 Then we have the full subcategories of t\+unital modules
$K\Modlt\subset K\Modln$ and $\Modrt K\subset\Modrn K$,
the full subcategory of t\+unital bimodules $K\tBimodt K\subset
K\nBimodn K$, and the full subcategory of c\+unital modules
$K\Modlc\subset K\Modln$.

\subsection{Multiplicative pairings}
\label{multiplicative-pairings-subsecn}
 The following definition is a modification of the one
from~\cite[Section~10.1.2]{Psemi}.
 Let $K$ be a (nonunital) $k$\+algebra and $\C$ be a (counital)
coalgebra over~$k$.
 A \emph{multiplicative pairing}
\begin{equation} \label{mult-pairing}
 \phi\:\C\times K\lrarrow k
\end{equation}
is a bilinear map satisfying the equation
\begin{equation} \label{mult-pairing-defined}
 \phi(c,r'r'')=\phi(c_{(1)},r'')\phi(c_{(2)},r')
\end{equation}
for all $c\in\C$ and $r'$, $r''\in K$.
 Here $\mu(c)=c_{(1)}\ot c_{(2)}$ is a simplified version of
Sweedler's symbolic notation~\cite[Section~1.2]{Swe} for
the comultiplication in~$\C$.

 Equivalently, one can pass from~$\phi$ to the induced/partial dual
map $\phi^*\:K\rarrow\C^*=\Hom_k(\C,k)$.
 Then $\phi$~is a multiplicative pairing if and only if $\phi^*$~is
a morphism of (nonunital) $k$\+algebras.

\begin{rem}
 Notice that the factors are switched in
the formula~\eqref{mult-pairing-defined}.
 This corresponds to our convention (unusual for the theory of
coalgebras over a field) concerning the definition of
multiplication in~$\C^*$.
 Our preference, reflected in the formula and assertion above, is
to define the multiplication in $\C^*$ so that left $\C$\+comodules
become left $\C^*$\+modules and right $\C$\+comodules become right
$\C^*$\+modules; then left $\C$\+contramodules also become left
$\C^*$\+modules.
 This is mentioned in the surveys~\cite[beginning of Section~1.4]{Prev}
and~\cite[Sections~6.3 and~9.2]{Pksurv}.
 In the more general context of corings over noncommutative rings,
there is no choice, and our convention becomes the only consistent
one (as in~\cite[Section~10.1.2]{Psemi}).
\end{rem}

\begin{prop} \label{underlying-K-module-structures}
 Let $K$ be a (nonunital) $k$\+algebra and\/ $\C$ be a (counital)
coalgebra over~$k$.
 Then the datum of a multiplicative pairing\/~$\phi$ induces exact,
faithful functors
\begin{alignat}{2}
 \phi_*&\:\C\Comodl &&\lrarrow K\Modln,
 \label{left-comodules-to-modules} \\
 \phi_*&\:\Comodr\C &&\lrarrow \Modrn K,
 \label{right-comodules-to-modules} \\
 \phi_*&\:\C\Contra &&\lrarrow K\Modln
 \label{left-contramodules-to-modules}
\end{alignat}
assigning to every\/ $\C$\+comodule or\/ $\C$\+contramodule
the induced $K$\+module structure on the same vector space.
\end{prop}

\begin{proof}
 The functors in question are the compositions
\begin{alignat*}{3}
 &\C\Comodl &&\overset{\Upsilon_\C}\lrarrow \C^*\Modl
  &&\lrarrow K\Modln, \\
 &\Comodr\C &&\overset{\Upsilon_{\C^\rop}}\lrarrow \Modr\C^*
  &&\lrarrow \Modrn K, \\
 &\C\Contra &&\overset{\Theta_\C}\lrarrow \C^*\Modl
  &&\lrarrow K\Modln
\end{alignat*}
of the comodule/contramodule inclusion/forgetful functors (depending
only on a coalgebra~$\C$) with the functors of restriction of scalars
with respect to the ring homomorphism $\phi^*\:K\rarrow\C^*$.
 We refer to the book~\cite[Section~2.1]{Swe} and
the preprint~\cite{Phff} for further details on
the comodule/contramodule inclusion/forgetful functors.

 The functors~(\ref{left-comodules-to-modules}\+-%
\ref{left-contramodules-to-modules}) can be also constructed
directly as follows.
 Given a left $\C$\+comodule $\M$, the composition
$$
 K\ot_k\M\overset\nu\lrarrow K\ot_k\C\ot_k\M\overset\phi
 \lrarrow\M
$$
of the map $K\ot_k\M\rarrow K\ot_k\C\ot_k\M$ induced by the coaction
map $\nu\:\M\rarrow\C\ot_k\M$ with the map $K\ot_k\C\ot_k\M\rarrow\M$
induced by the pairing $\phi\:\C\times K\rarrow k$
defines a left $K$\+module structure on~$\M$.
 This is the functor~\eqref{left-comodules-to-modules}.
 
 Given a right $\C$\+comodule $\N$, the similar composition
$$
 \N\ot_k K\overset\nu\lrarrow\N\ot_k\C\ot_k K\overset\phi
 \lrarrow\N
$$
of the map $\N\ot_k K\rarrow\N\ot_k\C\ot_kK$ induced by the coaction
map $\nu\:\N\rarrow\N\ot_k\C$ with the map $\N\ot_k\C\ot_k K\rarrow\N$
induced by the pairing $\phi\:\C\times K\rarrow k$
defines a left $K$\+module structure on~$\N$.
 This is the functor~\eqref{right-comodules-to-modules}.

 Given a left $\C$\+contramodule $\fP$, the composition
$$
 K\ot_k\fP\overset\phi\lrarrow\Hom_k(\C,\fP)\overset\pi\lrarrow\fP
$$
of the map $r\ot_k p\longmapsto(c\mapsto\phi(c,r)p)\:K\ot_k\fP\rarrow
\Hom_k(\C,\fP)$ induced by the pairing $\phi\:\C\times K\rarrow k$
with the contraaction map $\pi\:\Hom_k(\C,\fP)\rarrow\fP$ defines
a left $K$\+module structure on~$\fP$.
 This is the functor~\eqref{left-contramodules-to-modules}.
\end{proof}

\subsection{t-Unital multiplicative pairings}
\label{t-unital-pairings-subsecn}
 Let $\phi\:\C\times K\rarrow k$ be a multiplicative pairing for
an algebra $K$ and a coalgebra~$\C$.
 Then the construction of
the functors~$\phi_*$ \,(\ref{left-comodules-to-modules}\+-%
\ref{right-comodules-to-modules}) from
Proposition~\ref{underlying-K-module-structures} endows
the $\C$\+$\C$\+bicomodule $\C\in\C\Bicomod\C$ with
a $K$\+$K$\+bimodule structure, $\C\in K\nBimodn K$.

 We will say the pairing~$\phi$ is \emph{right t\+unital} if
$\C$ is a t\+unital right $K$\+module.
 Similarly, the pairing~$\phi$ is called \emph{left t\+unital} if
$\C$ is a t\+unital left $K$\+module.

\begin{prop} \label{t-unital-pairing-comodules-prop}
 Let $K$ be a nonunital $k$\+algebra, $\C$ be a coalgebra over~$k$,
and $\phi\:\C\times K\rarrow k$ be a multiplicative pairing.
 Assume that the full subcategory $\Modrt K$ is closed under kernels
in $\Modrn K$ (e.~g., this holds if $K$ is a flat left
$\widecheck K$\+module; see
Proposition~\ref{t-unital-kernels-flat-ring}).
 Then the following conditions are equivalent:
\begin{enumerate}
\item the pairing~$\phi$ is right t\+unital;
\item for every right\/ $\C$\+comodule\/ $\N$, the right $K$\+module
$\phi_*(\N)$ is t\+unital;
\item for every finite-dimensional right\/ $\C$\+comodule\/ $\N$,
the right $K$\+module $\phi_*(\N)$ is t\+unital.
\end{enumerate}
\end{prop}

\begin{proof}
 The implications (2)\,$\Longrightarrow$\,(1) and
(2)\,$\Longrightarrow$\,(3) are obvious.

 (1)\,$\Longrightarrow$\,(2) By~\cite[Proposition~4.2]{Quil}
or~\cite[Lemma~1.5]{Pnun}, the full subcategory $\Modrt K$ is closed
under direct sums in $\Modrn K$.
 Hence, if the right $\C$\+comodule $\C$ is a t\+unital right
$K$\+module, then all cofree right $\C$\+comodules are t\+unital
right $K$\+modules.
 It remains to recall that any $\C$\+comodule is the kernel of
a morphism of cofree ones.
 
 (3)\,$\Longrightarrow$\,(2) By~\cite[Proposition~4.2]{Quil}
or~\cite[Lemma~1.5]{Pnun}, the full subcategory $\Modrt K$ is closed
under direct limits in $\Modrn K$.
 So it remains to use the fact that any $\C$\+comodule is the union of
its finite-dimensional subcomodules~\cite[Theorem~2.1.3(b)]{Swe}.
\end{proof}

\begin{prop} \label{t-unital-pairing-contramodules-prop}
 Let $K$ be a nonunital $k$\+algebra, $\C$ be a coalgebra over~$k$,
and $\phi\:\C\times K\rarrow k$ be a multiplicative pairing.
 Assume that the full subcategory $K\Modlc$ is closed under cokernels
in $K\Modln$ (e.~g., this holds if $K$ is a projective left
$\widecheck K$\+module; see
Proposition~\ref{c-unital-cokernels-projective-ring}).
 Then the following conditions are equivalent:
\begin{enumerate}
\item the pairing~$\phi$ is right t\+unital;
\item for the free left\/ $\C$\+contramodule\/ $\Hom_k(\C,k)$,
the left $K$\+module $\phi_*(\Hom_k(\C,k))$ is c\+unital;
\item for every left\/ $\C$\+contramodule\/ $\fP$, the left $K$\+module
$\phi_*(\fP)$ is c\+unital.
\end{enumerate}
\end{prop}

\begin{proof}
 (1)\,$\Longleftrightarrow$\,(2) One needs to observe that
the functors~$\phi_*$ \,(\ref{right-comodules-to-modules}\+-%
\ref{left-contramodules-to-modules}) form a commutative square
diagram with the functor $\Hom_k({-},V)\:\Comodr\C\rarrow\C\Contra$
and the functor $\Hom_k({-},V)\:\Modrn K\rarrow K\Modln$ for
any $k$\+vector space~$V$ (in particular, for $V=k$).
 Then it remains to use Lemma~\ref{duality-preserves-reflects-t-c}.
 
 (3)\,$\Longrightarrow$\,(2) Obvious.
 
 (1)~or~(2)~$\Longrightarrow$~(3) By
Lemma~\ref{duality-preserves-reflects-t-c}, all free left
$\C$\+contramodules $\Hom_k(\C,V)$ are c\+unital left
$K$\+modules under the assumption of either~(1) or~(2).
 It remains to recall that any left $\C$\+contramodule is the cokernel
of a morphism of free ones.
\end{proof}

\begin{prop} \label{s-unital-pairings-characterized}
 Let $K$ be a nonunital $k$\+algebra, $\C$ be a coalgebra over~$k$,
and $\phi\:\C\times K\rarrow k$ be a multiplicative pairing.
 Then the following conditions are equivalent:
\begin{enumerate}
\item the right $K$\+module\/ $\C$ is s\+unital;
\item the left $K$\+module\/ $\C$ is s\+unital;
\item for every finite-dimensional vector subspace $W\subset\C$
there exists an element $e\in K$ such that $\epsilon_\C(w)=\phi(w,e)$
for all $w\in W$ (where $\epsilon_\C\:\C\rarrow k$ denotes
the counit map).
\end{enumerate}
\end{prop}

\begin{proof}
 (3)\,$\Longrightarrow$\,(1) Let $c\in\C$ be an element.
 Choose a finite-dimensional vector subspace $W\subset\C$ for which
$\mu(c)\in\C\ot_kW\subset\C\ot_k\C$.
 Let $e\in K$ be an element such that $\epsilon_\C(w)=\phi(w,e)$
for all $w\in W$.
 Then $ce=c_{(1)}\phi(c_{(2)},e)=c_{(1)}\epsilon_\C(c_{(2)})=c$,
where $\mu(c)=c_{(1)}\ot c_{(2)}\in\C\ot_k\C$ is Sweedler's
notation for the comultiplication and $c_{(1)}\epsilon_\C(c_{(2)})=c$
is the counitality equation for the coalgebra~$\C$.

 (1)\,$\Longrightarrow$\,(3)
 By (the left-right opposite version of)
Proposition~\ref{s-unital-simultaneous-for-several}, there exists
an element $e\in K$ such that $we=w$ in $\C$ for all $w\in W$.
 So we have $w_{(1)}\phi(w_{(2)},e)=w$, where $w_{(1)}\ot w_{(2)}
=\mu(w)\in\C\ot_k\C$.
 Applying the linear map~$\epsilon_\C$, we obtain
$\epsilon_\C(w_{(1)})\phi(w_{(2)},e)=\epsilon_\C(w)$,
and it remains to point out that the counitality equation
$\epsilon_\C(w_{(1)})w_{(2)}=w$ implies
$\epsilon_\C(w_{(1)})\phi(w_{(2)},e)=\phi(w,e)$.
\end{proof}

 We will say that a multiplicative pairing $\phi\:\C\times K\rarrow k$
is \emph{s\+unital} if it satisfies the equivalent conditions of
Proposition~\ref{s-unital-pairings-characterized}.

\begin{cor} \label{s-unital-t-unital-pairings-cor}
 Let $K$ be a right s\+unital $k$\+algebra, $\C$ be a coalgebra
over~$k$, and $\phi\:\C\times K\rarrow k$ be a multiplicative pairing.
 Then the pairing~$\phi$ is right t\+unital if and only if it is
s\+unital.
\end{cor}

\begin{proof}
 Follows from Proposition~\ref{s-unital-implies-t-unital}(b).
\end{proof}

\subsection{Associativity isomorphisms}
 The following result is our version
of~\cite[Propositions~1.2.3(a)]{Psemi}.

\begin{prop} \label{tensor-cotensor-assoc-prop}
 Let $K$ be a nonunital algebra, $\C$ be a coalgebra,
$\N$ be a right\/ $\C$\+comodule, $F$ be a left $K$\+module,
and\/ $\E$ be a left\/ $\C$\+comodule endowed with a right action
of the algebra $K$ by left\/ $\C$\+comodule endomorphisms.
 Then there is a natural map of vector spaces
\begin{equation} \label{tensor-cotensor-assoc-eqn}
 (\N\oc_\C\E)\ot_K F\lrarrow\N\oc_\C(\E\ot_K F).
\end{equation}
 The map~\eqref{tensor-cotensor-assoc-eqn} is an isomorphism whenever\/
$\E$ is a t\+unital right $K$\+module and $F$ is an nt\+flat
left $K$\+module.
\end{prop}

\begin{proof}
 By the definition of the cotensor product, we have a left exact
sequence of right $K$\+modules
\begin{equation} \label{cotensor-product-left-exact-sequence}
 0\lrarrow \N\oc_\C\E\lrarrow\N\ot_k\E\lrarrow\N\ot_k\C\ot_k\E
\end{equation}
and a similar left exact sequence of vector spaces
\begin{equation} \label{cotensor-product-of-tensor-products-sequence}
 0\lrarrow \N\oc_\C(\E\ot_KF)\lrarrow\N\ot_k\E\ot_KF\lrarrow
 \N\ot_k\C\ot_k\E\ot_KF.
\end{equation}
 Taking the tensor product
of~\eqref{cotensor-product-left-exact-sequence} with the left
$K$\+module $F$, we produce a three-term complex of vector spaces
\begin{equation} \label{tensor-of-cotensor-complex}
 (\N\oc_\C\E)\ot_KF\lrarrow\N\ot_k\E\ot_KF\lrarrow
 \N\ot_k\C\ot_k\E\ot_KF.
\end{equation}
 Comparing~\eqref{tensor-of-cotensor-complex}
with~\eqref{cotensor-product-of-tensor-products-sequence},
we immediately obtain the desired natural
map~\eqref{tensor-cotensor-assoc-eqn}.

 Furthermore, it is clear that the map~\eqref{tensor-cotensor-assoc-eqn}
is an isomorphism if and only if
the complex~\eqref{tensor-of-cotensor-complex} is a left exact sequence.
 Assume that $\E$ is a t\+unital right $K$\+module.
 Then, by~\cite[Proposition~4.2]{Quil} or~\cite[Lemma~1.5]{Pnun}, so
are the right $K$\+modules $\N\ot_k\E$ and $\N\ot_k\C\ot_k\E$.
 By the definition of nt\+flatness (see
Section~\ref{nt-nc-flat-proj-prelim-subsecn}), \,$F$ being an nt\+flat
left $K$\+module then implies left exactness of
the complex~\eqref{tensor-of-cotensor-complex}.
\end{proof}

 We refer to~\cite[end of Section~1.2.5]{Psemi} for a relevant
discussion of \emph{commutativity of the diagrams of associativity
isomorphisms} of cotensor products.

 The next proposition is a version
of~\cite[Propositions~3.2.3.2(a)]{Psemi}.

\begin{prop} \label{tensor-hom-cohom-assoc-prop}
 Let $K$ be a nonunital algebra, $\C$ be a coalgebra, $\fP$ be
a left\/ $\C$\+contramodule, $F$ be a left $K$\+module, and\/ $\E$ be
a left\/ $\C$\+comodule endowed with a right action of the algebra $K$
by left\/ $\C$\+comodule endomorphisms.
 Then there is a natural map of vector spaces
\begin{equation} \label{tensor-hom-cohom-assoc-eqn}
 \Cohom_\C(\E\ot_KF,\>\fP)\lrarrow\Hom_K(F,\Cohom_\C(\E,\fP)).
\end{equation}
 The map~\eqref{tensor-hom-cohom-assoc-eqn} is an isomorphism whenever\/
$\E$ is a t\+unital right $K$\+module and $F$ is an nc\+projective
left $K$\+module.
\end{prop}

\begin{proof}
 By the definition of cohomomorphisms, we have a right exact sequence
of left $K$\+modules
\begin{equation} \label{cohom-right-exact-sequence}
 \Hom_k(\C\ot_k\E,\>\fP)\lrarrow\Hom_k(\E,\fP)\lrarrow
 \Cohom_\C(\E,\fP)\lrarrow0
\end{equation}
and a similar right exact sequence of vector spaces
\begin{equation} \label{cohom-of-tensor-products-sequence}
 \Hom_k(\C\ot_k\E\ot_KF,\>\fP)\rarrow\Hom_k(\E\ot_KF,\>\fP)\rarrow
 \Cohom_\C(\E\ot_KF,\>\fP)\rarrow0.
\end{equation}
 Applying to~\eqref{cohom-right-exact-sequence} the covariant functor
$\Hom_K(F,{-})$, we produce a three-term complex of vector spaces
\begin{equation} \label{hom-into-cohom-complex}
 \Hom_k(\C\ot_k\E\ot_KF,\>\fP)\rarrow\Hom_k(\E\ot_KF,\>\fP)\rarrow
 \Hom_K(F,\Cohom_\C(\E,\fP)).
\end{equation}
 Comparing~\eqref{hom-into-cohom-complex}
with~\eqref{cohom-of-tensor-products-sequence}, we immediately obtain
the desired natural map~\eqref{tensor-hom-cohom-assoc-eqn}.

 Furthermore, it is clear that
the map~\eqref{tensor-hom-cohom-assoc-eqn} is an isomorphism if and only
if the complex~\eqref{hom-into-cohom-complex} is a right exact sequence.
 Assume that $\E$ is a t\+unital right $K$\+module.
 Then, by Lemma~\ref{duality-preserves-reflects-t-c},
the left $K$\+modules $\Hom_k(\E,\fP)$ and $\Hom_k(\C\ot_k\E,\>\fP)$
are c\+unital.
 By the definition of nc\+projectivity (see
Section~\ref{nt-nc-flat-proj-prelim-subsecn}), \,$F$ being
an nc\+projective left $K$\+module then implies right exactness of
the complex~\eqref{hom-into-cohom-complex}.
\end{proof}

 We refer to~\cite[end of Section~3.2.5]{Psemi} for a relevant
discussion of \emph{commutativity of the diagrams of associativity
isomorphisms} of iterated cohomomorphisms.

\subsection{Left flat, right integrated bimodules}
\label{integrated-bimodules-subsecn}
 Let $K$ be nonunital $k$\+algebra, $\C$ be a coalgebra over~$k$,
and $\phi\:\C\times K\rarrow k$ be a multiplicative pairing.
 Recall that the pairing~$\phi$ endows the $\C$\+$\C$\+bicomodule
$\C$ with the induced $K$\+$K$\+bimodule structure (see
Proposition~\ref{underlying-K-module-structures}).

 Let $B\in K\nBimodn K$ be a $K$\+$K$\+bimodule.
 Then the tensor product $\C\ot_KB$ has a natural pair of commuting
structures of a left $\C$\+comodule and a right $K$\+module, or in
other words, $\C\ot_KB$ is a left $\C$\+comodule endowed with
a right action of $K$ by left $\C$\+comodule endomorphisms.

 We will say that the $K$\+$K$\+bimodule $B$ is \emph{right integrated}
if the tensor product $\C\ot_KB$ is endowed with a right
$\C$\+comodule structure such that
\begin{enumerate}
\renewcommand{\theenumi}{\roman{enumi}}
\item the right $\C$\+comodule structure on $\C\ot_KB$ commutes with
the natural left $\C$\+comodule structure (arising from the left
$\C$\+comodule structure on~$\C$), so $\C\ot_KB$ is
a $\C$\+$\C$\+bicomodule;
\item applying the functor~$\phi_*$ \,\eqref{right-comodules-to-modules}
to the right $\C$\+comodule structure on $\C\ot_KB$ produces
the natural right $K$\+module structure (arising from the right
$K$\+module structure on~$B$).
\end{enumerate}

 In other words, a \emph{right integrated $K$\+$K$\+bimodule} is
a triple $B=(B,\B,\beta)$, where $B$ is a $K$\+$K$\+bimodule, $\B$ is
a $\C$\+$\C$\+bicomodule, and $\beta\:\C\ot_KB\simeq\B$ is
an isomorphism of left $\C$\+comodules which is also an isomorphism
of right $K$\+modules.

\begin{rem}
 Our terminology ``integrated bimodule'' comes from Lie theory.
 Let $G$ be an algebraic group over a field~$k$ and $\fg$~be
the Lie algebra of~$G$.
 Let $\C(G)$ be the coalgebra of regular functions on $G$ (with respect
to the convolution comultiplication) and $U(\fg)$ be the enveloping
algebra of~$\fg$.

 Then there is a natural multiplicative pairing
$\phi\:\C(G)\times U(\fg)\rarrow k$ respecting the unit of $U(\fg)$
and the counit of~$\C(G)$ \,\cite[Section~2.7]{Prev}.
 This pairing is responsible, as per the construction of
Proposition~\ref{underlying-K-module-structures}, for the Lie functor
assigning the underlying $\fg$\+module structure to every
$G$\+module (notice that by a \emph{$G$\+module}, or
a \emph{representation of~$G$}, one simply means a $\C(G)$\+comodule).
 The same proposition also assigns the underlying $\fg$\+module
structure to every $\C(G)$\+contramodule.

 One speaks colloquially of the process of recovering a Lie/algebraic
group from a Lie algebra, or an algebraic group action from a Lie
algebra action, as of ``integration''.
 Hence the word ``integrated'' in our terminology.
\end{rem}

 Assume that $K$ is a t\+unital algebra and $\phi$~is a right
t\+unital multiplicative pairing.
 Then we have a natural (multiplication) isomorphism of left
$\C$\+comodules and right $K$\+modules $\varkappa\:\C\ot_KK\rarrow\C$.
 Consequently, the $K$\+$K$\+bimodule $K$ can be endowed with
the natural structure of right integrated bimodule given by
the triple $(K,\C,\varkappa)$.
 Here the coalgebra $\C$ is endowed with its natural
$\C$\+$\C$\+bicomodule structure.

\begin{prop} \label{tensor-product-integrated}
 Let $K$ be a nonunital algebra, $\C$ be a coalgebra, and
$\phi\:\C\times K\rarrow k$ be a right t\+unital multiplicative pairing.
 Let $B'=(B',\B',\beta')$ and $B''=(B'',\B'',\beta'')$ be two
right integrated $K$\+$K$\+bimodules.
 Assume that $B''$ is an nt\+flat left $K$\+module.
 Then the tensor product $B=B'\ot_KB''$ acquires a natural structure of
right integrated $K$\+$K$\+bimodule, $B=(B,\B,\beta)$.
\end{prop}

\begin{proof}
 Applying Proposition~\ref{tensor-cotensor-assoc-prop} to
the right $\C$\+comodule $\N=\B'$, the $\C$\+$\C$\+bicomodule $\E=\C$,
and the left $K$\+module $F=B''$, we obtain a natural associativity
isomorphism
$$
 \B'\ot_KB''=(\B'\oc_\C\C)\ot_KB''\simeq\B'\oc_\C(\C\ot_KB''),
$$
that is,
\begin{equation} \label{integration-isomorphism}
 \C\ot_KB'\ot_KB''\simeq\B'\oc_\C\B''.
\end{equation}
 It remains to put $\B=\B'\oc_\C\B''$ and let $\beta$~be
the isomorphism~\eqref{integration-isomorphism}.
\end{proof}

 We will say that a right integrated $K$\+$K$\+bimodule $(B,\B,\beta)$
is \emph{left nt\+flat} if the left $K$\+module $B$ is nt\+flat.
 Similarly, $(B,\B,\beta)$ is called \emph{right t\+unital} if
the right $K$\+module $B$ is t\+unital.

\begin{prop} \label{nonunital-tensor-of-nt-flat-integrated}
  Let $K$ be a nonunital algebra, $\C$ be a coalgebra, and
$\phi\:\C\times K\rarrow k$ be a right t\+unital multiplicative pairing.
 Then the category of left nt\+flat, right t\+unital right integrated
$K$\+$K$\+bimodules $(B,\B,\beta)$ is an associative but
\emph{nonunital} tensor/monoidal category with respect to the operation
of tensor product of bimodules and cotensor product of bicomodules.
\end{prop}

\begin{proof}
 The point is that for any two left nt\+flat, right t\+unital
$K$\+$K$\+bimodules $B'$ and $B''$, the tensor product $B'\ot_KB''$
is also left nt\+flat and right t\+unital (by
Lemma~\ref{nt-flat-tensor-product} and~\cite[Lemma~1.2(b)]{Pnun}).
 So the assertion follows from
Proposition~\ref{tensor-product-integrated}.
\end{proof}

 The problem with nonunitality of the tensor/monoidal category in
Proposition~\ref{nonunital-tensor-of-nt-flat-integrated} arises from
the fact that the left $K$\+module $K$ is not nt\+flat in general (while
the right $K$\+module $\widecheck K$ is, of course, not t\+unital).
 Indeed, as explained in Section~\ref{nt-nc-flat-proj-prelim-subsecn},
the left $K$\+module $K$ is nt\+flat if and only if $K$ is a flat
left $\widecheck K$\+module.

 We will say that a right integrated $K$\+$K$\+bimodule $(B,\B,\beta)$
is \emph{left flat} if the left $\widecheck K$\+module $B$ is flat.
 Similarly, $(B,\B,\beta)$ is called \emph{t\+unital} if
the $K$\+$K$\+bimodule $B$ is t\+unital (i.~e., both left and right
t\+unital).

\begin{cor} \label{unital-monoidal-of-flat-integrated}
 Let $K$ be t\+unital algebra such that $K$ is a flat left
$\widecheck K$\+module.
 Let $\C$ be a coalgebra and $\phi\:\C\times K\rarrow k$ be
a right t\+unital multiplicative pairing.
 Then the category of t\+unital, left flat right integrated
$K$\+$K$\+bimodules $(B,\B,\beta)$ is an asssociative, \emph{unital}
monoidal category with the unit object $(K,\C,\varkappa)$ with respect
to the operation of tensor product of bimodules and cotensor product
of bicomodules.
\end{cor}

\begin{proof}
 It was mentioned in Section~\ref{nt-nc-flat-proj-prelim-subsecn}
that all flat $\widecheck K$\+modules are nt\+flat.
 By Proposition~\ref{t-c-unital-monoidal-model-categories}(a),
the category of t\+unital $K$\+$K$\+bimodules $K\tBimodt K$ is
an associative, unital monoidal category with the unit object $K$
with respect to the tensor product operation~$\ot_K$.
 As explained in Section~\ref{semialgebras-prelim-subsecn},
the category of $\C$\+$\C$\+bicomodules $\C\Bicomod\C$ is
an associative, unital monoidal category with the unit object $\C$
with respect to the cotensor product operation~$\oc_\C$.
 It remains to refer to Proposition~\ref{tensor-product-integrated}
or Proposition~\ref{nonunital-tensor-of-nt-flat-integrated}.
\end{proof}

\subsection{Two forgetful monoidal functors}
\label{two-forgetful-monoidal-functors-subsecn}
 We will simply say that a bimodule over a nonunital algebra $K$ is
\emph{left flat} if it is flat as a left $\widecheck K$\+module.
 In particular, $K$ itself is \emph{left flat} if $K$ is flat as
a left $\widecheck K$\+module.

 Notice that a left flat $k$\+algebra $K$ is t\+unital if and only
if $K^2=K$ \,\cite[Definition~2.3]{Quil}, \cite[Remark~1.6]{Pnun}.

\begin{rem} \label{lands-in-t-unital-remark}
 Let $K$ be a left flat $k$\+algebra, $\C$ be a coalgebra over~$k$,
and $\phi\:\C\times K\rarrow k$ be a right t\+unital pairing.
 Then Proposition~\ref{t-unital-pairing-comodules-prop}(2) tells
that the essential image of the functor $\phi_*\:\Comodr\C\rarrow
\Modrn K$ \,\eqref{right-comodules-to-modules} is contained in
the full subcategory of t\+unital modules $\Modrt K\subset\Modrn K$.
 So we can (and will) write $\phi_*\:\Comodr\C\rarrow\Modrt K$.

 Obviously, both the functors $\Comodr\C\rarrow\Modrn K$ and
$\Comodr\C\rarrow\Modrt K$ are always faithful; and one of them
is fully faithful if and only if the other one is.
\end{rem}

 Given a left flat t\+unital algebra $K$, a coalgebra $\C$, and
a right t\+unital pairing~$\phi$, the construction of
Corollary~\ref{unital-monoidal-of-flat-integrated} produces
an associative, unital monoidal category of left flat, right
integrated, t\+unital $K$\+$K$\+bimodules, which we will denote by
$$
 K\flattBimodtint K.
$$

 Let us also denote by $K\flattBimodt K\subset K\tBimodt K$ the full
subcategory of left flat t\+unital $K$\+$K$\+bimodules.
 Clearly, $K\flattBimodt K$ is a monoidal subcategory in $K\tBimodt K$.
 When the algebra $K$ is left flat, the unit object $K\in K\tBimodt K$
belongs to $K\flattBimodt K$.

 Similarly, let us denote by $\C\injBicomod\C\subset\C\Bicomod\C$
the full subcategory of left injective $\C$\+$\C$\+bicomodules (i.~e.,
injective as left $\C$\+comodules).
 By Lemma~\ref{cotensor-product-injective}, $\C\injBicomod\C$ is
a monoidal subcategory in $\C\Bicomod\C$ containing the unit object
$\C\in\C\Bicomod\C$.

\begin{lem} \label{tensor-product-with-flat-is-injective}
 Let $K$ be a left flat t\+unital algebra, $\C$ be a coalgebra, $F$ be
an nt\+flat (equivalently, t\+flat) left $K$\+module, and\/ $\E$ be
an injective left\/ $\C$\+comodule endowed with a t\+unital right
action of the algebra $K$ by left\/ $\C$\+comodule endomorphisms.
 Then the tensor product $\E\ot_KF$ is an injective left\/
$\C$\+comodule.
\end{lem}

\begin{proof}
 Following the discussion in
Section~\ref{adjusted-co-contra-mods-subsecn}, it suffices to show
that the functor $\N\longmapsto\N\oc_\C(\E\ot_KF)\:\Comodr\C\rarrow
k\Vect$ is exact.
 By Proposition~\ref{tensor-cotensor-assoc-prop}, the latter functor
is isomorphic to the functor $\N\longmapsto(\N\oc_\C\E)\ot_KF$.
 Finally, if $K$ is a left flat t\+unital algebra, then the full
subcategory $\Modrt K$ is closed under kernels in $\Modrn K$
by Proposition~\ref{t-unital-kernels-flat-ring}, and it follows that
$\N\oc_\C\E$ is a t\+unital right $K$\+module for every
right $\C$\+comodule~$\N$.

 Alternatively, one can notice the natural isomorphisms
$\E\ot_KF\simeq(\E\ot_KK)\ot_KF\simeq\E\ot_K(K\ot_KF)$.
 The left $K$\+module $K\ot_KF$ is flat by~\cite[Proposition~8.16
and Corollary~8.20]{Pnun}.
 This reduces the question to the case when $F$ is a flat left
$K$\+module.
 When the left $K$\+module $F$ is flat, the assumption that $\E$ is
a t\+unital right $K$\+module is no longer needed for the validity
of the lemma.
\end{proof}

 Let $K$ be a left flat t\+unital algebra, $\C$ be a coalgebra, and
$\phi\:\C\times K\rarrow k$ be a right t\+unital pairing.
 By Corollary~\ref{unital-monoidal-of-flat-integrated} and
Lemma~\ref{tensor-product-with-flat-is-injective} (for $\E=\C$),
we have two associative, unital monoidal forgetful functors of 
associative, unital monoidal categories, depicted in the upper part
of the diagram below.
 We also have two associative, unital monoidal fully faithful inclusion
functors of such monoidal categories, depicted in the lower part of
the diagram
\begin{equation} \label{two-monoidal-forgetful-functors-diagram}
\begin{gathered}
 \xymatrix{
  & K\flattBimodtint K \ar[ld] \ar[rd] \\
  K\flattBimodt K \ar@{>->}[d]
  && \C\injBicomod\C \ar@{>->}[d] \\
  K\tBimodt K && \C\Bicomod\C
 }
\end{gathered}
\end{equation}
 The leftmost diagonal functor assigns to an integrated bimodule
$(B,\B,\beta)$ its underlying $K$\+$K$\+bimodule~$B$.
 The rightmost diagonal functor assigns to $(B,\B,\beta)$
the $\C$\+$\C$\+bicomodule~$\B$.

 Since the functor $\phi_*\:\Comodr\C\rarrow\Modrt K\subset\Modrn K$
(cf.\ Remark~\ref{lands-in-t-unital-remark}) is faithful, the leftmost
diagonal functor on the diagram is always faithful.

\subsection{The fully faithful comodule inclusion case}
\label{fully-faithful-comodule-inclusion-secn}
 An important special case of the theory developed in
Section~\ref{integrated-bimodules-subsecn} occurs when the functor
$\phi_*\:\Comodr\C\rarrow\Modrn K$ \,\eqref{right-comodules-to-modules}
from Proposition~\ref{underlying-K-module-structures} is fully faithful.

 Given an algebra $K$, a coalgebra $\C$, and a multiplicative pairing
$\phi\:\C\times K\rarrow k$, one can extend~$\phi$ to a multiplicative
pairing $\check\phi\:\C\times\widecheck K\rarrow k$ by the obvious
rule $\check\phi(c,\>a+r)=a\epsilon_\C(c)+\phi(c,r)$ for all
$c\in\C$ and $a+r\in k\oplus K=\widecheck K$.
 Then the pairing~$\check\phi$ is compatible with the counit of $\C$
and the unit of $\widecheck K$, in the sense that $\check\phi(c,1)=
\epsilon_\C(c)$ for all $c\in\C$.
 Hence the induced unital $k$\+algebra homomorphism $\check\phi^*\:
\widecheck K\rarrow\C^*$.

 Notice that, for every subcoalgebra $\E\subset\C$, every $\E$\+comodule
can be considered as a $\C$\+comodule~\cite[Section~3.1]{Pksurv};
so we have natural fully faithful functors $\E\Comodl\rarrow\C\Comodl$
and $\Comodr\E\rarrow\Comodr\C$.

\begin{prop} \label{comodule-forgetful-full-and-faithfulness}
 Let $K$ be a nonunital $k$\+algebra, $\C$ be a coalgebra over~$k$,
and $\phi\:\C\times K\rarrow k$ be a multiplicative pairing.
 Consider the following conditions:
\begin{enumerate}
\item the functor $\phi_*\:\C\Comodl\rarrow K\Modln$
\,\eqref{left-comodules-to-modules} is fully faithful;
\item the functor $\phi_*\:\Comodr\C\rarrow\Modrn K$
\,\eqref{right-comodules-to-modules} is fully faithful;
\item for every finite-dimensional subcoalgebra\/ $\E\subset\C$,
the composition of functors\/ $\E\Comodl\rarrow\C\Comodl\rarrow
K\Modln$ is fully faithful;
\item for every finite-dimensional subcoalgebra\/ $\E\subset\C$,
the composition of functors\/ $\Comodr\E\rarrow\Comodr\C\rarrow
\Modrn K$ is fully faithful;
\item the pairing~$\check\phi$ is nondegenerate in\/~$\C$, i.~e., for
every $c\in\C$ there exists $\check r\in\widecheck K$ such that
$\check\phi(c,\check r)\ne0$;
\item the image of the $k$\+algebra homomorphism
$\check\phi^*\:\widecheck K\rarrow\C^*$ is dense in the natural
pro-finite-dimensional (linearly compact) topology on the vector
space/algebra\/~$\C^*$; \hbadness=1250
\item the pairing~$\phi$ is nondegenerate in\/~$\C$, i.~e., for
every $c\in\C$ there exists $r\in K$ such that $\phi(c,r)\ne0$;
\item the image of the $k$\+algebra homomorphism $\phi^*\:K\rarrow\C^*$
is dense in the natural pro-finite-dimensional topology on\/~$\C^*$.
\end{enumerate}
 Then the following implications and equivalences hold:
$$
 (1)\ \Longleftrightarrow\ (2)\ \Longleftrightarrow\ (3)
 \ \Longleftrightarrow\ (4)\ \Longleftarrow\ (5)
 \ \Longleftrightarrow\ (6)\ \Longleftarrow\ (7)
 \ \Longleftrightarrow\ (8).
$$
\end{prop}

\begin{proof}
 The point is that a monomorphism (in the category of) coassociative
coalgebras need not be an injective map~\cite[Example in
Section~3.2]{NT} (see also~\cite{Ag}).
 For this reason, the conditions~(5\+-6) are stronger than~(1\+-4).

 (1)\,$\Longrightarrow$\,(3) and (2)\,$\Longrightarrow$\,(4)
 These implications hold because the functors
$\E\Comodl\allowbreak\rarrow\C\Comodl$ and $\Comodr\E\rarrow\Comodr\C$
are fully faithful.

 (3)\,$\Longrightarrow$\,(1) and (4)\,$\Longrightarrow$\,(2)
 Any coassociative coalgebra $\C$ is the union of its finite-dimensional
subcoalgebras $\E$, and any $\C$\+comodule is the union of its
finite-dimensional subcomodules.
 Furthermore, any finite-dimensional $\C$\+comodule is a comodule over
a finite-dimensional subcoalgebra of~$\C$ \,\cite[Lemma~3.1]{Pksurv}.
 The desired implications follow from these observations.

 (3)\,$\Longleftrightarrow$\,(4)
 The multiplicative, unital pairing $\check\phi\:\C\times\widecheck K
\rarrow k$ induces functors $\check\phi_*\:\C\Comodl\rarrow
\widecheck K\Modl=K\Modln$ and $\check\phi_*\:\Comodr\C\rarrow\Modr
\widecheck K=\Modrn K$, which can be identified with
the functors~$\phi_*$.
 This reduces the question to the case of a unital algebra and
a unital pairing.

 Compose the unital $k$\+algebra homomorphism $\check\phi^*\:
\widecheck K\rarrow\C^*$ with the surjective unital $k$\+algebra
homomorphism $\C^*\rarrow\E^*$ dual to the inclusion $\E\rarrow\C$.
 Denote by $L$ the image of the composition $\widecheck K\rarrow\C^*
\rarrow\E^*$; so $L$ is a finite-dimensional unital $k$\+algebra.
 Let $\LL=L^*$ be the coalgebra dual to~$L$.
 The we have an injective homomorphism of finite-dimensional
unital $k$\+algebras $L\rarrow\E^*$ and the dual surjective
homomorphism of finite-dimensional coalgebras $\E\rarrow\LL$.

 It is clear that the functor $\E\Comodl\rarrow\widecheck K\Modl$
is fully faithful if and only if the functor $\E\Comodl\rarrow
\LL\Comodl$ induced by the coalgebra map $\E\rarrow\LL$
(see~\cite[Section~1]{Pksurv}) is fully faithful.
 Similarly, the functor $\Comodr\E\rarrow\Modr\widecheck K$ is fully
faithful if and only if the functor $\Comodr\E\rarrow\Comodr\LL$
is fully faithful.

 Finally, any one of the functors $\E\Comodl\rarrow\LL\Comodl$ and/or
$\Comodr\E\rarrow\Comodr\LL$ is fully faithful if and only if $\E
\rarrow\LL$ is a monomorphism in the category of (finite-dimensional)
coalgebras, which is a left-right symmetric
property~\cite[Proposition~XI.1.2]{Sten},
\cite[Proposition~3.2 or Theorem~3.5]{NT}.

 (6)\,$\Longrightarrow$\,(3) or (6)\,$\Longrightarrow$\,(4)
 The map $\check\phi^*\:\widecheck K\rarrow\C^*$ has dense image if
and only if, in the context of the discussion in the previous
paragraphs, for every finite-dimensional subcoalgebra $\E\subset\C$,
the composition $\widecheck K\rarrow\C^*\rarrow\E^*$ is surjective.
 In this case, the coalgebra map $\E\rarrow\LL$ is an isomorphism.
 Then the induced functors $\E\Comodl\rarrow\LL\Comodl$ and
$\Comodr\E\rarrow\Comodr\LL$ are category equivalences.

 (5)\,$\Longleftrightarrow$\,(6) and (7)\,$\Longleftrightarrow$\,(8)
 Straightforward.
 
 (7)\,$\Longrightarrow$\,(5) and (8)\,$\Longrightarrow$\,(6) Obvious.
\end{proof}

 Returning to the context of Section~\ref{integrated-bimodules-subsecn},
we observe that if the functor $\phi_*\:\Comodr\C\rarrow\Modrn K$
is fully faithful, then so is the functor assigning to a right
integrated bimodule $(B,\B,\beta)$ its underlying
$K$\+$K$\+bimodule~$B$.
 In particular, on
the diagram~\eqref{two-monoidal-forgetful-functors-diagram}
from Section~\ref{two-forgetful-monoidal-functors-subsecn},
the leftmost diagonal functor $K\flattBimodtint K\rarrow
K\flattBimodt K$ is fully faithful in this case.
{\hbadness=1500\par}

 In this special case, being right integrated becomes a \emph{property}
of a $K$\+$K$\+bimodule rather than an additional structure.
 So we will sometimes change our terminology and speak of
\emph{right integrable} instead of ``right integrated'' bimodules
when the functor $\phi_*\:\Comodr\C\rarrow\Modrn K$ is known
to be fully faithful.

 Moreover, if the functor $\phi_*\:\Comodr\C\rarrow\Modrn K$ is
fully faithful, then a right $\C$\+comodule structure on
a $k$\+vector space $\B$ commutes with the given left $\C$\+comodule
structure if and only if the underlying right $K$\+module structure
commutes with that given left $\C$\+comodule structure.
 So condition~(i) from Section~\ref{integrated-bimodules-subsecn}
\emph{need not} be imposed in the definition of a right integrable
bimodule~$B$; it holds automatically (provided that condition~(ii)
holds).

 Simply put: assuming the functor $\phi_*\:\Comodr\C\rarrow\Modrn K$
to be fully faithful, a $K$\+$K$\+bimodule $B$ is called \emph{right
integrable} if the right $K$\+module structure on the tensor
product $\C\ot_KB$ arises from some right $\C$\+comodule structure.

\Section{Description of Right Semimodules and Left Semicontramodules}
\label{description-of-semimod-semicontra-secn}

\subsection{Construction of semialgebra}
\label{construction-of-semialgebra-subsecn}
 A homomorphism of $k$\+algebras $f\:K\rarrow R$ is said to be
\emph{left t\+unital}~\cite[Section~9]{Pnun} if $f$~makes $R$
a t\+unital left $K$\+module.
 By~\cite[Proposition~9.5]{Pnun}, it then follows that the algebra $R$
is t\+unital (but we will not use this fact).
 Similarly, $f$~is said to be \emph{right t\+unital} if $R$ is
a t\+unital right $K$\+module; and $f$~is said to be \emph{t\+unital}
if it is both left and right t\+unital.
 In other words, $f$~is t\+unital if and only if $R$ is a t\+unital
$K$\+$K$\+bimodule.

 Assume that the $k$\+algebra $K$ is t\+unital.
 Then the datum of a t\+unital homomorphism of $k$\+algebras
$f\:K\rarrow R$ is equivalent to the datum of a unital monoid
object $R$ in the unital monoidal category $K\tBimodt K$.
 
 Let $f\:K\rarrow R$ be a t\+unital homomorphism of $k$\+algebras.
 Assume further that $K$ is a left flat algebra and $R$ is a left
flat $K$\+$K$\+bimodule (as defined in
Section~\ref{two-forgetful-monoidal-functors-subsecn}).
 So $R$ is a unital monoid object in the unital monoidal category
$K\flattBimodt K$.

 Finally, let $\phi\:\C\times K\rarrow k$ be a right t\+unital
multiplicative pairing (in the sense of
Sections~\ref{multiplicative-pairings-subsecn}\+-%
\ref{t-unital-pairings-subsecn}), and let $(R,\R,\rho)$ be a right
integrated bimodule structure on the $K$\+$K$\+bimodule~$R$
(in the sense of Section~\ref{integrated-bimodules-subsecn}).
 Then $(R,\R,\rho)$ is an object of the unital monoidal category
$K\flattBimodtint K$.

 We are interested in lifting the monoid structure on $R$ to
a monoid structure on $(R,\R,\rho)$, along the monoidal functor
$K\flattBimodtint K\rarrow K\flattBimodt K$ (the leftmost diagonal
functor on the diagram~\eqref{two-monoidal-forgetful-functors-diagram}
from Section~\ref{two-forgetful-monoidal-functors-subsecn}).
 The following proposition and corollary form our version
of~\cite[Section~10.2.1]{Psemi}.

\begin{prop} \label{lift-of-monoid-structure-prop}
 Let $K$ be a left flat t\+unital $k$\+algebra and
$\phi\:\C\times K\rarrow k$ be a right t\+unital multiplicative pairing.
 In this context, a monoid structure on an object $(R,\R,\rho)\in
K\flattBimodtint K$ is uniquely determined by the induced monoid
structure on the object $R\in K\flattBimodt K$.
 Given a t\+unital homomorphism of $K$\+algebras making $R$ a left
flat $K$\+$K$\+bimodule and a right integrated bimodule structure
$(R,\R,\rho)$ on $R$, the unital monoid structure on the object
$R\in K\flattBimodt K$ can be lifted to a unital monoid structure on
$(R,\R,\rho)\in K\flattBimodtint K$ if and only if the following
two conditions hold:
\begin{itemize}
\item the map\/ $\C=\C\ot_KK\rarrow\C\ot_KR=\R$ induced by the map
$f\:K\rarrow R$ is a right\/ $\C$\+comodule morphism;
\item the map\/ $\R\oc_\C\R=(\C\ot_KR)\oc_\C(\C\ot_KR)\simeq
\C\ot_KR\ot_KR\rarrow\C\ot_KR=\R$ induced by the multiplication
map $R\ot_KR\rarrow R$ is a right\/ $\C$\+comodule morphism.
 (Here the natural\/ $\C$\+$\C$\+bicomodule isomorphism
$(\C\ot_KR)\oc_\C(\C\ot_KR)\simeq\C\ot_KR\ot_KR$ is provided
by Proposition~\ref{tensor-cotensor-assoc-prop}.)
\end{itemize}
\end{prop}

\begin{proof}
 The point is that the forgetful functor $K\flattBimodtint K\rarrow
K\flattBimodt K$ taking $(R,\R,\rho)$ to $R$ is faithful.
 Any monoid structure on $(R,\R,\rho)$ is given by certain morphisms in
$K\flattBimodtint K$, and these morphisms are uniquely determined by
their images in $K\flattBimodt K$.
 Hence the first assertion of the proposition.

 Conversely, given a monoid structure on $R$, in order for it to be
liftable to a monoid structure on $(R,\R,\rho)$, one needs to check
that the structure morphisms of the given monoid structure in
$K\flattBimodt K$ come from some morphisms in
$K\flattBimodtint K$.
 These are the conditions listed in the second assertion of
the proposition.

 If the required morphisms exist in $K\flattBimodtint K$,
they will automatically satisfy the axioms of a monoid structure in
$K\flattBimodtint K$, because their images satisfy such axioms in
$K\flattBimodt K$ and the forgetful functor $K\flattBimodtint K
\rarrow K\flattBimodt K$ is faithful.
 Simply put, an equation on morphisms in $K\flattBimodtint K$
is satisfied whenever its image in $K\flattBimodt K$ is satisfied.
\end{proof}

\begin{cor} \label{fully-faithful-lift-of-monoid-cor}
 Let $K$ be a left flat t\+unital $k$\+algebra and $\phi\:\C\times
K\rarrow k$ be a right t\+unital multiplicative pairing such that
the functor $\phi_*\:\Comodr\C\rarrow\Modrt K$ is fully faithful.
 Let $f\:K\rarrow R$ be a t\+unital homomorphism of $K$\+algebras
making $R$ a left flat $K$\+$K$\+bimodule, and assume that $R$ is
a right integrable $K$\+$K$\+bimodule (as per the discussion in
the end of Section~\ref{fully-faithful-comodule-inclusion-secn}).
 Then there exists a unique (unital) monoid structure on the object
$(R,\R,\rho)\in K\flattBimodtint K$ lifting the monoid structure
on the object $R\in K\flattBimodt K$.
\end{cor}

\begin{proof}
 Under the full-and-faithfulness assumption of the corollary,
the functor $K\flattBimodtint K \rarrow K\flattBimodt K$ is
fully faithful, so the conditions of
Proposition~\ref{lift-of-monoid-structure-prop} are
satisfied automatically.

 In other words, the maps appearing in the conditions of
the proposition are always right $K$\+module maps.
 Under the full-and-faithfulness assumption, any right $K$\+module
map between right $\C$\+comodules is a right $\C$\+comodule map.
\end{proof}

 Suppose that we have managed to lift a monoid structure on
$R\in K\flattBimodt K$ to a monoid structure on
$(R,\R,\rho)\in K\flattBimodtint K$, as per
Proposition~\ref{lift-of-monoid-structure-prop} or
Corollary~\ref{fully-faithful-lift-of-monoid-cor}.
 Then, applying the monoidal functor $K\flattBimodtint K\rarrow
\C\injBicomod\C$ (the rightmost diagonal
functor on the diagram~\eqref{two-monoidal-forgetful-functors-diagram}
from Section~\ref{two-forgetful-monoidal-functors-subsecn}), we
obtain a unital monoid object $\bS=\R$ in the unital monoidal
category $\C\injBicomod\C$.
 In other words, we have produced a (semiassociative, semiunital)
\emph{semialgebra} $\bS$ over the coalgebra~$\C$, as defined
in Section~\ref{semialgebras-prelim-subsecn}.

 Explicitly, we have
$$
 \bS=\R=\C\ot_KR.
$$
 The left $\C$\+comodule structure on $\bS$ is induced by the left
$\C$\+comodule structure on~$\C$.
 The right $\C$\+comodule structure on $\bS$ is assumed given as
a part of the data: this is what is meant by $(R,\R,\rho)$ being
a right integrated $K$\+$K$\+bimodule.

 The semiunit map $\be\:\C\rarrow\bS$ is the composition
$$
 \C=\C\ot_KK\rarrow\C\ot_KR=\bS,
$$
as in the first condition of
Proposition~\ref{lift-of-monoid-structure-prop}.
 The semimultiplication map $\bm\:\bS\oc_\C\bS\rarrow\bS$ is
the composition
$$
 \bS\oc_\C\bS=(\C\ot_KR)\oc_\C(\C\ot_KR)\simeq
 \C\ot_KR\ot_KR\rarrow\C\ot_KR=\bS,
$$ 
as in the second condition of the proposition.

 Moreover, the semialgebra $\bS$ is an injective left $\C$\+comodule.
 By Propositions~\ref{comodule-categories-abelian}(b)
and~\ref{contramodule-category-abelian}, it follows that
the category of right $\bS$\+semimodules $\Simodr\bS$ and
the category of left $\bS$\+semicontramodules $\bS\Sicntr$ are
abelian (with exact forgetful functors $\Simodr\bS\rarrow\Comodr\C$
and $\bS\Sicntr\rarrow\C\Contra$).
 In the rest of Section~\ref{description-of-semimod-semicontra-secn},
our aim is to describe these abelian categories more explicitly.

\subsection{Functor between module categories of right comodules
and modules} \label{functor-module-categories-right-comods-subsecn}
 Let $K$ be a left flat t\+unital $k$\+algebra, $\C$ be a coalgebra
over~$k$, and $\phi\:\C\times K\rarrow k$ be a right t\+unital
multiplicative pairing.

 As mentioned in Section~\ref{semimodules-prelim-subsecn},
the category of right $\C$\+comodules $\Comodr\C$ is a right module
category over the monoidal category $\C\Bicomod\C$.
 Restricting the action of $\C\Bicomod\C$ in $\Comodr\C$ along
the monoidal functors
$$
 K\flattBimodtint K\lrarrow\C\injBicomod\C\lrarrow\C\Bicomod\C
$$
(the rightmost functors on
the diagram~\eqref{two-monoidal-forgetful-functors-diagram}),
we make $\Comodr\C$ a right module category over the monoidal
category $K\flattBimodtint K$.

 On the other hand, according to
Proposition~\ref{t-c-unital-monoidal-model-categories}(b),
the category of t\+unital right $K$\+modules $\Modrt K$ is
a right module category over the monoidal category $K\tBimodt K$.
 Restricting the action of $K\tBimodt K$ in $\Modrt K$ along
the monoidal functors
$$
 K\flattBimodtint K\lrarrow K\flattBimodt K\lrarrow K\tBimodt K
$$
(the leftmost functors on
the diagram~\eqref{two-monoidal-forgetful-functors-diagram}),
we make $\Modrt K$ a right module category of the monoidal
category $K\flattBimodtint K$.

\begin{prop} \label{functor-module-cats-right-comods-prop}
 The functor $\phi_*\:\Comodr\C\rarrow\Modrt K$ is an associative,
unital module functor between the associative, unital right module
categories\/ $\Comodr\C$ and\/ $\Modrt K$ over the monoidal category
$K\flattBimodtint K$.
\end{prop}

\begin{proof}
 For an explanation concerning the functor $\phi_*\:\Comodr\C\rarrow
\Modrt K$, see Remark~\ref{lands-in-t-unital-remark}.
 The point is that, for any object $(B,\B,\beta)\in K\flattBimodtint K$
and any right $\C$\+comodule $\N$, the natural isomorphism of
right $K$\+modules
$$
 \N\ot_K B=(\N\oc_\C\C)\ot_KB\simeq \N\oc_\C(\C\ot_KB)=\N\oc_\C\B
$$
holds by Proposition~\ref{tensor-cotensor-assoc-prop}.
\end{proof}

\subsection{Description of right semimodules}
 Let $K$ be a left flat t\+unital $k$\+algebra, $\C$ be a coalgebra
over~$k$, and $\phi\:\C\times K\rarrow k$ be a right t\+unital
multiplicative pairing.
 Let $f\:K\rarrow R$ be a t\+unital homomorphism of $k$\+algebras
making $R$ a left flat $K$\+$K$\+bimodule, and let $(R,\bS,\rho)$
be a monoid object in $K\flattBimodtint K$ lifting the monoid object
$R\in K\flattBimodt K$.
 So $\bS\in\C\injBicomod\C$ is a semialgebra over~$\C$.

 By the definition (see Section~\ref{semimodules-prelim-subsecn}),
a right semimodule $\bN$ over $\bS$ is a module object in
the right module category $\Comodr\C$ over the monoid object
$\bS$ in the monoidal category $\C\Bicomod\C$.
 In our context, when the monoid object $\bS$ in $\C\Bicomod\C$ comes
from a monoid object $(R,\bS,\rho)$ in $K\flattBimodtint K$,
the datum of a module object over $\bS$ in the right module category
$\Comodr\C$ over $\C\Bicomod\C$ is equivalent to the datum of
a module object over $(R,\bS,\rho)$ in the right module category
$\Comodr\C$ over the monoidal category $K\flattBimodtint K$.
 Here the definition of the structure of a right module category
over $K\flattBimodtint K$ on the category $\Comodr\C$ was explained
in the previous
Section~\ref{functor-module-categories-right-comods-subsecn}.

 According to Proposition~\ref{functor-module-cats-right-comods-prop},
we have a functor of right module categories $\phi_*\:\Comodr\C\rarrow
\Modrt K$ over the monoidal category $K\flattBimodtint K$.
 Applying this functor to a right semimodule $\bN$ over $\bS$, we
obtain a module object over $(R,\bS,\rho)$ in the right module
category $\Modrt K$ over $K\flattBimodtint K$.
 The datum of a module object in $\Modrt K$ over a monoid object
$(R,\bS,\rho)$ in $K\flattBimodtint K$ is equivalent to the datum of
a module object in $\Modrt K$ over the monoid object $R\in K\tBimodt K$.
 This means a right $R$\+module $N$ that is t\+unital as a right
$K$\+module.
 The latter condition is actually equivalent to $N$ being t\+unital
as a right $R$\+module.

\begin{rem}
 Let $f\:K\rarrow R$ be a right t\+unital homomorphism of
$k$\+algebras.
 Then a right $R$\+module is t\+unital if and only if it is t\+unital
as a right $K$\+module.
 This is the result of~\cite[Corollary~9.7(a)]{Pnun}.
\end{rem}

 Explicitly, let $\bN$ be a right $\bS$\+semimodule.
 Then, first of all, $\bN$ is a right $\C$\+comodule.
 The construction of the functor~$\phi_*$
\,\eqref{right-comodules-to-modules} endows $\bN$ with a (t\+unital)
right $K$\+module structure.
  Now the composition
\begin{equation} \label{right-action-underlying-right-semiaction}
 \bN\ot_KR=(\bN\oc_\C\C)\ot_KR\simeq\bN\oc_\C(\C\ot_KR)
 =\bN\oc_\C\bS\lrarrow\bN
\end{equation}
of the associativity isomorphism from
Proposition~\ref{tensor-cotensor-assoc-prop} and
the semiaction map $\bn\:\bN\oc_\C\bS\rarrow\bN$ provides
an action map $\bN\ot_KR\rarrow\bN$ making $\bN$ a right $R$\+module.
 The following proposition and corollary form our version
of~\cite[first paragraph of Section~10.2.2]{Psemi}.

\begin{prop} \label{right-semimodules-described-prop}
 The datum of a right\/ $\bS$\+semimodule\/ $\bN$ is equivalent to
the datum of a right\/ $\C$\+comodule\/ $\N$, a (t\+unital) right
$R$\+module $N$, and an isomorphism of right $K$\+modules
$\phi_*(\N)\simeq N$ satisfying the following condition.
 The map\/ $\N\oc_\C\bS\rarrow\N$ constructed by reverting
the formula~\eqref{right-action-underlying-right-semiaction},
$$
 \N\oc_\C\bS=\N\oc_\C(\C\ot_KR)\simeq\N\ot_KR
 \simeq N\ot_KR\rarrow N\simeq\N
$$
must be a right $\C$\+comodule morphism.
\end{prop}

\begin{proof}
 The logic here is similar to that in
Proposition~\ref{lift-of-monoid-structure-prop}.
 The point is that the functor of module categories
$\phi_*\:\Comodr\C\rarrow\Modrt K$ from
Proposition~\ref{functor-module-cats-right-comods-prop} is faithful.
 Consequently, a module structure on an object $\N\in\C\Comodl$
over a monoid object $(R,\bS,\rho)\in K\flattBimodtint K$ is
uniquely determined by the induced module structure over
$(R,\bS,\rho)$ on the object $N=\phi_*(\N)\in K\Modlt$.

 Conversely, given a t\+unital right $K$\+module $N$ with a module
structure over $(R,\bS,\rho)$ (which means an $R$\+module structure),
and given a right $\C$\+comodule $\N$ such that $\phi_*(\N)=N$,
in order for the $R$\+module structure on $N$ to be liftable to
an $\bS$\+semimodule structure on $\N$, one needs to check that
the structure morphism of the given module structure in $\Modrt K$
comes from some morphism in $\Comodr\C$.
 This is the condition stated in the proposition.

 If the required morphism exists in $\Comodr\C$, it will automatically
satisfy the axioms of a module structure in $\Comodr\C$, because
its image satisfies such axioms in $\Modrt K$.
 Simply put, an equation on morphisms in $\Comodr\C$ is satsified
whenever its image in $\Modrt K$ is satisfied.
\end{proof}

\begin{cor} \label{right-semimodules-described-fully-faithful-case}
 Assume that the functor $\phi_*\:\Comodr\C\rarrow\Modrt K$ is fully
faithful.
 Then the datum of a right\/ $\bS$\+semimodule\/ $\bN$ is equivalent
to the datum of a right\/ $\C$\+comodule\/ $\N$, a (t\+unital) right
$R$\+module $N$, and an isomorphism of right $K$\+modules
$\phi_*(\N)\simeq N$.
 In other words, this is the datum of a (t\+unital) right $R$\+module
$N$ whose underlying (t\+unital) right $K$\+module structure has
the property that it arises from some right\/ $\C$\+coomdule structure.
\end{cor}

\begin{proof}
 This is similar to Corollary~\ref{fully-faithful-lift-of-monoid-cor}.
 The map appearing in the condition of
Proposition~\ref{right-semimodules-described-prop} is always
a morphism of right $K$\+modules.
 Under the full-and-faithfulness assumption of the corollary,
any right $K$\+module morphism between right $\C$\+comodules is
a right $\C$\+comodule morphism.
\end{proof}

\subsection{Projectivity instead of flatness}
 We start with a version of
Proposition~\ref{nonunital-tensor-of-nt-flat-integrated} with
the nt\+flatness condition replaced by nc\+projectivity.
 We say that a right integrated $K$\+$K$\+bimodule $(B,\B,\beta)$
is \emph{left nc\+projective} if the left $K$\+module $B$ is
nc\+projective.

\begin{prop} \label{nonunital-tensor-of-nc-projective-integrated}
 Let $K$ be a nonunital algebra, $\C$ be a coalgebra, and
$\phi\:\C\times K\rarrow k$ be a right t\+unital multiplicative pairing.
 Then the category of left nc\+projective, right t\+unital right
integrated $K$\+$K$\+bimodules $(B,\B,\beta)$ is a \emph{nonunital}
tensor/monoidal full subcategory of the nonunital tensor/monoidal
category from Proposition~\ref{nonunital-tensor-of-nt-flat-integrated}.
\end{prop}

\begin{proof}
 Follows from Lemma~\ref{nc-projective-tensor-product}
and Proposition~\ref{nc-projective-are-nt-flat}.
\end{proof}

 We will say that a bimodule over a nonunital algebra $K$ is
\emph{left projective} if it is projective as a left
$\widecheck K$\+module.
 In particular, $K$ itself is \emph{left projective} if $K$ is
projective as a left $\widecheck K$\+module.
 A right integrated bimodule $(B,\B,\beta)$ is called \emph{left
projective} if the $K$\+$K$\+bimodule $B$ is left projective.

\begin{rem} \label{lands-in-c-unital-remark}
 Let $K$ be a left projective $k$\+algebra, $\C$ be a coalgebra
over~$k$, and $\phi\:\C\times K\rarrow k$ be a right t\+unital pairing.
 Then Proposition~\ref{t-unital-pairing-contramodules-prop}(3) tells
that the essential image of the functor $\phi_*\:\C\Contra\rarrow
K\Modln$ \,\eqref{left-contramodules-to-modules} is contained in
the full subcategory of c\+unital modules $K\Modlc\subset K\Modln$.
 So we can (and will) write $\phi_*\:\C\Contra\rarrow K\Modlc$.

 Obviously, both the functors $\C\Contra\rarrow K\Modln$ and
$\C\Contra\rarrow K\Modlc$ are always faithful; and one of them
is fully faithful if and only if the other one is.
\end{rem}

 We will denote by $K\projtBimodt K\subset K\flattBimodt K$ the full
subcategory of left projective $K$\+$K$\+bimodules.
 Clearly, $K\projtBimodt K$ is a monoidal subcategory in
$K\tBimodt K$.
 When the algebra $K$ is left projective, the unit object
$K\in K\tBimodt K$ belongs to $K\projtBimodt K$.

\begin{cor} \label{unital-monoidal-of-projective-integrated}
 Let $K$ be a left projective t\+unital algebra, $\C$ be a coalgebra,
and $\phi\:\C\times K\rarrow k$ be a right t\+unital multiplicative
pairing.
 Then the category of t\+unital, left projective right integrated
$K$\+$K$\+bimodules $(B,\B,\beta)$ is a monoidal full subcategory of
the \emph{unital} monoidal category from
Corollary~\ref{unital-monoidal-of-flat-integrated}, containing
the unit object $(K,\C,\varkappa)$.
\end{cor}

\begin{proof}
 The assertion is straightforward, so we only recall that any
projective module is nc\+projective, as mentioned in
Section~\ref{nt-nc-flat-proj-prelim-subsecn}.
\end{proof}

 We will denote the unital monoidal full subcategory defined in
Corollary~\ref{unital-monoidal-of-projective-integrated} by
$$
 K\projtBimodtint K\,\subset\, K\flattBimodtint K.
$$
 For any left projective t\+unital algebra $K$, coalgebra $\C$,
and a right t\+unital multiplicative pairing
$\phi\:\C\times K\rarrow k$, we have the following commutative diagram
extending the diagram~\eqref{two-monoidal-forgetful-functors-diagram}
from Section~\ref{two-forgetful-monoidal-functors-subsecn}:
\begin{equation} \label{two-monoidal-forgetful-functors-extended-diag}
\begin{gathered}
 \xymatrix{
  & K\projtBimodtint K \ar@{>->}[d] \ar[ld] \\
  K\projtBimodt K \ar@{>->}[d]
  & K\flattBimodtint K \ar[ld] \ar[rd] \\
  K\flattBimodt K \ar@{>->}[d]
  && \C\injBicomod\C \ar@{>->}[d] \\
  K\tBimodt K && \C\Bicomod\C
 }
\end{gathered}
\end{equation}
 All the functors are associative, unital monoidal functors of
associative, unital monoidal categories.
 The four vertical arrows denote fully faithful functors.

\subsection{Functor between module categories of left contramodules
and modules}  \label{functor-module-categories-left-contramods-subsecn}
 Let $K$ be a left projective t\+unital $k$\+algebra, $\C$ be
a coalgebra over~$k$, and $\phi\:\C\times K\rarrow k$ be a right
t\+unital multiplicative pairing.

 As mentioned in Section~\ref{semicontramodules-prelim-subsecn},
the opposite category $\C\Contra^\sop$ to the category of left
$\C$\+contramodules $\C\Contra$ is a right module category over
the monoidal category $\C\Bicomod\C$.
 Restricting the action of $\C\Bicomod\C$ in $\C\Contra^\sop$ along
the monoidal functors
$$
 K\projtBimodtint K\rarrow K\flattBimodtint K
 \rarrow\C\injBicomod\C\rarrow\C\Bicomod\C
$$
(the right-hand side of
the diagram~\eqref{two-monoidal-forgetful-functors-extended-diag}),
we make $\C\Contra^\sop$ a right module category over the monoidal
category $K\projtBimodtint K$.

 On the other hand, according to
Proposition~\ref{t-c-unital-monoidal-model-categories}(c),
the opposite category $K\Modlc^\sop$ to the category of c\+unital
left $K$\+modules $K\Modlc$ is a right module category over
the monoidal category $K\tBimodt K$.
 Restricting the action of $K\tBimodt K$ in $K\Modlc^\sop$ along
the monoidal functors
$$
 K\projtBimodtint K\rarrow K\projtBimodt K
 \rarrow K\flattBimodt K\rarrow K\tBimodt K
$$
(the left-hand side of
the diagram~\eqref{two-monoidal-forgetful-functors-extended-diag}),
we make $K\Modlc^\sop$ a right module category over the monoidal
category $K\projtBimodtint K$.

\begin{prop} \label{functor-module-cats-left-contramods-prop}
 The functor $\phi_*^\sop\:\C\Contra^\sop\rarrow K\Modlc^\sop$ is
an associative, unital module functor between the associative, unital
right module categories\/ $\C\Contra^\sop$ and\/ $K\Modlc^\sop$ over
the monoidal category $K\projtBimodtint K$.
\end{prop}

\begin{proof}
 The functor $\phi_*\:\C\Contra\rarrow K\Modlc$ was discussed in
Remark~\ref{lands-in-c-unital-remark}.
 The proposition holds because, for any object $(B,\B,\beta)\in
K\projtBimodtint K$ and any left $\C$\+contramodule $\fP$, there is
a natural isomorphism of left $K$\+modules
$$
 \Hom_K(B,\fP)=\Hom_K(B,\Cohom_\C(\C,\fP))\simeq
 \Cohom_\C(\C\ot_KB,\>\fP)=\Cohom_\C(\B,\fP)
$$
provided by Proposition~\ref{tensor-hom-cohom-assoc-prop}.
\end{proof}

\subsection{Description of left semicontramodules}
 Let $K$ be a left projective t\+unital $k$\+algebra, $\C$ be
a coalgebra over~$k$, and $\phi\:\C\times K\rarrow k$ be a right
t\+unital multiplicative pairing.
 Let $f\:K\rarrow R$ be a t\+unital homomorphism of $k$\+algebras
making $R$ a left projective $K$\+$K$\+bimodule, and let $(R,\bS,\rho)$
be a monoid object in $K\projtBimodtint K$ lifting the monoid object
$R\in K\projtBimodt K$.
 So $\bS\in\C\injBicomod\C$ is a semialgebra over~$\C$.

 By the definition (see Section~\ref{semicontramodules-prelim-subsecn}),
a left semicontramodule $\bP$ over $\bS$ is essentially (up to
a passage to the opposite category) a module object in the right
module category $\C\Contra^\sop$ over the monoid object $\bS$ in
the monoidal category $\C\Bicomod\C$.
 In our context, when the monoid object $\bS$ in $\C\Bicomod\C$ comes
from a monoid object $(R,\bS,\rho)$ in $K\projtBimodtint K$,
the datum of a module object over $\bS$ in the right module category
$\C\Contra^\sop$ over $\C\Bicomod\C$ is equivalent to the datum of
a module object over $(R,\bS,\rho)$ in the right module category
$\C\Contra^\sop$ over the monoidal category $K\projtBimodtint K$.
 Here the definition of the structure of a right module category
over $K\projtBimodtint K$ on the category $\C\Contra^\sop$ was
explained in the previous
Section~\ref{functor-module-categories-left-contramods-subsecn}.

 According to
Proposition~\ref{functor-module-cats-left-contramods-prop}, we have
a functor of right module categories
$\phi_*^\sop\:\C\Contra^\sop\rarrow K\Modlc^\sop$ over the monoidal
category $K\projtBimodtint K$.
 Applying this functor to a left semicontramodule $\bP$ over $\bS$,
we obtain a module object over $(R,\bS,\rho)$ in the right module
category $K\Modlc^\sop$ over $K\projtBimodtint K$.
 The datum of a module object in $K\Modlc^\sop$ over a monoid object
$(R,\bS,\rho)$ in $K\projtBimodtint K$ is equivalent to the datum of
a module object in $K\Modlc^\sop$ over the monoid object
$R\in K\tBimodt K$.
 This means a left $R$\+module $P$ that is c\+unital as a left
$K$\+module.
 The latter condition is actually equivalent to $P$ being c\+unital
as a left $R$\+module.

\begin{rem}
 Let $f\:K\rarrow R$ be a right t\+unital homomorphism of $k$\+algebras.
 Then a left $R$\+module is c\+unital if and only if it is c\+unital as
a left $K$\+module.
 This is the result of~\cite[Corollary~9.7(b)]{Pnun}.
\end{rem}

 Explicitly, let $\bP$ be a left $\bS$\+semicontramodule.
 Then, first of all, $\bP$ is a left $\C$\+contramodule.
 The construction of the functor~$\phi_*$
\,\eqref{left-contramodules-to-modules} endows $\bP$ with
a (c\+unital) left $K$\+module structure.
 Now the composition
\begin{multline} \label{left-action-underlying-left-semicontraaction}
 \bP\lrarrow\Cohom_\C(\bS,\bP)=\Cohom_\C(\C\ot_KR,\>\bP) \\
 \simeq\Hom_K(R,\Cohom_\C(\C,\bP))=\Hom_K(R,\bP)
\end{multline}
of the semicontraaction map $\bp\:\bP\rarrow\Cohom_\C(\bS,\bP)$ and
the associativity isomorphism from
Proposition~\ref{tensor-hom-cohom-assoc-prop} provides an action map
$\bP\rarrow\Hom_K(R,\bP)$ making $\bP$ a left $R$\+module.
 The following proposition is our version
of~\cite[second paragraph of Section~10.2.2]{Psemi}.

\begin{prop} \label{left-semicontramodules-described-prop}
 The datum of a left\/ $\bS$\+semicontramodule\/ $\bP$ is equivalent
to the datum of a left\/ $\C$\+contramodule $\fP$, a (c\+unital) left
$R$\+module $P$, and an isomorphism of left $K$\+modules $\phi_*(\fP)
\simeq P$ satisfying the following condition.
 The map\/ $\fP\rarrow\Cohom_\C(\bS,\fP)$ constructed by reverting
the formula~\eqref{left-action-underlying-left-semicontraaction},
\begin{multline*}
 \fP\simeq P\lrarrow\Hom_K(R,P)\simeq\Hom_K(R,\fP) \\
 \simeq\Cohom_\C(\C\ot_KR,\>\fP)=\Cohom_\C(\bS,\fP)
\end{multline*}
must be a left\/ $\C$\+contramodule morphism.
\end{prop}

\begin{proof}
 The proof is similar to that of
Proposition~\ref{right-semimodules-described-prop}.
 The point is that the functor $\phi_*\:\C\Contra\rarrow K\Modlc$
is always faithful.
\end{proof}

\begin{cor} \label{left-semicontramods-described-fully-faithful-case}
 Assume that the functor $\phi_*\:\C\Contra\rarrow K\Modlc$ is fully
faithful.
 Then the datum of a left\/ $\bS$\+semicontramodule\/ $\bP$ is
equivalent to the datum of a left\/ $\C$\+contramodule\/ $\fP$,
a (c\+unital) left $R$\+module $P$, and an isomorphism of left
$K$\+modules $\phi_*(\fP)\simeq P$.
 In other words, this is the datum of a (c\+unital) left $R$\+module $P$
whose underlying (c\+unital) left $K$\+module structure has the property
that it arises from some left\/ $\C$\+contramodule structure.
\end{cor}

\begin{proof}
 This is similar to
Corollary~\ref{right-semimodules-described-fully-faithful-case}.
 The map appearing in the condition of
Proposition~\ref{left-semicontramodules-described-prop} is always
a morphism of left $K$\+modules.
 Under the full-and-faithfulness assumption of the corollary,
any left $K$\+module morphism between left $\C$\+contramodules
is a left $\C$\+contramodule morphism.
\end{proof}

\begin{rem}
 The conditions under which the comodule forgetful functor
$\phi_*\:\Comodr\C\rarrow\Modrn K$ is fully faithful were discussed
in Proposition~\ref{comodule-forgetful-full-and-faithfulness}.
 The contramodule forgetful functor $\phi_*\:\C\Contra\rarrow
K\Modln$ being fully faithful is a more rare occurrence.
{\hbadness=1800\par}

 In particular, in the case when the dual $k$\+algebra to $\C$
plays the role of $K$, the comodule forgetful/inclusion functor
$\Comodr\C\rarrow\Modr\C^*$ is \emph{always} fully
faithful~\cite[Propositions~2.1.1\+-2.1.2 and Theorem~2.1.3(e)]{Swe}.
 This is \emph{not} true for the contramodule forgetful functor
$\C\Contra\rarrow\C^*\Modl$.
 A complete characterization is known in the particular case when
the coalgebra $\C$ is \emph{conilpotent}: the functor
$\C\Contra\rarrow\C^*\Modl$ is fully faithful if and only if
the coalgebra $\C$ is \emph{finitely cogenerated}.
 See~\cite[Theorem~2.1]{Psm} for the ``if'' implication
and~\cite[Example~7.2]{Phff} for the ``only if''.
 Using the full result of~\cite[Theorem~2.1]{Psm}, one can show that,
for a conilpotent coalgebra $\C$, the functor $\phi_*\:\C\Contra\rarrow
K\Modln$ is fully faithful if and only if $\C$ is finitely cogenerated
and the map $\check\phi^*\:\widecheck K\rarrow\C^*$ has dense image
in~$\C^*$.

 Not so much is known outside of the conilpotent case, but examples
of fully faithful contramodule forgetful functors $\phi_*\:\C\Contra
\rarrow K\Modlc$ are provided by~\cite[Theorem~3.7]{Plfin} or
Theorem~\ref{strictly-lower-contramodule-full-and-faithfulness}
below (use~\cite[Proposition~6.6]{Pnun} or
Proposition~\ref{functors-as-c-unital-modules} below to establish
a comparison between our present setting and the one in~\cite{Plfin}).
 These examples will be discussed in the next
Sections~\ref{loc-fin-subcategories-secn}\+-\ref{examples-secn}.
\end{rem}

\Section{Locally Finite Subcategories in $k$-Linear Categories}
\label{loc-fin-subcategories-secn}

 As above, in this section $k$~denotes a fixed ground field.
 We build upon the expositions in the preprints~\cite{Plfin}
and~\cite[Section~6]{Pnun}.

\subsection{Recollections on $k$-linear categories}
 A \emph{$k$\+linear category} $\sE$ is a category enriched in
$k$\+vector spaces.
 In other words, for every pair of objects $x$, $y\in\sE$,
a $k$\+vector space $\sE_{y,x}=\Hom_\sE(x,y)$ is given, and for
every triple of objects $x$, $y$, $z\in\sE$, a $k$\+linear
multiplication/composition map
$$
 \Hom_\sE(y,z)\ot_k\Hom_\sE(x,y)\lrarrow\Hom_\sE(x,z)
$$
is defined.
 Identity morphisms $\id_x\in\Hom_\sE(x,x)$ must exist for all
$x\in\sE$, and the associativity and unitality axioms are imposed
on the compositions of morphisms.

 To every small $k$\+linear category $\sE$, a nonunital $k$\+algebra
$$
 R_\sE=\bigoplus\nolimits_{x,y\in\sE}\Hom_\sE(x,y),
$$
with the multiplication map $R_\sE\ot_k R_\sE\rarrow R_\sE$
induced by the composition of morphisms in~$\sE$, is assigned.

\begin{prop} \label{category-algebras-t-unital-and-projective}
 For every small $k$\+linear category\/ $\sE$, the $k$\+algebra
$R=R_\sE$ is t\+unital.
 In fact, $R_\sE$ is both left and right s\+unital.
 The algebra $R$ is a projective left and a projective right
$\widecheck R$\+module.
\end{prop}

\begin{proof}
 This is~\cite[Lemmas~6.1 and~6.2]{Pnun}, which are applicable in
view of~\cite[Lemma~6.4]{Pnun}.
\end{proof}

 Let $\sE$ be a small $k$\+linear category.
 Then a \emph{left\/ $\sE$\+module} $\sM$ is defined as a covariant
$k$\+linear functor $\sM\:\sE\rarrow k\Vect$.
 This is the same thing as a covariant additive functor
$\sM\:\sE\rarrow\Ab$, since the action of the scalar multiples of
the identity morphism~$\id_x$ always endows the abelian group $\sM(x)$
with a $k$\+vector space structure.

 Explicitly, a left $\sE$\+module $\sM$ is the datum of a vector space
$\sM(x)$ for every $x\in\sE$ and a linear map
$$
 \Hom_\sE(x,y)\ot_k\sM(x)\lrarrow\sM(y)
$$
for every pair of objects $x$, $y\in\sE$.
 The usual associativity and unitality axioms are imposed.
 We denote the abelian category of left $\sE$\+modules by $\sE\Modl$.

 Similarly, a \emph{right\/ $\sE$\+module} $\sN$ is a contravariant
$k$\+linear functor $\sN\:\sE^\sop\rarrow k\Vect$, or equivalently,
a contravariant additive functor $\sN\:\sE^\sop\rarrow\Ab$.
 In other words, it is the datum of a vector space $\sN(x)$
for every $x\in\sE$ and a linear map
$$
 \Hom_\sE(x,y)\ot_k\sN(y)\lrarrow\sN(x)
$$
for every pair of objects $x$, $y\in\sE$, satisfying the usual
associativity and unitality axioms.
 The abelian category of right $\sE$\+modules is denoted by $\Modr\sE$.

\begin{prop} \label{functors-as-t-unital-modules}
 Let\/ $\sE$ be a small $k$\+linear category and $R=R_\sE$ be
the related nonunital $k$\+algebra. \par
\textup{(a)} There is a natural equivalence of abelian categories
$$
 \sE\Modl\simeq R_\sE\Modlt.
$$
 To any left\/ $\sE$\+module\/ $\sM$, the t\+unital left $R_\sE$\+module
$M=\bigoplus_{x\in\sE}\sM(x)$, with the left action of $R_\sE$ in $M$
induced by the left action of\/ $\sE$ in\/ $\sM$, is assigned. \par
\textup{(b)} There is a natural equivalence of abelian categories
$$
 \Modr\sE\simeq \Modrt R_\sE.
$$
 To any right\/ $\sE$\+module\/ $\sN$, the t\+unital right
$R_\sE$\+module $N=\bigoplus_{x\in\sE}\sN(x)$, with the right action
of $R_\sE$ in $N$ induced by the right action of\/ $\sE$ in\/ $\sN$,
is assigned.
\end{prop}

\begin{proof}
 This is~\cite[Proposition~6.5]{Pnun}.
\end{proof}

\begin{prop} \label{functors-as-c-unital-modules}
 Let\/ $\sE$ be a small $k$\+linear category and $R=R_\sE$ be
the related nonunital $k$\+algebra.
 Then there is a natural equivalence of abelian categories
$$
 \sE\Modl\simeq R_\sE\Modlc.
$$
 To any left\/ $\sE$\+module\/ $\sP$, the c\+unital left $R_\sE$\+module
$P=\prod_{x\in\sE}\sP(x)$, with the left action of $R_\sE$ in $P$
induced by the left action of\/ $\sE$ in\/ $\sP$, is assigned.
\end{prop}

\begin{proof}
 This is~\cite[Proposition~6.6]{Pnun}.
\end{proof}

\subsection{Recollections on locally finite categories}
\label{locally-finite-categories-subsecn}
 Let $\sF$ be a small $k$\+linear category.
 We define a preorder relation~$\preceq$ on the set of all objects of
$\sF$ as the reflexive, transitive binary relation generated by
the following elementary comparisons: $x\preceq y$ if there exists
a nonzero morphism from~$x$ to~$y$ in $\sF$, that is, if
$\Hom_\sF(x,y)\ne0$.
 In other words, we write $x\preceq y$ if there exists a finite chain
of composable nonzero morphisms in $\sF$ going from~$x$ to~$y$
\,\cite[Definition~2.1]{Plfin}.
 Furthermore, we will write $x\prec y$ whenever $x\preceq y$ but
$y\not\preceq x$.

 We say that a small $k$\+linear category $\sF$ is \emph{locally
finite}~\cite[Definition~2.2]{Plfin} if the following two conditions
are satisfied:
\begin{itemize}
\item Hom-finiteness: for any two objects $x$ and $y\in\sF$,
the $k$\+vector space $\Hom_\sF(x,y)$ is finite-dimensional;
\item interval finiteness: for any two objects $x$ and $y\in\sF$,
the set of all objects $z\in\sF$ such that $x\preceq z\preceq y$
is finite.
\end{itemize}

 To every locally finite $k$\+linear category $\sF$, the construction
spelled out in~\cite[Construction~2.3]{Plfin} assigns a coassociative,
counital coalgebra
$$
 \C_\sF=\bigoplus\nolimits_{x,y\in\sF}\Hom_\sF(x,y)^*
$$
over the field~$k$.
 The collection of all the natural duality pairings
$$
 \phi_{x,y}\:\Hom_\sF(x,y)^*\times\Hom_\sF(x,y)\lrarrow k
$$
induces, by passing to the direct sum over all the objects
$x$, $y\in\sF$, a multiplicative bilinear pairing
$$
 \phi_\sF\:\C_\sF\times R_\sF\lrarrow k.
$$

\begin{lem} \label{category-pairings-t-unital}
 For any locally finite $k$\+linear category $\sF$, the multiplicative
pairing $\phi_\sF\:\C_\sF\times R_\sF\rarrow k$ is both left and
right t\+unital (and in fact, s\+unital).
\end{lem}

\begin{proof}
 The condition~(3) of
Proposition~\ref{s-unital-pairings-characterized} is obviously
satisfied.
 So it remains to recall the second assertion of
Proposition~\ref{category-algebras-t-unital-and-projective} and
apply Corollary~\ref{s-unital-t-unital-pairings-cor}.
\end{proof}

 For any locally finite $k$\+linear category $\sF$,
the \emph{comodule inclusion functors}
\begin{alignat}{2}
 \Upsilon_\sF\: &\C_\sF\Comodl &&\lrarrow \sF\Modl, 
 \label{left-comodules-into-functors-inclusion} \\
 \Upsilon_{\sF^\sop}\: &\Comodr\C_{\sF} &&\lrarrow \Modr\sF
 \label{right-comodules-into-functors-inclusion}
\end{alignat}
were constructed in~\cite[Construction~2.4]{Plfin}.
 In the language of the present paper, the functors $\Upsilon_\sF$
and $\Upsilon_{\sF^\sop}$ are interpreted as the functors~$\phi_*$
\,(\ref{left-comodules-to-modules}\+-\ref{right-comodules-to-modules})
from Proposition~\ref{underlying-K-module-structures} and
Remark~\ref{lands-in-t-unital-remark} for the t\+unital
multiplicative pairing $\phi=\phi_\sF$ (cf.\
Proposition~\ref{functors-as-t-unital-modules}).

\begin{lem} \label{Upsilon-fully-faithful}
 For any locally finite $k$\+linear category\/ $\sF$, the functors\/
$\Upsilon_\sF$ and\/ $\Upsilon_{\sF^\sop}$
\,(\ref{left-comodules-into-functors-inclusion}\+-%
\ref{right-comodules-into-functors-inclusion}) are fully faithful.
\end{lem}

\begin{proof}
 This is~\cite[Proposition~2.5]{Plfin}.
 Alternatively, one can observe that the multiplicative
pairing~$\phi_\sF$ is obviously nondegenerate in\/ $\C_\sF$,
i.~e., condition~(7) of
Proposition~\ref{comodule-forgetful-full-and-faithfulness}
is satisfied for it.
 So Proposition~\ref{comodule-forgetful-full-and-faithfulness} tells
that conditions~(1) and~(2) are satisfied as well.
\end{proof}

 For any locally finite $k$\+linear category $\sF$,
the \emph{contramodule forgetful functor}
\begin{equation} \label{contramodules-into-functors-forgetful}
 \Theta_\sF\:\C_\sF\Contra\lrarrow\sF\Modl
\end{equation}
was constructed in~\cite[Construction~3.1]{Plfin}.
 In the language of the present paper, the functor $\Theta_\sF$
is interpreted as the functor~$\phi_*$
\,\eqref{left-contramodules-to-modules}
from Proposition~\ref{underlying-K-module-structures} and
Remark~\ref{lands-in-c-unital-remark} for the t\+unital
multiplicative pairing $\phi=\phi_\sF$ (cf.\
Proposition~\ref{functors-as-c-unital-modules}).

 A sufficient condition for full-and-faithfulness of
the functor~$\Theta_\sF$ was obtained in~\cite[Theorem~3.7]{Plfin}.
 In order to formulate it, we need to recall the next definition.

 A locally finite $k$\+linear category $\sF$ is said
to be \emph{lower strictly locally finite}~\cite[Definition~3.6]{Plfin}
if for every object $y\in\sF$ there exists a finite set of objects
$X_y\subset\sF$ such that $x\prec y$ for all $x\in X_y$ and
the composition map
$$
 \bigoplus\nolimits_{x\in X_y}\Hom_\sF(x,y)\ot_k\Hom_\sF(z,x)
 \lrarrow\Hom_\sF(z,y)
$$
is surjective for all $z\prec y$.

 Dually, a locally finite $k$\+linear category $\sG$ is said to be
\emph{upper strictly locally finite} if the opposite category
$\sG^\sop$ is lower strictly locally finite.

\begin{thm} \label{strictly-lower-contramodule-full-and-faithfulness}
 For any lower strictly locally finite $k$\+linear category\/ $\sF$,
the functor\/ $\Theta_\sF\:\C_\sF\Contra\rarrow\sF\Modl$ is
fully faithful.
\end{thm}

\begin{proof}
 This is~\cite[Theorem~3.7]{Plfin}.
\end{proof}

 The following definition~\cite[Definition~7.3]{Plfin} specifies
an even more restrictive condition on a $k$\+linear category.
 A locally finite $k$\+linear category $\sF$ is said to be
\emph{lower finite} if, for every object $y\in\sF$, the set of all
objects $x\in\sF$ such that $x\preceq y$ is finite.
 Clearly, any lower finite $k$\+linear category is lower strictly
locally finite.

 Dually, $\sF$ is \emph{upper finite} if, for every object $x\in\sF$,
the set of all objects $y\in\sF$ such that $x\preceq y$ is finite.

\begin{prop} \label{upper-lower-finite-categories-prop}
\textup{(a)} For any upper finite $k$\+linear category\/ $\sF$,
the comodule inclusion functor\/ $\Upsilon_\sF\:\C_\sF\Comodl
\rarrow\sF\Modl$ is a category equivalence. \par
\textup{(b)} For any lower finite $k$\+linear category\/ $\sF$,
the comodule inclusion functor\/ $\Upsilon_{\sF^\sop}\:
\Comodr\C_\sF\rarrow\Modr\sF$ is a category equivalence. \par
\textup{(c)} For any lower finite $k$\+linear category\/ $\sF$,
the contramodule forgetful functor\/ $\Theta_\sF\:\C_\sF\Contra
\rarrow\sF\Modl$ is a category equivalence.
\end{prop}

\begin{proof}
 This is~\cite[Propositions~7.4 and~7.5]{Plfin}.
\end{proof}

\subsection{Functors surjective on objects are t-unital}
 Let $\sE$ and $\sF$ be two small $k$\+linear categories.
 A \emph{$k$\+linear functor} $F\:\sF\rarrow\sE$ is a map assigning to
every object $u\in\sF$ an object $F(u)\in\sE$ and to every pair of
objects $u$, $v\in\sF$ a $k$\+linear map
$$
 F(u,v)\:\Hom_\sF(u,v)\lrarrow\Hom_\sE(F(u),F(v)).
$$
 The usual axioms of multiplicativity and unitality must be
satisfied.
 Clearly, any $k$\+linear functor $F\:\sF\rarrow\sE$ induces
a homomorphism of $k$\+algebras $f\:R_\sF\rarrow R_\sE$.

 Given a left $\sE$\+module $\sM\:\sE\rarrow k\Vect$ and a $k$\+linear
functor $F\:\sF\rarrow\sE$, the composition $\sF\rarrow\sE\rarrow\sM$
defines the \emph{underlying left\/ $\sF$\+module structure} on~$\sM$.
 Similarly, given a right $\sE$\+module $\sN\:\sE^\sop\rarrow k\Vect$
and a functor $F$ as above, the composition $\sF^\sop\rarrow\sE^\sop
\rarrow\sN$ defines the \emph{underlying right\/ $\sF$\+module
structure} on~$\sN$.

 The definitions of a \emph{left} (\emph{right}) \emph{t\+unital}
homomorphism of nonunital algebras $f\:K\rarrow R$ were given
in the beginning of Section~\ref{construction-of-semialgebra-subsecn}.
 Let us also recall~\cite[Section~9]{Pnun} that $f$~is called
\emph{left s\+unital} if $f$~makes $R$ an s\+unital left $K$\+module,
and $f$~is \emph{right s\+unital} if $R$ is an s\+unital right
$K$\+module.
 A ring homomorphism~$f$ is called \emph{s\+unital} if it is both
left and right s\+unital.

 An object $z\in\sE$ is said to be \emph{zero} if $\id_z=0$ in
$\Hom_\sE(z,z)$.
 In this case one has $\Hom_\sE(x,z)=0=\Hom_\sE(z,y)$ for all
$x$, $y\in\sE$.

\begin{prop} \label{t-unital-functors}
 Let $F\:\sF\rarrow\sE$ be a $k$\+linear functor of small
$k$\+linear categories, and $f\:R_\sF\rarrow R_\sE$ be the induced
$k$\+algebra homomorphism.
 Then the following conditions are equivalent:
\begin{enumerate}
\item $f$~is left t\+unital;
\item $f$~is right t\+unital;
\item $f$~is left s\+unital;
\item $f$~is right s\+unital;
\item for each nonzero object $x\in\sE$ there exists an object
$u\in\sF$ such that $F(u)=x$.
\end{enumerate}
\end{prop}

 Let us emphasize that condition~(5) requires an \emph{equality} of
objects $F(u)=x$ and \emph{not} just an isomorphism.
 This is the condition of surjectivity of the functor $F$ on
(nonzero) objects, and \emph{not} only essential surjectivity.
 The following simple example is instructive.

\begin{ex}
 Let $\sE$ be a $k$\+linear category with only two objects $x$ and~$y$
that are isomorphic in $\sE$ and nonzero.
 Let $\sF\subset\sE$ be the full subcategory on the object~$x$,
and let $F\:\sF\rarrow\sE$ be the inclusion functor; so $F$ is
a $k$\+linear category equivalence.
 Then $R_\sF$ and $R_\sE$ are actually unital algebras (because
the sets of objects of $\sF$ and $\sE$ are finite), but
the ring homomorphism $f\:R_\sF\rarrow R_\sE$ does \emph{not} take
the unit to the unit.
 Consequently, $f$~is neither left nor right t\+unital.
\end{ex}
 
\begin{proof}[Proof of Proposition~\ref{t-unital-functors}]
 (1)\,$\Longleftrightarrow$\,(3) The ring $R_\sF$ is s\+unital by
Proposition~\ref{category-algebras-t-unital-and-projective}, and it
remains to refer to~\cite[Lemma~9.4(a)]{Pnun}.

 (2)\,$\Longleftrightarrow$\,(4) This is~\cite[Lemma~9.4(b)]{Pnun}.

 (3)~or (4)~$\Longrightarrow$~(5)
 For the sake of contradiction, assume that (5)~does not hold.
 Let $z\in\sE$ be a nonzero object not belonging to the image of~$F$.
 Put $r=\id_z\in\bigoplus_{x,y\in\sE}\Hom_\sE(x,y)=R_\sE$.
 Then for every $e\in R_\sF$ one has $f(e)r=rf(e)=0$ in $R_\sE$,
even though $r\ne0$.
 This contradicts both~(3) and~(4).

 (5)~$\Longrightarrow$~(3) and~(4)
 Let $r\in R_\sE=\bigoplus_{x,y\in\sE}\Hom_\sE(x,y)$ be an element.
 Then there is a finite set of nonzero objects $Z\subset\sE$
such that $r\in\bigoplus_{z,w\in Z}\Hom_\sE(z,w)\subset R_\sE$.
 Let $U\subset\sF$ be a finite set of objects such that for
every $z\in Z$ there exists $u\in U$ with $F(u)=z$.
 Put $e=\sum_{u\in U}\id_u\in R_\sF$.
 Then $f(e)r=r=rf(e)$ in~$R_\sE$.
\end{proof}

\subsection{Tensor products and triangular decompositions}
\label{tensor-and-triangular-subsecn}
 Let $\sH$ be a small $k$\+linear category, $\sN$ be a right
$\sH$\+module, and $\sM$ be a left $\sH$\+module.
 Then the \emph{tensor product} $\sN\ot_\sH\sM$ is a $k$\+vector
space constructed as the cokernel of (the difference of) the
natural pair of maps
$$
 \bigoplus\nolimits_{x,y\in\sH}\sN(y)\ot_k\Hom_\sH(x,y)\ot_k\sM(x)
 \,\rightrightarrows\,\bigoplus\nolimits_{z\in\sE}\sN(z)\ot_k\sM(z).
$$
 Here one of the maps is induced by the right action of $\sH$ in $\sN$
and the other one by the left action of $\sH$ in~$\sM$.

 In fact, let $N$ be the t\+unital right $R_\sH$\+module corresponding
to $N$ and $M$ be the t\+unital left $R_\sH$\+module corresponding
to~$M$ (as per Proposition~\ref{functors-as-t-unital-modules}).
 Then there is a natural isomorphism of $k$\+vector spaces
$$
 \sN\ot_\sH\sM\simeq N\ot_{R_\sH}M.
$$

 An \emph{$\sH$\+$\sH$\+bimodule} $\sB$ is defined as a $k$\+linear
functor of two arguments $\sB\:\sH\times\sH^\sop\rarrow k\Vect$.
 In other words, it is the datum of a vector space $\sB(x,y)$ for
every pair of objects $x$, $y\in\sH$, and a linear map
$$
 \Hom_\sH(x,y)\ot_k\sB(z,x)\ot_k\Hom_\sH(w,z)\lrarrow
 \Hom_\sH(w,y)
$$
for every quadruple of objects $x$, $y$, $z$, $w\in\sH$.
 The usual associativity and unitality axioms are imposed.
 
 Strictly speaking, an $\sH$\+$\sH$\+bimodule $\sB$ has \emph{no}
underlying left or right $\sH$\+module.
 Rather, there is an \emph{underlying collection of left\/
$\sH$\+modules} $(\sM_x\in\sE\Modl)_{x\in\sH}$ with
$\sM_x(y)=\sB(x,y)$ for all $y\in\sH$, and a similar \emph{underlying
collection of right\/ $\sH$\+modules} $(\sN_y\in\Modr\sE)_{y\in\sH}$
with $\sN_y(x)=\sB(x,y)$ for all $x\in\sH$.

 On the other hand, to any $\sH$\+$\sH$\+bimodule $\sB$ one can
assign a t\+unital $R_\sH$\+$R_\sH$\+bimod\-ule $B$ defined by
the obvious rule
$$
 B=\bigoplus\nolimits_{x,y\in\sE}\sB(x,y).
$$

 Let $\sB'$ and $\sB''$ be two $\sH$\+$\sH$\+bimodules.
 Then the \emph{tensor product} $\sB=\sB'\ot_\sH\sB''$ is
an $\sH$\+$\sH$\+bimodule defined by the rule
$$
 \sB(x,y)=\sN'_y\ot_\sH\sM''_x,
$$
where, as above, $\sN'_y$ is the right $\sH$\+module with
$\sN'_y(x)=\sB'(x,y)$, and $\sM''_x$ is the left $\sH$\+module with
$\sM''_x(y)=\sB''(x,y)$.
 The t\+unital $R_\sH$\+$R_\sH$\+bimodule $B$ corresponding to
the $\sH$\+$\sH$\+bimodule $B$ can be constructed as
$$
 B=B'\ot_{R_\sH}B''.
$$

 Now let $\sE$ be a small $k$\+linear category and $\sH\subset\sE$ be
a $k$\+linear subcategory \emph{on the same set of objects}.
 This means that, for every pair of objects $x$, $y\in\sE$,
a vector subspace $\Hom_\sH(x,y)\subset\Hom_\sE(x,y)$ is chosen in
such a way that $\id_x\in\Hom_\sH(x,x)$ for all $x\in\sE$ and
the composition of morphisms in $\sE$ belonging to $\sH$ again
belongs to~$\sH$.
 Then there is a $k$\+linear natural inclusion functor $\sH\rarrow\sE$.

 In this setting, the rule $\sB_\sE(x,y)=\Hom_\sE(x,y)$ defines
an $\sH$\+$\sH$\+bimodule~$\sB_\sE$.
 The corresponding t\+unital $R_\sH$\+$R_\sH$\+bimodule is
$B_\sE=R_\sE$, with the $R_\sH$\+$R_\sH$\+bimodule structure on $R_\sE$
induced by the $k$\+algebra homomorphism $R_\sH\rarrow R_\sE$
corresponding to the inclusion functor $\sH\rarrow\sE$.

 Let $\sF$ and $\sG\subset\sE$ be two other $k$\+linear subcategories
on the same set of objects.
 Assume that $\sH\subset\sF$ and $\sH\subset\sG$.
 Then the multiplication of morphisms in $\sE$ induces
an $\sH$\+$\sH$\+bimodule map
\begin{equation} \label{three-subcategories-bimodule-map}
 \sB_\sF\ot_\sH\sB_\sG\lrarrow\sB_\sE,
\end{equation}
which we will simply denote, by an abuse of notation, by
\begin{equation} \label{simplified-notation-three-subcategories-map}
 \sF\ot_\sH\sG\lrarrow\sE.
\end{equation}
 The corresponding morphism of t\+unital $R_\sH$\+$R_\sH$\+bimodules is
\begin{equation} \label{three-subcategories-t-unital-bimodule-map}
 R_\sF\ot_{R_\sH}R_\sG\lrarrow R_\sE.
\end{equation}

 We will say that a triple of $k$\+linear subcategories on the same
set of objects $\sF$, $\sG$, $\sH\subset\sE$ forms a \emph{triangular
decomposition} of $\sE$ if
the map~\eqref{three-subcategories-bimodule-map}, or which is the same,
the map~\eqref{simplified-notation-three-subcategories-map} is
an isomorphism of $\sH$\+$\sH$\+bimodules.
 Equivalently, three subcategories $\sF$, $\sG$, $\sH\subset\sE$ form
a triangular decomposition of $\sE$ if and only if
the $R_\sH$\+$R_\sH$\+bimodule
map~\eqref{three-subcategories-t-unital-bimodule-map} is an isomorphism.
 Explicitly, this means that the following sequence of $k$\+vector
spaces is right exact for all pairs of objects $x$, $y\in\sE$:
\begin{multline} \label{triangular-decomposition-right-exact-sequence}
 \bigoplus\nolimits_{u,v\in\sH}
 \Hom_\sF(v,y)\ot_k\Hom_\sH(u,v)\ot_k\Hom_\sG(x,u) \\ \lrarrow
 \bigoplus\nolimits_{z\in\sE}\Hom_\sF(z,y)\ot_k\Hom_\sG(x,z)
 \lrarrow\Hom_\sE(x,y)\lrarrow0.
\end{multline}

\subsection{Triangular decompositions and projectivity}
\label{triangular-and-projectivity-subsecn}
 We will say that a $k$\+linear category $\sH$ is \emph{discrete}
if $\dim_k\Hom_\sH(x,y)=\delta_{x,y}$ for all $x$, $y\in\sH$.
 So $\Hom_\sH(x,y)=0$ if $x\ne y$, and the vector space $\Hom_\sH(x,x)$
consists of the scalar multiples of the identity morphism~$\id_x$,
which is assumed to be nonzero.

 From now on it will be convenient for us to assume that there are
no zero objects in the small $k$\+linear category~$\sE$.
 Then the category $\sE^\id$ defined by the rules
$$
 \Hom_{\sE^\id}(x,y)=
 \begin{cases}
  k\cdot\id_x & \text{if $x=y\in\sE$}, \\
  0 & \text{if $x\ne y\in\sE$}
 \end{cases}
$$
is the unique discrete $k$\+linear subcategory on the same set of
objects in~$\sE$.

 In the sequel, we will be mostly interested in triangular
decompositions $\sF$, $\sG$, $\sH\subset\sE$ with a discrete
$k$\+linear category~$\sH$; so $\sH=\sE^\id$.
 In this case, the condition of right exactness of
the sequence~\eqref{triangular-decomposition-right-exact-sequence}
takes the form of an isomorphism
\begin{equation} \label{triangular-over-discrete-isomorphism}
 \bigoplus\nolimits_{z\in\sE}\Hom_\sF(z,y)\ot_k\Hom_\sG(x,z)
 \overset\simeq\lrarrow\Hom_\sE(x,y)
\end{equation}
for all objects $x$, $y\in\sE$.

\begin{lem} \label{t-unital-projective-over-discrete}
 Let $\sH$ be a small discrete $k$\+linear category and $H=R_\sH$ be
the related nonunital $k$\+algebra.
 Then all t\+unital $H$\+modules are projective as unital
$\widecheck H$\+modules, while all c\+unital $H$\+modules are
injective as unital $\widecheck H$\+modules.
\end{lem}

\begin{proof}
 The point is that the abelian category
$$
 H\Modlt\simeq \widecheck H\Modl/\widecheck H\Modlz\simeq
 H\Modlc\simeq \sH\Modl
$$
(see formula~\eqref{four-category-equivalence} in
Section~\ref{null-modules-prelim-subsecn}
and~\cite[Proposition~6.7]{Pnun}) is semisimple.
 In fact, it is clear that the category $\sH\Modl$ is equivalent
to the Cartesian product of copies of $k\Vect$ taken over all
the objects $x\in\sH$.

 Consequently, all objects in this abelian category are both
projective and injective.
 Then it remains to refer to the discussion in
Section~\ref{t-c-flat-proj-inj-prelim-subsecn} for the assertions
that a t\+unital $K$\+module is projective in $\widecheck K\Modl$
if and only it is projective in $K\Modlt\simeq\widecheck K\Modl/
\widecheck K\Modlz$, and a c\+unital $K$\+module is injective
in $\widecheck K\Modl$ if and only if it is injective in
$K\Modlc\simeq\widecheck K\Modl/\widecheck K\Modlz$.
\end{proof}

\begin{prop} \label{triangular-decomposition-projectivity-prop}
 Let\/ $\sE$ be a small $k$\+linear category and\/
$\sF\ot_\sH\sG\simeq\sE$ be a triangular decomposition of\/ $\sE$
with a discrete $k$\+linear category\/ $\sH=\sE^\id$.
 Put $K=R_\sF$ and $R=R_\sE$, and denote by $f\:K\rarrow R$
the morphism of nonunital $k$\+algebras induced by the inclusion
functor\/ $\sF\rarrow\sE$.
 Then $R$ is a left projective $K$\+$K$\+bimodule.
\end{prop}

\begin{proof}
 We need to prove that $R$ is a projective left $\widecheck K$\+module.
 Indeed, we have
$$
 R=R_\sE\simeq R_\sF\ot_{R_\sH}R_\sG=K\ot_{\widecheck H}R_\sG,
$$
where $H=R_\sH$ (see
formula~\eqref{three-subcategories-t-unital-bimodule-map}).
 It remains to recall that $R_\sG$ is a t\+unital left $R_\sH$\+module
by Proposition~\ref{t-unital-functors}, hence $R_\sG$ is
a projective left $\widecheck H$\+module by
Lemma~\ref{t-unital-projective-over-discrete}; and at the same time
$K$ is a projective left $\widecheck K$\+module by
Proposition~\ref{category-algebras-t-unital-and-projective}.
\end{proof}

\subsection{Conclusion}
 Let us summarize the results we have proved so far, in the context
of triangular decompositions of $k$\+linear categories.

 Let $\sE$ be a small $k$\+linear category and $\sF\ot_\sH\sG\simeq\sE$
be a triangular decomposition of $\sE$ with a discrete $k$\+linear
category $\sH=\sE^\id$.
 Put $K=R_\sF$ and $R=R_\sE$, and denote by $f\:K\rarrow R$ the morphism
of $k$\+algebras induced by the inclusion $\sF\rarrow\sE$.
 Then, by Proposition~\ref{category-algebras-t-unital-and-projective},
\,$K$ is a left projective t\+unital $k$\+algebra.
 By Propositions~\ref{t-unital-functors}
and~\ref{triangular-decomposition-projectivity-prop},
\,$f$~is a t\+unital homomorphism of $k$\+algebras making $R$
a left projective $K$\+$K$\+bimodule.

 Assume further that the $k$\+linear category $\sF$ is locally finite
(in the sense of Section~\ref{locally-finite-categories-subsecn}),
and consider the coalgebra $\C=\C_\sF$.
 Then $\phi=\phi_\sF\:\C\times K\rarrow k$ is a right t\+unital
multiplicative pairing (by Lemma~\ref{category-pairings-t-unital}).
 Furthermore, the induced functor $\phi_*=\Upsilon_{\sF^\sop}\:
\Comodr\C\rarrow\Modrt K=\Modr\sF$ is fully faithful
(by Lemma~\ref{Upsilon-fully-faithful}).

 It remains to assume that the (t\+unital) right $K$\+module structure
on the tensor product $\C\ot_KR$ arises from some right $\C$\+comodule
structure, as in Section~\ref{integrated-bimodules-subsecn}.
 In other words, assume that $R$ is a right integrable
$K$\+$K$\+bimodule (as defined in the end of
Section~\ref{fully-faithful-comodule-inclusion-secn}).
 This condition holds automatically when the $k$\+linear category $\sF$
is lower finite (by 
Proposition~\ref{upper-lower-finite-categories-prop}(b)).

\begin{cor} \label{category-semialgebra-right-semimodules}
 Under the list of assumptions above, the tensor product\/
$\bS=\C\ot_K\nobreak R$ acquires a natural structure of
semiassociative, semiunital semialgebra over the coalgebra\/~$\C$.
 Moreover, $\bS$ is an injective left\/ $\C$\+comodule.
 The abelian category of right\/ $\bS$\+semimodules\/ $\Simodr\bS$ is
equivalent to the full subcategory in the module category\/ $\Modr\sE$
consisting of all the right\/ $\sE$\+modules whose underlying right\/
$\sF$\+module structure arises from some right\/ $\C$\+comodule
structure via the functor $\phi_*=\Upsilon_{\sF^\sop}\:\Comodr\C
\rarrow\Modr\sF$.
\end{cor}

\begin{proof}
 The first two assertions are
Corollary~\ref{fully-faithful-lift-of-monoid-cor} with
the subsequent discussion.
 The last assertion is
Corollary~\ref{right-semimodules-described-fully-faithful-case}.
 The equivalences of categories $\Modrt R\simeq\Modr\sE$ and
$\Modrt K\simeq\Modr\sF$ are provided by
Proposition~\ref{functors-as-t-unital-modules}.
\end{proof}

 On top of the assumptions above, let us now assume additionally
that the functor $\phi_*=\Theta_\sF\:\C\Contra\rarrow K\Modlc=
\sF\Modl$ is fully faithful.
 This condition holds whenever the $k$\+linear category $\sF$ is
lower strictly locally finite
(by Theorem~\ref{strictly-lower-contramodule-full-and-faithfulness}),
which includes the easy special case when $\sF$ is lower finite.

\begin{cor} \label{category-semialgebra-left-semicontramodules}
 Under the full list of assumptions above, the abelian category of
left\/ $\bS$\+semicontramodules\/ $\bS\Sicntr$ is equivalent to
the full subcategory in the module category\/ $\sE\Modl$ consisting
of all the left\/ $\sE$\+modules whose underlying left\/ $\sF$\+module
structure arises from some left\/ $\C$\+contramodule structure via
the functor $\phi_*=\Theta_\sF\:\C\Contra\rarrow\sF\Modl$.
\end{cor}

\begin{proof}
 This is
Corollary~\ref{left-semicontramods-described-fully-faithful-case}.
  The equivalences of categories $R\Modlc\simeq\sE\Modl$ and
$K\Modlc\simeq\sF\Modl$ are provided by
Proposition~\ref{functors-as-c-unital-modules}.
\end{proof}

\begin{cor} \label{category-semialgebra-lower-finite-case}
 Under the list of assumptions above, assume additionally that
the $k$\+linear category\/ $\sF$ is lower finite.
 Then the natural forgetful functors are equivalences of abelian
categories\/ $\Simodr\bS\simeq\Modr\sE$ and\/
$\bS\Sicntr\simeq\sE\Modl$.
\end{cor}

\begin{proof}
 Follows from
Corollaries~\ref{category-semialgebra-right-semimodules}\+-%
\ref{category-semialgebra-left-semicontramodules} together with
Proposition~\ref{upper-lower-finite-categories-prop}(b\+-c).
\end{proof}

\Section{Examples}  \label{examples-secn}
 
\subsection{The trivial triangular decomposition}
 Let $X$ be a fixed set.
 Denote by $\C_X=\bigoplus_{x\in X}k$ the direct sum of $X$ copies
of the coalgebra~$k$ over~$k$ (with the natural, essentially
unique structure of counital $k$\+coalgebra on~$k$).

 By a \emph{$k$\+linear $X$\+category} we mean a $k$\+linear
category $\sE$ with the set of objects identified with~$X$.
 A \emph{$k$\+linear $X$\+functor} is a $k$\+linear functor of
$k$\+linear $X$\+categories acting by the identity map on
the objects.

 The following observation is attributed to Chase in Aguiar's
dissertation~\cite[Definition~2.3.1 and Section~9.1]{Agu}, where
the term ``internal categories'' was used for what we call
``semialgebras''.
 This observation and, slightly more generally, the language of
semialgebras over cosemisimple coalgebras were used in
the paper~\cite[Section~2]{HL}.

\begin{prop}[\cite{Agu}] \label{categories-as-semialgebras}
 For any set $X$, the category of $k$\+linear $X$\+categories\/ $\sE$,
with $k$\+linear $X$\+functors as morphisms, is equivalent to
the category of semialgebras\/ $\bS$ over the coalgebra\/~$\C_X$.
 For a semialgebra\/ $\bS$ corresponding to a $k$\+linear category\/
$\sE$, the category of right\/ $\bS$\+semimodules is equivalent to
the category of right\/ $\sE$\+modules, while both the categories
of left\/ $\bS$\+semimodules and left\/ $\bS$\+semicontramodules
are equivalent to the category of left\/ $\sE$\+modules:
$$
 \Simodr\bS\simeq\Modr\sE \quad\text{and}\quad
 \bS\Simodl\simeq\sE\Modl\simeq\bS\Sicntr.
$$
\end{prop}

\begin{rem}
 One can state the assertion of the proposition more generally, by
letting the set $X$ vary and considering the whole category of
small $k$\+linear categories and arbitrary $k$\+linear functors.
 Then the category of small $k$\+linear categories is embedded as
a full subcategory into the category of pairs (coalgebra~$\C$,
semialgebra~$\bS$ over~$\C$), with compatible pairs (morphism
of coalgebras, morphism of semialgebras) as morphisms.
 We refer to~\cite[Chapter~8]{Psemi} for a detailed discussion of
such morphisms of pairs (coalgebra, semialgebra).
 The full subcategory appearing here consists of all the pairs
$(\C,\bS)$ such that the coalgebra $\C$ has the form $\C=\C_X$
for some set~$X$.
 The set~$X$ serves as the set of objects of the related $k$\+linear
category.
\end{rem}

\begin{proof}[Proof of Proposition~\ref{categories-as-semialgebras}]
 The assertions are easy, and the argument below is intented to serve
as an illustration/example of the theory developed above in this paper.
 Without striving for full generality, let us restrict ourselves to
$k$\+linear $X$\+categories $\sE$ without zero objects; and suppress
the discussion of $X$\+functors, limiting ourselves to a construction
of the semialgebra associated to a $k$\+linear category.

 For any $k$\+linear $X$\+category $\sE$ (without zero objects), we
consider the \emph{trivial triangular decomposition}: $\sF=\sH=\sE^\id$
and $\sG=\sE$.
 Then the discrete $k$\+linear categories $\sF$ and $\sH$ are locally
finite, and in fact, both upper and lower finite.
 The coalgebra $\C_\sF=\C_\sH=\C$ is naturally isomorphic to~$\C_X$.
 Therefore, Corollaries~\ref{category-semialgebra-right-semimodules}\+-%
\ref{category-semialgebra-lower-finite-case} are applicable, and they
produce a semialgebra $\bS$ over~$\C_X$ such that
$\Simodr\bS\simeq\Modr\sE$ and $\bS\Sicntr\simeq\sE\Modl$.

 It remains to explain the equivalence $\bS\Simodl\simeq\sE\Modl$.
 One can observe that the construction of semialgebra in this
proposition is actually left-right symmetric: it assigns the opposite
semialgebra to the opposite category.
 Indeed, in the situation at hand, the $K$\+$K$\+bimodule $\C$ is
naturally isomorphic to $K=R_\sF=R_\sH$; so one has
$\bS=\C\ot_KR_\sE\simeq R_\sE$.
 Thus the equivalence $\bS\Simodl\simeq\sE\Modl$ is analogous to
$\Simodr\bS\simeq\Modr\sE$, which we have already proved.

 Alternatively, a coalgebra $\C$ is called \emph{cosemisimple} if
the abelian category $\C\Comodl$ is semisimple, or equivalently,
the abelian category $\Comodr\C$ is semisimple, or equivaletnly,
the abelian category $\C\Contra$ is
semisimple~\cite[Chapters~VIII\+-IX]{Swe},
\cite[Lemma~3.3]{Pksurv}, \cite[Section~A.2]{Psemi},
\cite[beginning of Section~4.5]{Pkoszul}, \cite[Theorem~6.2]{PS3}.
 For any semialgebra $\bS$ over a cosemisimple coalgebra~$\C$,
there is a natural equivalence of abelian categories
$\bS\Simodl\simeq\bS\Sicntr$; this is a particular case of
the \emph{underived semico-semicontra
correspondence}~\cite[Section~0.3.7]{Psemi}, \cite[Section~3.5]{Prev}.
\end{proof}

\subsection{Brauer diagram category} \label{brauer-subsecn}
 The following example was suggested to me by Catharina Stroppel.
 Let $k$~be a field and $\delta\in k$ be an arbitrary chosen element.
 The \emph{Brauer diagram category} $\Br(\delta)$
\,\cite[Section~5]{Bra}, \cite[Section~2]{Wen}, \cite[Example~2.10]{BM}
is the following small $k$\+linear category.

 The objects of $\Br(\delta)$ are the nonnegative integers $n\ge0$
interpreted as finite sets $x_n=\{1,\dotsc,n\}$.
 We will denote the cardinality of a finite set $x$ by
$|x|\in\boZ_{\ge0}$.
 The assertions below will sometimes depend on the assumption that
\emph{no other finite sets but $x_n=\{1,\dotsc,n\}$ are allowed
as objects of the category\/ $\Br(\delta)$}, while on other occasions
it will be harmless and convenient to pretend that the objects of
$\Br(\delta)$ are arbitrary finite sets.

 Given two finite sets $x$ and~$y$, the $k$\+vector space of
morphisms $\Hom_\Br(x,y)$ is zero when the parities of
the integers $|x|$ and~$|y|$ are different.
 When the cardinalities $|x|$ and~$|y|$ have the same parity, i.~e.,
the cardinality of the disjoint union $x\sqcup y$ is an even number,
the $k$\+vector space $\Hom_\Br(x,y)$ has a basis consisting of all
the partitions~$a$ of $x\sqcup y$ into a disjoint union of sets of
cardinality~$2$.
 The basis elements $a\in\Hom_\Br(x,y)$ are interpreted as unoriented
graphs $G_a$, without loops, with the set of vertices $x\sqcup y$ and
every vertex adjacent to exactly one edge.
 The edges connecting two vertices from~$x$ are called \emph{caps},
the edges connecting two vertices from~$y$ are called \emph{cups},
and the edges connecting a vertex from~$x$ with a vertex from~$y$
are called \emph{propagating strands}~\cite[Section~2]{DG}.
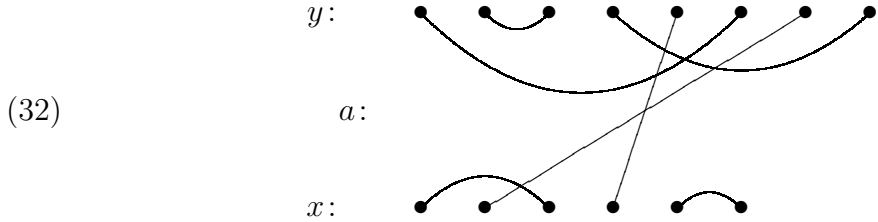
\begin{equation} \label{brauer-morphism-example}
\begin{gathered}
 \xymatrix{
  y\: & \bul \ar@{-}@/^-2.5pc/[rrrrr] 
  & \bul \ar@{-}@/^-0.5pc/[r] & \bul
  & \bul \ar@{-}@/^-1.8pc/[rrrr]
  & \bul \ar@{-}[ldd] & \bul & \bul \ar@{-}[llllldd] & \bul \\
  \quad a\:\mkern-18mu \\
  x\: & \bul \ar@{-}@/^1pc/[rr] & \bul & \bul & \bul
  & \bul \ar@{-}@/^0.5pc/[r] & \bul
 }
\end{gathered}
\end{equation}

 Given three finite sets $x$, $z$, $y$ and two basis elements
$b\in\Hom_\Br(x,z)$ and $c\in\Hom_\Br(z,y)$, the product
$cb\in\Hom_\Br(x,y)$ is defined as follows.
 Consider the graph $G=G_b\cup G_c$ with the set of vertices
$x\sqcup z\sqcup y$ and the set of edges equal to the disjoint
union of the sets of edges of $G_b$ and~$G_c$.
 Let $L$ be a connected component of~$G$.
 Then there are two possibilities.
 Either $L$ has the shape of a line segment with two ends in
$x\sqcup y$, passing through some vertices from~$z$: so there are
exactly two vertices in $x\sqcup y$ belonging to $L$, and these are
the only two vertices in $L$ which are adjacent to exactly one edge;
while all the other vertices in $L$ belong to~$z$ and are adjacent to
exactly two edges.
 Or $L$ has the shape of a circle passing only through vertices
from~$z$: so there are no vertices from $x\sqcup y$ belonging to $L$,
and every vertex in $L$ belongs to~$z$ and is adjacent to exactly
two edges.

 By the definition, one puts $cb=\delta^na\in\Hom_\Br(x,y)$, where
$\delta\in k$ is our chosen scalar parameter and $n\ge0$ is the number
of all connected components in $G$ that have the shape of a circle.
 The basis element $a\in\Hom_\Br(x,y)$ corresponds to the graph $G_a$
consisting of all the connected components $L\subset G$ which have
the shape of a line segment.
 The intermediate vertices belonging to~$z$ are removed from any such
component $L$, and it is viewed as a single edge, i.~e.,
a propagating strand connnecting a vertex from~$x$ with a vectex
from~$y$.

 For example, the composition~$cb$ of the graph/partition~$b$ depicted
in the lower half of the next
diagram~\eqref{brauer-multiplication-diagram} with
the graph/partition~$c$ depicted in the upper half of the same diagram
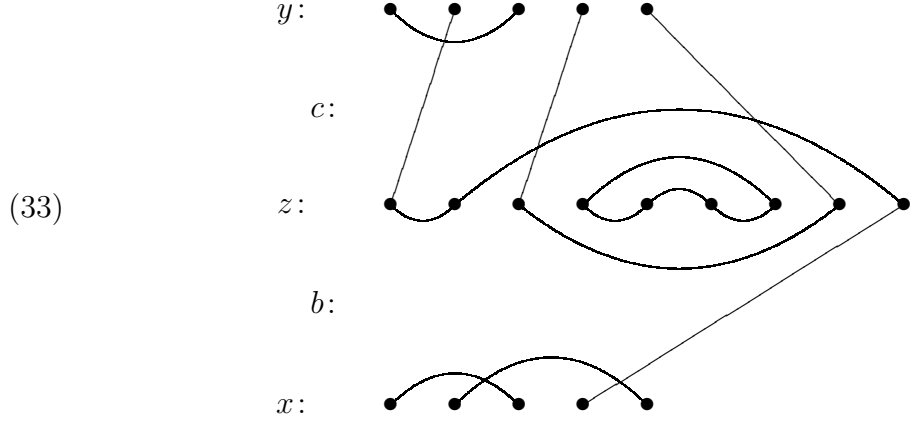
\begin{equation} \label{brauer-multiplication-diagram}
\begin{gathered}
 \xymatrix{
  y\: & \bul \ar@{-}@/^-1pc/[rr] & \bul \ar@{-}[ldd] & \bul
  & \bul \ar@{-}[ldd] &\bul \ar@{-}[rrrdd] \\
  \quad c\: \mkern-18mu \\
  z\: & \bul \ar@{-}@/^-0.5pc/[r] & \bul \ar@{-}@/^3pc/[rrrrrrr]
  & \bul \ar@{-}@/^-2pc/[rrrrr]
  & \bul \ar@{-}@/^-0.5pc/[r] \ar@{-}@/^1.5pc/[rrr]
  & \bul \ar@{-}@/^0.5pc/[r]
  & \bul \ar@{-}@/^-0.5pc/[r] & \bul
  & \bul & \bul \ar@{-}[llllldd] \\
  \quad b\: \mkern-18mu \\
  x\: & \bul \ar@{-}@/^1pc/[rr] & \bul \ar@{-}@/^1.5pc/[rrr]
  & \bul & \bul & \bul
 }
\end{gathered}
\end{equation}
is equal to $\delta$~times the graph/partition~$a$ depicted on
the diagram
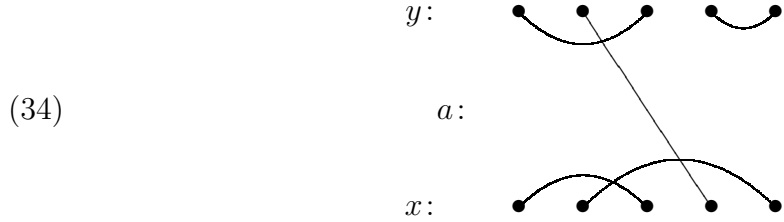
\begin{equation}
\begin{gathered}
 \xymatrix{
  y\: & \bul \ar@{-}@/^-1pc/[rr] & \bul \ar@{-}[rrdd]
  & \bul & \bul \ar@{-}@/^-0.5pc/[r] & \bul \\
  \quad a\: \mkern-18mu \\
  x\: & \bul \ar@{-}@/^1pc/[rr] & \bul \ar@{-}@/^1.5pc/[rrr]
  & \bul & \bul & \bul
 }
\end{gathered}
\end{equation}
so $cb=\delta a$ in $\Hom_\Br(x,y)$ in this case.

 Now let us introduce notation for some $k$\+linear subcategories
in $\Br(\delta)$ which we are interested in.
 All such subcategories will have the same set of objects as
$\Br(\delta)$; so the objects are the nonnegative integers.
 The subcategory $\Br^{+=}(\delta)\subset\Br(\delta)$ consists of
all the morphisms in $\Br(\delta)$ which \emph{contain no caps}.
 So one has $\Hom_{\Br^{+=}}(z,y)=0$ if $|z|>|y|$, while if
$|y|-|z|$ is a nonnegative even number, a basis in the $k$\+vector
space $\Hom_{\Br^{+=}}(z,y)$ is formed by all the partitions~$c$
of the set $z\sqcup y$ into two-element subsets in which no
two elements of~$z$ are grouped together.

 Similarly, the subcategory $\Br^{=-}(\delta)\subset\Br(\delta)$
consists of all the morphisms in $\Br(\delta)$ which
\emph{contain no cups}.
 So $\Hom_{\Br^{=-}}(x,z)=0$ if $|z|>|x|$, while if $|x|-|z|$ is
a nonnegative even number, a basis in $\Hom_{\Br^{=-}}(x,z)$ is
formed by all the partitions~$b$ of the set $x\sqcup z$ into
pairs of elements in which no two elements of~$z$ are grouped
in one pair.
 For example, the partitions $b$ and~$c$ on
the diagram~\eqref{brauer-multiplication-diagram} are typical examples
of morphisms in $\Br(\delta)$ \emph{not} belonging to
$\Br^{=-}(\delta)$ and $\Br^{+=}(\delta)$, respectively.

 The subcategory $\Br^=(\delta)\subset\Br(\delta)$ consists of
all the morphisms containing \emph{neither caps nor cups}.
 So the vector space $\Hom_{\Br^=}(z,w)$ is only nonzero when
$|z|=|w|$, and a basis in this vector space is formed by all
bijections between~$z$ and~$w$.
 The $k$\+algebra $\Hom_{\Br^=}(z,z)$ is the group algebra of
the symmetric group, $\Hom_{\Br^=}(z,z)\simeq k[\mathbb S_n]$,
where $n=|z|$.

\begin{prop} \label{triangular-decomposition-over-symmetric-group}
 The triple of $k$\+linear subcategories\/ $\sF=\Br^{+=}(\delta)$,
\,$\sG=\Br^{=-}(\delta)$, and\/ $\sH=\Br^=(\delta)$ forms
a \emph{triangular decomposition} of the $k$\+linear category\/
$\Br(\delta)$ in the sense of
Section~\ref{tensor-and-triangular-subsecn}.
 In other words,
the map~\eqref{simplified-notation-three-subcategories-map}
$$
 \Br^{+=}(\delta)\ot_{\Br^=(\delta)}\Br^{=-}(\delta)\lrarrow
 \Br(\delta)
$$
is an isomorphism; or equivalently,
the sequence~\eqref{triangular-decomposition-right-exact-sequence}
is right exact.
\end{prop}

\begin{proof}
 The assertion is essentially straightforward or geometrically
intuitive, and instead of a formal proof, we will draw two
illustrative diagrams for a specific example.
 The morphism (basis element) $a\in\Hom_\Br(x,y)$ from
diagram~\eqref{brauer-morphism-example} can be factorized into
a composition of two basis elements of the spaces of morphisms in
the subcategories $\Br^{=-}(\delta)$ and $\Br^{+=}(\delta)$
in two ways:
$$
 \mkern-16mu
 
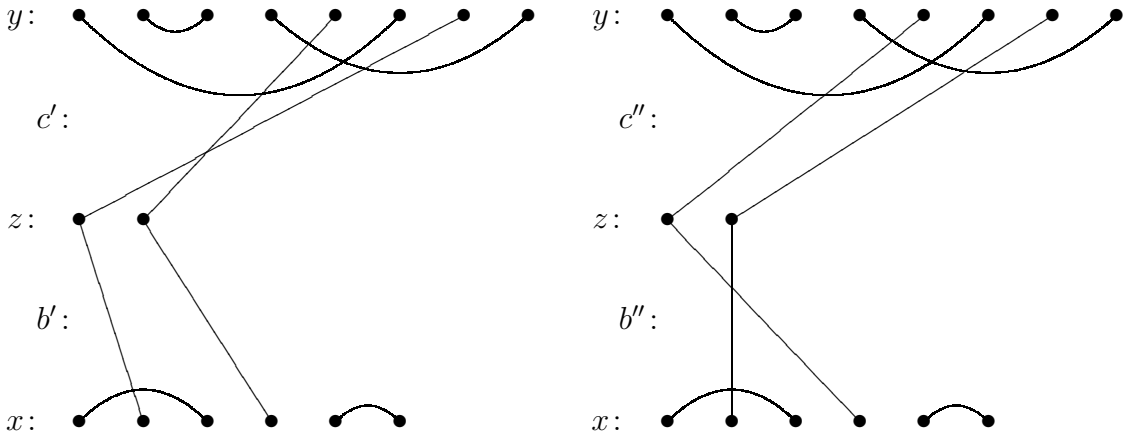
\xymatrix{
  y\:\mkern-24mu & \bul \ar@{-}@/^-2.5pc/[rrrrr] 
  & \bul \ar@{-}@/^-0.5pc/[r] & \bul
  & \bul \ar@{-}@/^-1.8pc/[rrrr]
  & \bul \ar@{-}[llldd] & \bul & \bul \ar@{-}[lllllldd] & \bul \\
  \quad c'\:\mkern-42mu \\
  z\:\mkern-24mu & \bul \ar@{-}[rdd] & \bul \ar@{-}[rrdd] \\
  \quad b'\:\mkern-42mu \\
  x\:\mkern-24mu & \bul \ar@{-}@/^1pc/[rr] & \bul & \bul & \bul
  & \bul \ar@{-}@/^0.5pc/[r] & \bul
 }
 \mkern28mu
 \xymatrix{
  y\: \mkern-24mu & \bul \ar@{-}@/^-2.5pc/[rrrrr] 
  & \bul \ar@{-}@/^-0.5pc/[r] & \bul
  & \bul \ar@{-}@/^-1.8pc/[rrrr]
  & \bul \ar@{-}[lllldd] & \bul & \bul \ar@{-}[llllldd] & \bul \\
  \quad c''\:\mkern-42mu \\
  z\: \mkern-24mu & \bul \ar@{-}[rrrdd] & \bul \ar@{-}[dd] \\
  \quad b''\:\mkern-42mu \\
  x\: \mkern-24mu & \bul \ar@{-}@/^1pc/[rr] & \bul & \bul & \bul
  & \bul \ar@{-}@/^0.5pc/[r] & \bul
 }
$$

 The expressions $c'\ot b'$ and $c''\ot b''$ represent
two different elements in the tensor product
$\Br^{+=}(\delta)\ot_{\Br(\delta)^\id}\Br^{=-}(\delta)$ over
the discrete $k$\+linear category $\Br(\delta)^\id$,
but they are one and the same element of the tensor product
$\Br^{+=}(\delta)\ot_{\Br^=(\delta)}\Br^{=-}(\delta)$
over the $k$\+linear subcategory $\Br^=(\delta)\subset\Br(\delta)$.
\end{proof}

 So far we viewed our category objects $x_n=\{1,\dotsc,n\}$ as
finite sets; now let us view them as \emph{linearly ordered}
finite sets.
 This (less invariant) point of view allows to define another two
subcategories in $\Br(\delta)$.
 The set of objects of both the subcategories is still the same as
in $\Br(\delta)$.

 The subcategory $\Br^+(\delta)\subset\Br^{+=}(\delta)$ consists
of all the morphisms in $\Br(\delta)$ which contain \emph{no cups}
and \emph{no intersecting propagatings strands}.
 So a basis in the $k$\+vector space $\Hom_{\Br^+}(z,y)$ is formed
by all the partitions~$b$ of the set $z\sqcup y$ into two-element
subsets such that
\begin{itemize}
\item no two elements of~$z$ are grouped together;
\item if an element $s\in z$ is grouped together with an element
$i\in y$ in the partition~$b$, and an element $t\in z$ is grouped
together with an element $j\in y$ in~$b$, and if $s<t$ in~$z$,
then $i<j$ in~$y$.
\end{itemize}

 Similarly, the subcategory $\Br^-(\delta)\subset\Br^{=-}(\delta)$
consists of all the morphisms in $\Br(\delta)$ which contain
\emph{no caps} and \emph{no intersecting propagating strands}.
 
\begin{prop} \label{two-small-triangular-decompositions}
\textup{(a)} The triple of $k$\+linear subcategories\/
$\sF=\Br^+(\delta)$, \,$\sG=\Br^=(\delta)$, and\/ $\sH=\Br(\delta)^\id$
forms a triangular decomposition of the $k$\+linear category\/
$\Br^{+=}(\delta)$ in the sense of
Sections~\ref{tensor-and-triangular-subsecn}\+-%
\ref{triangular-and-projectivity-subsecn}.
 In other words,
the map~\eqref{simplified-notation-three-subcategories-map}
$$
 \Br^+(\delta)\ot_{\Br(\delta)^\id}\Br^=(\delta)\lrarrow
 \Br^{+=}(\delta)
$$
is an isomorphism; or equivalently,
the map~\eqref{triangular-over-discrete-isomorphism}
is an isomorphism. \par
\textup{(b)} The triple of $k$\+linear subcategories\/
$\sF=\Br^=(\delta)$, \,$\sG=\Br^-(\delta)$, and\/ $\sH=\Br(\delta)^\id$
forms a triangular decomposition of the $k$\+linear category\/
$\Br^{=-}(\delta)$ in the sense of
Sections~\ref{tensor-and-triangular-subsecn}\+-%
\ref{triangular-and-projectivity-subsecn}.
 In other words,
the map~\eqref{simplified-notation-three-subcategories-map}
$$
 \Br^=(\delta)\ot_{\Br(\delta)^\id}\Br^-(\delta)\lrarrow
 \Br^{=-}(\delta)
$$
is an isomorphism; or equivalently,
the map~\eqref{triangular-over-discrete-isomorphism}
is an isomorphism. \qed
\end{prop}

 Comparing the results of
Propositions~\ref{triangular-decomposition-over-symmetric-group}
and~\ref{two-small-triangular-decompositions}, we come to
an isomorphism (in the same symbolic notation of
the formula~\eqref{simplified-notation-three-subcategories-map})
\begin{equation} \label{big-decomposition-into-small-pieces-eqn}
 \Br^+(\delta)\ot_{\Br(\delta)^\id}\Br^=(\delta)
 \ot_{\Br(\delta)^\id}\Br^-(\delta)\overset\simeq\lrarrow
 \Br(\delta).
\end{equation}
 Finally, from~\eqref{big-decomposition-into-small-pieces-eqn}
we arrive to the following corollary, spelling out two versions of
the triangular decomposition that are most relevant for us.

\begin{cor} \label{two-main-triangular-decompositions-for-brauer}
\textup{(a)} The triple of $k$\+linear subcategories\/
$\sF=\Br^+(\delta)$, \,$\sG=\Br^{=-}(\delta)$, and\/
$\sH=\Br(\delta)^\id$ forms a triangular decomposition of
the $k$\+linear category\/ $\Br(\delta)$
in the sense of Sections~\ref{tensor-and-triangular-subsecn}\+-%
\ref{triangular-and-projectivity-subsecn}.
 In other words,
the map~\eqref{simplified-notation-three-subcategories-map}
$$
 \Br^+(\delta)\ot_{\Br(\delta)^\id}\Br^{=-}(\delta)\lrarrow
 \Br(\delta)
$$
is an isomorphism; or equivalently,
the map~\eqref{triangular-over-discrete-isomorphism}
is an isomorphism. \par
\textup{(b)} The triple of $k$\+linear subcategories\/
$\sF=\Br^{+=}(\delta)$, \,$\sG=\Br^-(\delta)$, and\/
$\sH=\Br(\delta)^\id$ forms a triangular decomposition of
the $k$\+linear category\/ $\Br(\delta)$
in the sense of Sections~\ref{tensor-and-triangular-subsecn}\+-%
\ref{triangular-and-projectivity-subsecn}.
 In other words,
the map~\eqref{simplified-notation-three-subcategories-map}
$$
 \Br^{+=}(\delta)\ot_{\Br(\delta)^\id}\Br^-(\delta)\lrarrow
 \Br(\delta)
$$
is an isomorphism; or equivalently,
the map~\eqref{triangular-over-discrete-isomorphism}
is an isomorphism. \qed
\end{cor}

 It is worth mentioning that all the five categories
$\Br^{+=}(\delta)$, \,$\Br^{=-}(\delta)$, \,$\Br^=(\delta)$,
\,$\Br^+(\delta)$, and $\Br^-(\delta)$, \emph{viewed as abstract
$k$\+linear  categories} (rather than subcategories in $\Br(\delta)$),
do not depend on the parameter~$\delta$.
 It is only the whole category $\Br(\delta)$ that depends on~$\delta$.
 We only keep the symbol~$\delta$ in our notation for the five
subcategories in order to make the system of notation more transparent.

 Let us now discuss the local finiteness conditions from
Section~\ref{locally-finite-categories-subsecn} in application
to the Brauer diagram category and its subcategories defined above.
 First of all, we observe that the whole category $\Br(\delta)$ has
finite-dimensional\/ $\Hom$ spaces; but it is \emph{not} locally
finite in the sense of the definition in
Section~\ref{locally-finite-categories-subsecn}, because it does not
satisfy the interval finiteness condition (as
the diagram~\eqref{brauer-multiplication-diagram} illustrates).

\begin{lem} \label{brauer-lower-upper-finiteness}
\textup{(a)} The $k$\+linear categories\/ $\Br^+(\delta)$,
$\Br^{+=}(\delta)$, and\/ $\Br^=(\delta)$ are locally finite,
and moreover, they are lower finite. \par
\textup{(b)} The $k$\+linear categories\/ $\Br^-(\delta)$,
$\Br^{=-}(\delta)$, and\/ $\Br^=(\delta)$ are locally finite,
and moreover, they are upper finite.
\end{lem}

\begin{proof}
 This is trivial: for every given $n\ge0$, the condition
$|x_m|\le|x_n|$ only holds for a finite number of integers $m\ge0$.
\end{proof}

\begin{lem} \label{brauer-strict-local-finiteness}
\textup{(a)} The $k$\+linear categories\/ $\Br^+(\delta)$
and\/ $\Br^{+=}(\delta)$ are upper strictly locally finite. \par
\textup{(b)} The $k$\+linear categories\/ $\Br^-(\delta)$
and\/ $\Br^{=-}(\delta)$ are lower strictly locally finite.
\end{lem}

\begin{proof}
 Part~(b): for any object $y=x_n$ in $\Br^-(\delta)$
or $\Br^{=-}(\delta)$, setting $X_y=\{x_{n+1}\}$ satisfies
the condition in the definition of lower strict local finiteness
from Section~\ref{locally-finite-categories-subsecn}.
 Part~(b) is dual.
\end{proof}

 Finally, we come to the following theorem, which is out main
result in application to the Brauer diagram category.

\begin{thm}
\textup{(a)} Let\/ $\C^+=\C_{\Br^+(\delta)}$ be the coalgebra
corresponding to the locally finite $k$\+linear category
$\Br^+(\delta)$, as per the construction from
Section~\ref{locally-finite-categories-subsecn}.
 Let $K^+=R_{\Br^+(\delta)}$ and $R=R_{\Br(\delta)}$ be the nonunital
algebras corresponding to the $k$\+linear categories\/
$\Br^+(\delta)$ and\/ $\Br(\delta)$.
 Then there is a semialgebra\/ $\bS^+=\C^+\ot_{K^+}R$ over
the coalgebra\/ $\C^+$ such that the abelian category of right\/
$\bS^+$\+semimodules is equivalent to the category of right\/
$\Br(\delta)$\+modules, while the abelian category of left\/
$\bS^+$\+semicontramodules is equivalent to the category of left\/
$\Br(\delta)$\+modules,
$$
 \Simodr\bS^+\simeq\Modr\Br(\delta) \quad\text{and}\quad
 \bS^+\Sicntr\simeq\Br(\delta)\Modl.
$$
 The semialgebra\/ $\bS^+$ is an injective left\/ $\C^+$\+comodule.
\par
\textup{(b)} Let\/ $\C^{+=}=\C_{\Br^{+=}(\delta)}$ be the coalgebra
corresponding to the locally finite $k$\+linear category
$\Br^{+=}(\delta)$, as per the construction from
Section~\ref{locally-finite-categories-subsecn}.
 Let $K^{+=}=R_{\Br^{+=}(\delta)}$ and $R=R_{\Br(\delta)}$ be
the nonunital algebras corresponding to the $k$\+linear categories\/
$\Br^{+=}(\delta)$ and\/ $\Br(\delta)$.
 Then there is a semialgebra\/ $\bS^{+=}=\C^{+=}\ot_{K^{+=}}R$ over
the coalgebra\/ $\C^{+=}$ such that the abelian category of right\/
$\bS^{+=}$\+semimodules is equivalent to the category of right\/
$\Br(\delta)$\+modules, while the abelian category of left\/
$\bS^{+=}$\+semicontramodules is equivalent to the category of left\/
$\Br(\delta)$\+modules,
$$
 \Simodr\bS^{+=}\simeq\Modr\Br(\delta) \quad\text{and}\quad
 \bS^{+=}\Sicntr\simeq\Br(\delta)\Modl.
$$
 The semialgebra\/ $\bS^{+=}$ is an injective left\/
$\C^{+=}$\+comodule.
\end{thm}

 Let us emphasize that \emph{we do not know} whether the semialgebra
$\bS^+$ is an injective right $\C^+$\+comodule, or whether
the semialgebra $\bS^{+=}$ is an injective right $\C^{+=}$\+comodule.

\begin{proof}
 Part~(a): put $\sF=\Br^+(\delta)$, \,$\sG=\Br^{=-}(\delta)$,
and $\sE=\Br(\delta)$.
 Then, by
Corollary~\ref{two-main-triangular-decompositions-for-brauer}(a),
we have a triangular decomposition $\sF\ot_\sH\sG\simeq\sE$,
where $\sH=\Br(\delta)^\id$.
 By Lemma~\ref{brauer-lower-upper-finiteness}(a), the $k$\+linear
category $\sF=\Br^+(\delta)$ is lower finite.
 Therefore, Corollaries~\ref{category-semialgebra-right-semimodules}\+-%
\ref{category-semialgebra-lower-finite-case} are applicable and
provide the desired assertions.

 Part~(b): put $\sF=\Br^{+=}(\delta)$, \,$\sG=\Br^-(\delta)$,
and $\sE=\Br(\delta)$.
 Then, by
Corollary~\ref{two-main-triangular-decompositions-for-brauer}(b),
we have a triangular decomposition $\sF\ot_\sH\sG\simeq\sE$,
where $\sH=\Br(\delta)^\id$.
 By Lemma~\ref{brauer-lower-upper-finiteness}(a), the $k$\+linear
category $\sF=\Br^{+=}(\delta)$ is lower finite.
 Thus Corollaries~\ref{category-semialgebra-right-semimodules}\+-%
\ref{category-semialgebra-lower-finite-case} are applicable. 
\end{proof}

 Inverting the roles of the left and right sides in the discussion
above, one can consider the coalgebras $\C^-=\C_{\Br^-(\delta)}$
and $\C^{=-}=\C_{\Br^{=-}(\delta)}$, as well as the algebras
$K^-=R_{\Br^-(\delta)}$ and $K^{=-}=R_{\Br^{=-}(\delta)}$.
 Then, applying the left-right opposite version of our discussion
in Sections~\ref{flat-integrable-bimodules-secn}\+-%
\ref{description-of-semimod-semicontra-secn}, one can construct two
semialgebras $\bS^-=R\ot_{K^-}\C^-$ and $\bS^{=-}=R\ot_{K^{=-}}\C^{=-}$,
which are injective as right comodules over their respective
coalgebras $\C^-$ and $\C^{=-}$.
 One obtains equivalences of abelian categories
$$
 \bS^-\Simodl\simeq\bS^{=-}\Simodl\simeq\Br(\delta)\Modl
 \ \ \text{and}\ \
 \SicntrR\bS^-\simeq\SicntrR\bS^{=-}\simeq\Modr\Br(\delta),
$$
where $\SicntrR\bS$ denotes the category of right semicontramodules
over a semialgebra~$\bS$.

\subsection{Temperley--Lieb diagram category}
\label{temperley-lieb-subsecn}
 The following example was also suggested to me by Catharina Stroppel.
 We keep the notation from the beginning of
Section~\ref{brauer-subsecn}; so $k$~is a field and
$\delta\in k$ is an element.
 The \emph{Temperley--Lieb diagram category} $\TL(\delta)$
\,\cite{TL}, \cite[Section~4]{Kau}, \cite[Section~2]{BM},
\cite[Section~2]{RS}, \cite[Section~2]{DG}
is the following subcategory in $\Br(\delta)$.

 The objects of $\TL(\delta)$ are the nonnegative integers $n\ge0$
interpreted as linearly ordered finite sets $x_n=\{1,\dotsc,n\}$.
 Given two linearly ordered finite sets $x$ and~$y$,
the $k$\+vector subspace $\Hom_\TL(x,y)\subset\Hom_\Br(x,y)$ consists
of all the morphisms which contain \emph{no intersecting arcs}.

 More precisely, the vector subspace $\Hom_\TL(x,y)\subset\Hom_\Br(x,y)$
is spanned by all the basis vectors, i.~e., partitions~$a$ of
the set $x\sqcup y$ into a disjoint union of sets of cardinality~$2$,
satisfying the following condition.
 Let $n=|x|$ and $m=|y|$.
 Let us interpret the elements of~$x$ as the points $(0,1)$,~\dots,
$(0,n)$ in the real plane $\boR^2$ with coordinates $(u,v)$, and
the elements of~$y$ as the points $(1,1)$,~\dots, $(1,m)$ in
the same real plane.
 So we have $u=0$ for all the points of~$x$ and $u=1$ for all
the points of~$y$.
 Then the condition is that it must be possible to represent
the edges of the graph~$G_a$ by nonintersecting arcs connecting
our points in $\boR^2$, in such a way that all the arcs lie in
the region $0\le u\le 1$ in $\boR^2$ (i.~e., between the two lines
on which the points of $x$ and~$y$ are situated).
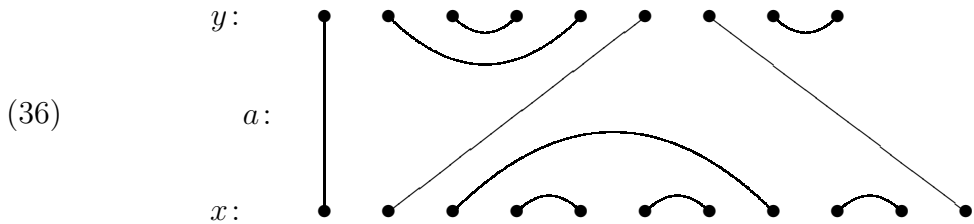
\begin{equation} \label{temperley-lieb-morphism-example}
\begin{gathered}
 \xymatrix{
  y\: & \bul \ar@{-}[dd] & \bul \ar@{-}@/^-1.5pc/[rrr] 
  & \bul \ar@{-}@/^-0.5pc/[r] & \bul & \bul & \bul \ar@{-}[lllldd]
  & \bul \ar@{-}[rrrrdd] &\bul \ar@{-}@/^-0.5pc/[r] & \bul \\
  \quad a\:\mkern-18mu \\
  x\: & \bul & \bul & \bul \ar@{-}@/^2.5pc/[rrrrr] 
  & \bul \ar@{-}@/^0.5pc/[r] & \bul
  & \bul \ar@{-}@/^0.5pc/[r] & \bul & \bul 
  & \bul \ar@{-}@/^0.5pc/[r] & \bul &\bul
 }
\end{gathered}
\end{equation}

 Now let us define two $k$\+linear subcategories in $\TL(\delta)$.
 Both the subcategories have the same set of objects as $\TL(\delta)$;
so the objects are the nonnegative integers.

 The subcategory $\TL^+(\delta)\subset\TL(\delta)$ consists of
all the morphisms in $\TL(\delta)$ which \emph{contain no caps}.
 Formally, we put
$$
 \TL^+(\delta)=\TL(\delta)\cap\Br^+(\delta)=
 \TL(\delta)\cap\Br^{+=}(\delta).
$$

 Similarly, the subcategory $\TL^-(\delta)\subset\TL(\delta)$
consists of all the morphisms in $\TL(\delta)$ which \emph{contain
no cups}.
 Formally,
$$
 \TL^-(\delta)=\TL(\delta)\cap\Br^-(\delta)=
 \TL(\delta)\cap\Br^{=-}(\delta).
$$

 It is worth mentioning that the categories $\TL^+(\delta)$ and
$\TL^-(\delta)$, \emph{viewed as abstract $k$\+linear categories}
(rather than subcategories in $\TL(\delta)$), do not depend on
the parameter~$\delta$.
 It is only the whole category $\TL(\delta)$ that depends on~$\delta$.
 We only keep the symbol~$\delta$ in our notation for the two
subcategories in order to make the system of notation more transparent.

\begin{prop} \label{triangular-decomposition-for-temperley-lieb}
 The triple of $k$\+linear subcategories\/ $\sF=\TL^+(\delta)$,
\,$\sG=\TL^-(\delta)$, and\/ $\sH=\TL(\delta)^\id$ forms a triangular
decomposition of the $k$\+linear category\/ $\TL(\delta)$
in the sense of Sections~\ref{tensor-and-triangular-subsecn}\+-%
\ref{triangular-and-projectivity-subsecn}.
 In other words,
the map~\eqref{simplified-notation-three-subcategories-map}
$$
 \TL^+(\delta)\ot_{\TL(\delta)^\id}\TL^-(\delta)\lrarrow
 \TL(\delta)
$$
is an isomorphism; or equivalently,
the map~\eqref{triangular-over-discrete-isomorphism}
is an isomorphism. 
\end{prop}

\begin{proof}
 We restrict ourselves to drawing a diagram for the factorization
of the morphism (basis element) $a\in\Hom_\TL(x,y)$ from
diagram~\eqref{temperley-lieb-morphism-example} into a composition
of two basis elements in the spaces of morphisms
$b\in\Hom_{\TL^-}(x,z)$ and $c\in\Hom_{\TL^+}(z,y)$:
$$
 \xymatrix{
  y\: & \bul \ar@{-}[dd] & \bul \ar@{-}@/^-1.5pc/[rrr] 
  & \bul \ar@{-}@/^-0.5pc/[r] & \bul & \bul & \bul \ar@{-}[lllldd]
  & \bul \ar@{-}[lllldd] &\bul \ar@{-}@/^-0.5pc/[r] & \bul \\
  \quad c\:\mkern-18mu \\
  z\: & \bul \ar@{-}[dd] & \bul \ar@{-}[dd]
  & \bul \ar@{-}[rrrrrrrrdd] \\
  \quad b\:\mkern-18mu \\
  x\: & \bul & \bul & \bul \ar@{-}@/^2.5pc/[rrrrr] 
  & \bul \ar@{-}@/^0.5pc/[r] & \bul
  & \bul \ar@{-}@/^0.5pc/[r] & \bul & \bul 
  & \bul \ar@{-}@/^0.5pc/[r] & \bul &\bul
 }
$$
 Here we have $a=cb$.
\end{proof}

 Let us discuss the local finiteness conditions from
Section~\ref{locally-finite-categories-subsecn} in application to
the Temperley--Lieb diagram category and its two subcategories
defined above.
 Similarly to the discussion in Section~\ref{brauer-subsecn},
we start with observing that the whole category $\TL(\delta)$ has
finite-dimensional\/ $\Hom$ spaces; but it is \emph{not} locally
finite in the sense of the definition in
Section~\ref{locally-finite-categories-subsecn}, because it does not
satisfy the interval finiteness condition.
 The two subcategories, however, are locally finite.

\begin{lem} \label{temperley-lieb-lower-upper-finiteness}
\textup{(a)} The $k$\+linear category\/ $\TL^+(\delta)$ is locally
finite, and moreover, it is lower finite. \par
\textup{(b)} The $k$\+linear category\/ $\TL^-(\delta)$ is locally
finite, and moreover, it is upper finite.
\end{lem}

\begin{proof}
 Follows from Lemma~\ref{brauer-lower-upper-finiteness}.
\end{proof}

\begin{lem} \label{temperley-lieb-strict-local-finiteness}
\textup{(a)} The $k$\+linear category\/ $\TL^+(\delta)$ is upper
strictly locally finite. \par
\textup{(b)} The $k$\+linear category\/ $\TL^-(\delta)$ is lower
strictly locally finite.
\end{lem}

\begin{proof}
 Similar to Lemma~\ref{brauer-strict-local-finiteness}.
\end{proof}

 So we come to the following theorem, which is out main result in
application to the Temperley--Lieb diagram category.

\begin{thm}
 Let\/ $\C^+=\C_{\TL^+(\delta)}$ be the coalgebra corresponding to
the locally finite $k$\+linear category $\TL^+(\delta)$, as per
the construction from Section~\ref{locally-finite-categories-subsecn}.
 Let $K^+=R_{\TL^+(\delta)}$ and $R=R_{\TL(\delta)}$ be the nonunital
algebras corresponding to the $k$\+linear categories\/
$\TL^+(\delta)$ and\/ $\TL(\delta)$.
 Then there is a semialgebra\/ $\bS^+=\C^+\ot_{K^+}R$ over
the coalgebra\/ $\C^+$ such that the abelian category of right\/
$\bS^+$\+semimodules is equivalent to the category of right\/
$\TL(\delta)$\+modules, while the abelian category of left\/
$\bS^+$\+semicontramodules is equivalent to the category of left\/
$\TL(\delta)$\+modules,
$$
 \Simodr\bS^+\simeq\Modr\TL(\delta) \quad\text{and}\quad
 \bS^+\Sicntr\simeq\TL(\delta)\Modl.
$$
 The semialgebra\/ $\bS^+$ is an injective left\/ $\C^+$\+comodule.
\end{thm}

 Let us emphasize that \emph{we do not know} whether the semialgebra
$\bS^+$ is an injective right $\C^+$\+comodule.

\begin{proof}
 Put $\sF=\TL^+(\delta)$, \,$\sG=\TL^-(\delta)$, and
$\sE=\TL(\delta)$.
 Then, by
Proposition~\ref{triangular-decomposition-for-temperley-lieb},
we have a triangular decomposition $\sF\ot_\sH\sG\simeq\sE$,
where $\sH=\TL(\delta)^\id$.
 By Lemma~\ref{temperley-lieb-lower-upper-finiteness}(a),
the $k$\+linear category $\sF=\TL^+(\delta)$ is lower finite.
 Therefore, Corollaries~\ref{category-semialgebra-right-semimodules}\+-%
\ref{category-semialgebra-lower-finite-case} are applicable and
provide the desired assertions.
\end{proof}

 Inverting the roles of the left and right sides in the discussion
above, one can consider the coalgebra $\C^-=\C_{\TL^-(\delta)}$ and
the algebra $K^-=R_{\TL^-(\delta)}$.
 Then, applying the left-right opposite version of our discussion
in Sections~\ref{flat-integrable-bimodules-secn}\+-%
\ref{description-of-semimod-semicontra-secn}, one can construct
a semialgebra $\bS^-=R\ot_{K^-}\C^-$, which is injective as a right
comodule over~$\C^-$.
 One obtains equivalences of abelian categories
$$
 \bS^-\Simodl\simeq\TL(\delta)\Modl
 \quad\text{and}\quad
 \SicntrR\bS^-\simeq\Modr\TL(\delta).
$$

\subsection{The Reedy category of a simplicial set}
 The following example was suggested to me by Jan \v St\!'ov\'\i\v cek.
 Without going into the details of the general definition of
a \emph{Reedy category}~\cite[Section~15.1]{Hir}, \cite[Section~2]{RV},
let us spell out a specific example relevant in our context.

 The very standard definition of the \emph{cosimplicial indexing
category} $\Delta$ is as follows~\cite[Definition~15.1.7]{Hir}.
 The objects of $\Delta$ are the nonempty finite linearly ordered
sets $[n]=\{0,\dotsc,n\}$, where $n\ge0$ are nonnegative integers.
 The morphisms $\sigma\:[n]\rarrow[m]$ in $\Delta$ are
the nonstrictly monotone maps $[n]\rarrow[m]$, i.~e., the functions
$\sigma\:\{0,\dotsc,n\}\rarrow\{0,\dotsc,m\}$ such that $i\le j$
implies $\sigma(i)\le\sigma(j)$.

 Two full subcategories $\Delta^+$ and $\Delta^-\subset\Delta$ are
defined by the following obvious rules.
 Both $\Delta^+$ and $\Delta^-$ have the same objects as~$\Delta$.
 The morphisms in $\Delta^+$ (known as \emph{face maps} and their
compositions) are those morphisms in $\Delta$ that correspond to
\emph{injective} maps of linearly ordered finite sets
$\sigma\:[n]\rarrow[m]$.
 The morphisms in $\Delta^-$ (known as \emph{degeneracy maps} and
their compositions) are those morphisms in $\Delta$ that correspond
to \emph{surjective} maps $\sigma\:[n]\rarrow[m]$.

 Clearly, any morphism~$\sigma$ in $\Delta$ can be uniquely
factorized as $\sigma=\tau\pi$, where $\tau$ is a morphism
in~$\Delta^+$ and $\pi$~is a morphism in~$\Delta^-$.
 In other words, the composition map
\begin{equation} \label{delta-triangular-decomposition}
 \coprod\nolimits_{[l]\in\Delta}\Mor_{\Delta^+}([l],[m])\times
 \Mor_{\Delta^-}([n],[l])\lrarrow\Mor_\Delta([n],[m])
\end{equation}
is bijective for all pairs of objects $[n]$, $[m]\in\Delta$.
 Here we denote by $\Mor_\sC(x,y)$ the set of morphisms $x\rarrow y$
in a nonadditive category~$\sC$.

 A \emph{simplicial set} $S$ is a contravariant functor
$S\:\Delta^\sop\lrarrow\Sets$, where $\Sets$ denotes the category
of sets.
 Let us denote by $S_n\in\Sets$ the set assigned to an object
$[n]\in\Delta$ by the functor~$S$.

 Given a simplicial set $S$, the \emph{category} $\Delta S$
\emph{of simplices of~$S$}~\cite[Example~15.1.14 and
Definition~15.1.16]{Hir} is defined as follows.
 The objects of $\Delta S$ are the simplices of~$S$; so the set of
objects of $\Delta S$ is the disjoint union $\coprod_{n\ge0}S_n$.
 Given two simplices $x\in S_n$ and $y\in S_m$, the set of
morphisms $\Mor_{\Delta S}(x,y)$ is, by the definition, bijective
to the set of all morphisms $\sigma\in\Hom_\Delta([n],[m])$ such
that $\sigma(y)=x$ in~$S$.
 The composition of morphisms in $\Delta S$ is induced by
the composition of morphisms in~$\Delta$.

 Two full subcategories $\Delta^+S$ and $\Delta^-S\subset\Delta S$
arise from the two full subcategories $\Delta^+$ and $\Delta^-
\subset\Delta$.
 By definition, both $\Delta^+S$ and $\Delta^-S$ have the same objects
as $\Delta S$.
 The morphisms in $\Delta^+S$ are those morphisms in $\Delta S$
that correspond to morphisms from $\Delta^+\subset\Delta$.
 The morphisms in $\Delta^-S$ are those morphisms in $\Delta S$
that correspond to morphisms from $\Delta^-\subset\Delta$.

 The unique factorization~\eqref{delta-triangular-decomposition}
of morphisms in $\Delta$ induces a similar unique factorization
of morphisms in~$\Delta S$.
 Any morphism~$\sigma$ in $\Delta S$ can be uniquely
factorized as $\sigma=\tau\pi$, where $\tau$ is a morphism
in~$\Delta^+S$ and $\pi$~is a morphism in~$\Delta^-S$.
 In other words, the composition map
\begin{equation} \label{category-of-simplices-triangular-decomposition}
 \coprod\nolimits_{z\in\Delta S}\Mor_{\Delta^+S}(z,y)\times
 \Mor_{\Delta^-S}(x,z)\lrarrow\Mor_{\Delta S}(x,y)
\end{equation}
is bijective for all pairs of objects $x$, $y\in\Delta S$.

 Now we want to pass from the categories enriched in sets to
the categories enriched in $k$\+vector spaces.
 The functor assigning to a set $C$ the $k$\+vector space
$k[C]$ freely spanned by $C$ is a monoidal functor $\Sets\rarrow
k\Vect$ from the monoidal category of sets (with the Cartesian
product) to the monoidal category of $k$\+vector spaces (with
the tensor product).
 Applying this monoidal functor to a category $\sC$, one obtains
a $k$\+linear category $k[\sC]$.
 Explicitly, the objects of $k[\sC]$ are the same as the objects
of~$\sC$.
 For any two objects $x$ and $y\in\sC$, the elements of the set
$\Hom_\sC(x,y)$ form a basis in the $k$\+vector space
$\Hom_{k[\sC]}(x,y)$.
 The composition of morphisms in $k[\sC]$ is induced by the composition
of morphisms in~$\sC$.

 Applying this construction to the category $\Delta S$ and its
subcategories $\Delta^+S$ and $\Delta^-S$, we obtain a $k$\+linear
category $k[\Delta S]$ with two $k$\+linear subcategories
$k[\Delta^+S]$ and $k[\Delta^-S]\subset k[\Delta S]$.
 These subcategories form the triangular decomposition which we
are interested in.

\begin{prop} \label{triangular-decomposition-for-simplicial}
 The triple of $k$\+linear subcategories\/ $\sF=k[\Delta^+S]$,
\,$\sG=k[\Delta^-S]$, and\/ $\sH=k[\Delta S]^\id$ forms a triangular
decomposition of the $k$\+linear category $k[\Delta S]$
in the sense of Sections~\ref{tensor-and-triangular-subsecn}\+-%
\ref{triangular-and-projectivity-subsecn}.
 In other words,
the map~\eqref{simplified-notation-three-subcategories-map}
$$
 k[\Delta^+S]\ot_{k[\Delta S]^\id}k[\Delta^-S]\lrarrow
 k[\Delta S]
$$
is an isomorphism; or equivalently,
the map~\eqref{triangular-over-discrete-isomorphism}
is an isomorphism. 
\end{prop}

\begin{proof}
 Apply the freely generated $k$\+vector space functor to
the bijection~\eqref{category-of-simplices-triangular-decomposition}.
\end{proof}

 Now it is the turn to discuss the local finiteness conditions
from Section~\ref{locally-finite-categories-subsecn} in application to
the category $k[\Delta S]$ and its two subcategories defined above.
 Similarly to Sections~\ref{brauer-subsecn}\+-%
\ref{temperley-lieb-subsecn}, the whole category $k[\Delta S]$ has
finite-dimensional\/ $\Hom$ spaces; but it is \emph{not} locally
finite in the sense of the definition in
Section~\ref{locally-finite-categories-subsecn}, because it does not
satisfy the interval finiteness condition.

\begin{lem} \label{morphism-order-in-freely-generated}
 Let\/ $\sC$ be a category and\/ $\sE=k[\sC]$ be the related
$k$\+linear category.
 Then one has $x\preceq y$ in\/ $\sE$ if and only if there exists
a morphism $x\rarrow y$ in\/~$\sC$.
\end{lem}

\begin{proof}
 The proof is straightforward.
\end{proof}

\begin{lem} \label{simplicial-lower-upper-finiteness}
\textup{(a)} The $k$\+linear category $k[\Delta^+S]$ is locally
finite, and moreover, it is lower finite. \par
\textup{(b)} The $k$\+linear category $k[\Delta^-S]$ is locally
finite, and moreover, it is upper finite.
\end{lem}

\begin{proof}
 Part~(a): let $y\in S_m$ be an object of~$\Delta S$.
 Notice that the datum of an element $x\in S_n$ together with
a morphism $\sigma\:x\rarrow y$ in $\Delta S$ is uniquely determined
by the related morphism $\sigma\:[n]\rarrow[m]$ in~$\Delta$: one
recovers the element $x\in S_n$ as $x=\sigma(y)$.
 Furthermore, for any fixed $m\ge0$, there is only a finite number
of morphisms $[n]\rarrow[m]$ in $\Delta^+$ (as one necessarily
has $n\le m$).
 Consequently, for any fixed $y\in S_m$, the set of all objects~$x$
in $\Delta S$ for which there exists a morphism $x\rarrow y$
in $\Delta^+S$ is finite.
 By Lemma~\ref{morphism-order-in-freely-generated}, this set coincides
with the set of all objects~$x$ in $k[\Delta^+S]$ such that
$x\preceq y$ in $k[\Delta^+S]$.

 Part~(b): let $x\in S_n$ be an object of~$\Delta S$.
 Then the datum of an element $y\in S_m$ together with
a morphism $\sigma\:x\rarrow y$ in $\Delta S$ is \emph{not}
uniquely determined by the related morphism
$\sigma\:[n]\rarrow[m]$ in~$\Delta$, generally speaking (one simplex
can be a face of arbitrarily many other simplices of given dimension).

 However, let us restrict ourselves to morphisms in~$\Delta^-S$.
 Then the datum of an element $y\in S_m$ together with a morphism
$\pi\:x\rarrow y$ in $\Delta^-S$ is still uniquely determined
by the related morphism $\pi\:[n]\rarrow[m]$ in~$\Delta^-$.
 Indeed, $\pi$~is a surjective nonstrictly order-preserving map
of finite sets $\{0,\dotsc,n\}\rarrow\{0,\dotsc,m\}$.
 Therefore, $\pi$~admits an order-preserving section
$\tau\:\{0,\dotsc,m\}\rarrow\{0,\dotsc,n\}$.
 So there exists a morphism $\tau\:[m]\rarrow[n]$ in~$\Delta$
such that $\pi\tau=\id_{[m]}$.
 Now we have $x=\pi(y)$ in $S_n$, and we can recover the element
$y\in S_m$ as $y=\tau(x)$.

 The rest of the argument is similar to part~(a).
 For any fixed $n\ge0$, there is only a finite number of
morphisms $[n]\rarrow[m]$ in $\Delta^-$ (as one necessarily has
$m\le n$).
 Consequently, for any fixed $x\in S_n$, the set of all objects~$y$
in $\Delta S$ for which there exists a morphism $x\rarrow y$
in $\Delta^-S$ is finite.
 By Lemma~\ref{morphism-order-in-freely-generated}, this set coincides
with the set of all objects~$y$ in $k[\Delta^-S]$ such that
$x\preceq y$ in $k[\Delta^-S]$.
\end{proof}

\begin{lem} \label{simplicial-strict-local-finiteness}
\textup{(a)} Let $S$ be a simplicial set such that the set $S_n$ is
finite for every $n\ge0$.
 Then the $k$\+linear category $k[\Delta^+S]$ is upper strictly
locally finite. \par
\textup{(b)} For any simplicial set $S$, the $k$\+linear category
$k[\Delta^-S]$ is lower strictly locally finite.
\end{lem}

\begin{proof}
 Part~(a): let $x\in S_n$ and $w\in S_m$ be two objects of
$\Delta^+S$.
 Then one has $x\prec w$ in $k[\Delta^+S]$ if and only if $n<m$
and there exists a morphism $x\rarrow w$ in $\Delta^+S$.
 If this is the case, then such morphism $x\rarrow w$ factorizes
through some object $y\in S_{n+1}$ for which $x\prec y$
in~$k[\Delta^+S]$.
 By assumption, there is only a finite set of objects~$y$
that can occur here as $x$~is fixed and $w$~varies.

 Part~(b): let $z\in S_n$ and $y\in S_m$ be two objects of
$\Delta^-S$.
 Then one has $z\prec y$ in $k[\Delta^-S]$ if and only if $n>m$
and there exists a morphism $z\rarrow y$ in $\Delta^-S$.
 If this is the case, then such morphism $z\rarrow y$ factorizes through
some object $x\in S_{m+1}$ for which $x\prec y$ in $k[\Delta^-S]$.
 In other words, there exists a morphism $x\rarrow y$
in~$\Delta^-S$.
 There is only a finite set of $m+1$ such objects~$x$,
indexed by the morphisms $[m+1]\rarrow[m]$ in~$\Delta^-$,
that can occur here as $y$~is fixed and $z$~varies.
\end{proof}

 We come to the following theorem, which is out main result in
application to category of simplices of a simplicial set.

\begin{thm}
 Let $S$ be a simplicial set.  In this context: \par
\textup{(a)} Let\/ $\C^+=\C_{k[\Delta^+S]}$ be the coalgebra
corresponding to the locally finite $k$\+linear category
$k[\Delta^+S]$, as per the construction from
Section~\ref{locally-finite-categories-subsecn}.
 Let $K^+=R_{k[\Delta^+S]}$ and $R=R_{k[\Delta S]}$ be the nonunital
algebras corresponding to the $k$\+linear categories\/
$k[\Delta^+S]$ and $k[\Delta S]$.
 Then there is a semialgebra\/ $\bS^+=\C^+\ot_{K^+}R$ over
the coalgebra\/ $\C^+$ such that the abelian category of right\/
$\bS^+$\+semimodules is equivalent to the category of right
$k[\Delta S]$\+modules, while the abelian category of left\/
$\bS^+$\+semicontramodules is equivalent to the category of left
$k[\Delta S]$\+modules,
$$
 \Simodr\bS^+\simeq\Modr k[\Delta S] \quad\text{and}\quad
 \bS^+\Sicntr\simeq k[\Delta S]\Modl.
$$
 The semialgebra\/ $\bS^+$ is an injective left\/ $\C^+$\+comodule.
\par
\textup{(b)} Let\/ $\C^-=\C_{k[\Delta^-S]}$ be the coalgebra
corresponding to the locally finite $k$\+linear category
$k[\Delta^-S]$, as per the construction from
Section~\ref{locally-finite-categories-subsecn}.
 Let $K^-=R_{k[\Delta^-S]}$ and $R=R_{k[\Delta S]}$ be the nonunital
algebras corresponding to the $k$\+linear categories\/
$k[\Delta^-S]$ and $k[\Delta S]$.
 Then there is a semialgebra\/ $\bS^-=R\ot_{K^-}\C^-$ over
the coalgebra\/ $\C^-$ such that the abelian category of left\/
$\bS^-$\+semimodules is equivalent to the category of left
$k[\Delta S]$\+modules, while the abelian category of right\/
$\bS^-$\+semicontramodules is equivalent to the category of right
$k[\Delta S]$\+modules,
$$
 \bS^-\Simodl\simeq k[\Delta S]\Modl \quad\text{and}\quad
 \SicntrR\bS^-\simeq\Modr k[\Delta S].
$$
 The semialgebra\/ $\bS^-$ is an injective right\/ $\C^-$\+comodule.
\end{thm}

 Let us emphasize that \emph{we do not know} whether the semialgebra
$\bS^+$ is an injective right $\C^+$\+comodule, or whether
the semialgebra $\bS^-$ is an injective left $\C^-$\+comodule.

\begin{proof}
 Part~(a): put $\sF=k[\Delta^+S]$, \,$\sG=k[\Delta^-S]$,
and $\sE=k[\Delta S]$.
 Then, by
Proposition~\ref{triangular-decomposition-for-simplicial},
we have a triangular decomposition $\sF\ot_\sH\sG\simeq\sE$,
where $\sH=k[\Delta S]^\id$.
 By Lemma~\ref{simplicial-lower-upper-finiteness}(a), the $k$\+linear
category $\sF=k[\Delta^+S]$ is lower finite.
 Therefore, Corollaries~\ref{category-semialgebra-right-semimodules}\+-%
\ref{category-semialgebra-lower-finite-case} are applicable and
provide the desired assertions.

 Part~(b): $\sF=k[\Delta^-S]^\sop$, \,$\sG=k[\Delta^+S]^\sop$,
and $\sE=k[\Delta S]^\sop$.
 Then, by
Proposition~\ref{triangular-decomposition-for-simplicial},
we have a triangular decomposition $\sF\ot_\sH\sG\simeq\sE$,
where $\sH=k[\Delta S]^{\sop,\id}$.
 By Lemma~\ref{simplicial-lower-upper-finiteness}(b), the $k$\+linear
category $k[\Delta^-S]$ is upper finite; so the opposite $k$\+linear
category $\sF=k[\Delta^-S]^\sop$ is lower finite.
 Therefore, Corollaries~\ref{category-semialgebra-right-semimodules}\+-%
\ref{category-semialgebra-lower-finite-case} are applicable and
provide the desired assertions, up to a passage to the left-right
opposite coalgebra and semialgebra.
\end{proof}

\end{document}